\documentclass[12pt]{amsart}
\usepackage{amssymb}
\usepackage{xypic}
\usepackage{hhline}
\usepackage{hyperref}

\newtheorem{theorem}{Theorem}[section]

\newtheorem{rem}[theorem]{Remark}
\newtheorem{hypothesis}[theorem]{Hypothesis}
\newtheorem{notation}[theorem]{Notation}

\newtheorem{proposition}[theorem]{Proposition}
\newtheorem{corollary}[theorem]{Corollary}

\theoremstyle{remark}

\numberwithin{equation}{section}

\begin{document}



\newcommand{\OP}[1]{\operatorname{#1}}
\newcommand{\GO}{\OP{GO}}
\newcommand{\GT}{\OP{GT}}
\newcommand{\AI}{\OP{AI}}
\newcommand{\Asai}{\OP{Asai}}
\newcommand{\leftexp}[2]{{\vphantom{#2}}^{#1}{#2}}
\newcommand{\cusp}{\OP{cusp}}
\newcommand{\Res}{\OP{Res}}
\newcommand{\tr}{\OP{tr}}
\newcommand{\PGSp}{\OP{PGSp}}
\newcommand{\Sp}{\OP{Sp}}
\newcommand{\Nrd}{\OP{Nrd}}
\newcommand{\JL}{\OP{JL}}
\newcommand{\BC}{\OP{BC}}
\newcommand{\Hom}{\OP{Hom}}
\newcommand{\Sym}{\OP{Sym}}
\newcommand{\GSp}{\OP{GSp}}
\newcommand{\GL}{\OP{GL}}
\newcommand{\GSO}{\OP{GSO}}
\newcommand{\sen}{\OP{sen}}
\newcommand{\art}{\OP{art}}
\newcommand{\bdd}{\OP{bdd}}
\newcommand{\sph}{\OP{sph}}
\newcommand{\Aut}{\OP{Aut}}
\newcommand{\disc}{\OP{disc}}
\newcommand{\temp}{\OP{temp}}
\newcommand{\CM}{\OP{CM}}
\newcommand{\cris}{\OP{cris}}
\newcommand{\End}{\OP{End}}
\newcommand{\barQ}{\OP{\overline{\mathbf{Q}}}}
\newcommand{\barQp}{\OP{\overline{\mathbf{Q}}_{\it p}}}
\newcommand{\Gal}{\OP{Gal}}
\newcommand{\WD}{\OP{WD}}
\newcommand{\Ind}{\OP{Ind}}
\newcommand{\St}{\OP{St}}
\newcommand{\rec}{\OP{rec}}
\newcommand{\SL}{\OP{SL}}
\newcommand{\Frob}{\OP{Frob}}
\newcommand{\Id}{\OP{Id}}
\newcommand{\GSpin}{\OP{GSpin}}
\newcommand{\Norm}{\OP{Norm}}

\title[Galois representations ...]{Galois representations attached to automorphic forms on $GL_2$ over $CM$ fields}

\author{Chung Pang Mok}

\address{Hamilton Hall, McMaster University, Hamilton, Ontario, L8S 4K1}

\email{cpmok@math.mcmaster.ca}

\begin{abstract}
In this paper we generalize the work of Harris-Soudry-Taylor and construct the compatible systems of two-dimensional Galois representations attached to cuspidal automorphic representations of cohomological type on $\GL_2$ over a CM field with a suitable condition on their central characters. We also prove a local-global compatibility statement, up to semisimplification.
\end{abstract}

\maketitle

\section{Introduction}

In the works \cite{HST} of Harris-Soudry-Taylor and \cite{T2} of Taylor, the authors construct the compatible system of two-dimensional $p$-adic Galois representations attached to a cuspidal automorphic representation of cohomological type on $\GL_2$ over an imaginary quadratic field, whose central character satisfies a suitable invariance condition. The improvement \cite{BH} of Berger-Harcos  shows that the the Galois representations constructed by \cite{HST,T2} satisfies the correct equality of Frobenius and Hecke polynomials, outside an explicit finite set of primes. The aim of the present work is to generalize the results of \cite{T2,BH} to a general $\CM$ field, and to prove a local-global compatibility statement, up to semi-simplification, at primes not dividing $p$.

 To state the main result of this paper, we set up some notations. Let $E$ be a $\CM$ field, with $F$ its maximal totally real subfield. Denote by $\tau$ the Galois conjugation of $E$ over $F$. Suppose that $\pi$ is a cuspidal automorphic representation on $\GL_2(\mathbf{A}_E)$ (where $\mathbf{A}_E$ is the ring of adeles of $E$). We assume that $\pi$ is of cohomological type. Denote by $\omega_{\pi}$ the central character of $\pi$. In this paper, we are interested in such cohomological $\pi$, whose central character $\omega_{\pi}$ satisfies the following confition:

\bigskip

\noindent (\textbf{Char}) There exists an algebraic idele class character $\widetilde{\omega}$ of $\mathbf{A}_F^{\times}$ such that
\[
\omega_{\pi} = \widetilde{\omega} \circ \Norm_{E/F}
\]
(here $\Norm_{E/F}$ is the norm map from $\mathbf{A}_E^{\times}$ to $\mathbf{A}_F^{\times}$), and $\widetilde{\omega} =\bigotimes_{v} \widetilde{\omega}_v$ satisfies in addition the following: the sign $\widetilde{\omega}_v(-1)$ takes the same value for all archimedean places of $F$.

\bigskip

Condition (\textbf{Char}) in particular implies that $\omega_{\pi}$ is invariant under $\tau$, i.e. that $\omega_{\pi} \circ \tau = \omega_{\pi}$. In the case where $E$ is an imaginary quadratic field, i.e. $F=\mathbf{Q}$ as considered in \cite{HST,T2}, condition (\textbf{Char}) is equivalent to requiring that $\omega_{\tau}$ is invariant under $\tau$. But for $F \neq \mathbf{Q}$ it is stronger than the invariance.

Now fix an embedding $ \barQ \hookrightarrow \mathbf{C}$. Similarly, for each prime $p$, fix an embedding $ \barQ \hookrightarrow \barQp$. Fix an isomorphism $\iota_p:  \barQp \cong \mathbf{C}$ compatible with these two embeddings of $\barQ$. 

As usual, we denote by $G_L$ the absolute Galois group of any field $L$.

\begin{theorem}
Suppose that the central character of $\pi$ satisfies condition (\textbf{Char}). Then for each prime $p$, there exists a continuous irreducible $p$-adic Galois representation:
\[
\rho_p : G_E \rightarrow \GL_2(\barQp)
\]
such that, for each finite prime $w$ of $E$ not dividing $p$, we have the local-global compatibility statement, up to semi-simplification:
\begin{eqnarray*}
\iota_p  \WD(\rho_p|_{G_{E_w}})^{ss} \cong    \mathcal{L}_{E_w}(\pi_w \otimes | \det |_w^{-1/2})^{ss}.
\end{eqnarray*}
Here $ \WD$ is the Weil-Deligne functor, and $\mathcal{L}_{E_w}$ is the local Langlands correspondence for $GL_2(E_w)$. Furthermore, if $\pi_w$ is not of the form $\St_{E_w} \otimes \chi$, where $\St_{E_w}$ is the Steinberg representation of $\GL_2(E_w)$, and $\chi$ is a character of $E_w^{\times}$, then we have the full-local-global compatibility (as usual up to Frobenius semi-simplification):
\[
\iota_p  \WD(\rho_p|_{G_{E_w}})^{F-ss} \cong  \mathcal{L}_{E_w}(\pi_w \otimes | \det |_w^{-1/2}).
\]
\end{theorem}

According to general conjectures, since $\pi$ is cohomological, the Galois representation $\rho_p$ should be {\it geometric} in the sense of Fontaine-Mazur, i.e. $\rho_p$ should be deRham (in the sense of Fontaine) at places of $E$ above $p$. We are only able to show that $\rho_p$ is Hodge-Tate in general. More precisely we have:

\begin{theorem}
For each finite prime $w$ of $E$ above $p$, the representation $\rho_p$ is Hodge-Tate at $w$, with the correct Hodge-Tate weights (determined by the archimedean components of $\pi$; see theorem 5.17 for the precise statement).

Furthermore, suppose that $w$ is a place of $E$ above $p$, inert over $F$, and such that $\pi_w$ is spherical, with distinct Satake parameters $\alpha_w \neq \beta_w$. Then $\rho_p$ is crystalline at $w$.

Suppose that $w$ is a prime above $p$ that splits over $F$, with conjugate prime $w^{\tau}$, such that $\pi_w$ and $\pi_{w^{\tau}}$ are spherical, with Satake parameters $\alpha_w,\beta_w$ for $\pi_w$, and $\alpha_{w^{\tau}},\beta_{w^{\tau}}$ for $\pi_{w^{\tau}}$ respectively. Suppose that the elements $\{\alpha_w,\beta_w,\alpha_{w^{\tau}} ,\beta_{w^{\tau}}\}$ are all distinct. Then $\rho_p$ is crystalline at w (and also at $w^{\tau}$).
\end{theorem}

\begin{rem}
\end{rem} 
\noindent The Galois representations $\rho_p$ is constructed using $p$-adic limit process from Galois representations which are geometric. However, taking $p$-adic limit process does not in general preserve the deRham property, which is the reason that we are not able to show that $\rho_p$ is deRham at primes above $p$ in general. For exactly the same reason, the existence of $\rho_p$ alone does not imply the Ramanujan conjecture for $\pi$. 

We remark that the crystalline assertions in theorem 1.2, under the stated conditions, was proved by A. Jorza [J] in the case where $E$ is imaginary quadratic.

\bigskip

\subsection*{About the proof}

As in the case of \cite{T2,BH}, the fundamental difficulty in the construction of the Galois representation is that the group $\Res_{E/\mathbf{Q}} \GL_{2/E}$ (Weil restriction of scalars) does not admit a Shimura variety. The basic idea of Harris-Soudry-Taylor \cite{HST} is that when the cuspidal automorphic representation $\pi$ on $\GL_2(\mathbf{A}_E)$ satisfies condition (\textbf{Char}), the representation $\pi$ admits a lifting $\Pi$ as a cuspidal automorphic representation on the group $\GSp_4(\mathbf{A}_F)$. One can choose the lifting $\Pi$ so that it is represented as a holomorphic Siegel-Hilbert modular form. The argument of \cite{T2,BH}, involving consideration of quadratic twists of $\pi$, allows one to extract the two dimensional Galois representation associated to $\pi$ from the four dimensional Galois representation associated to $\Pi$, if the latter can be constructed.

One of the difficulty in the construction of the Galois representation associated to $\Pi$ is that it is not of cohomological type. More precisely, its archimedean components belong to holomorphic limit of discrete series. To circumvent this difficulty, we draw on the results of \cite{MT} (which is a generalization of the results of Kisin-Lai \cite{KL} and A. Jorza \cite{J} to the Siegel-Hilbert case). Suppose first that the cuspidal automorphic representation $\Pi$ admits Iwahori fixed vectors at primes of $F$ above $p$. Then the main result of \cite{MT} gives a one-parameter family of $p$-adic deformation of $\Pi$, with the property that there is a Zariski dense set of so-called classical points in the family, corresponding to cuspidal automorphic representations of $\GSp_4(\mathbf{A}_F)$ of cohomologcal type (actually having archimedean components in holomorphic discrete series), and these points accumulate to $\Pi$. A standard argument using the theory of pseudo-character allows one to construct the Galois representation attached to $\Pi$, if one can construct the four dimensional Galois representations attached to a cuspidal automorphic representations on $\GSp_4(\mathbf{A}_F)$ (say $\Pi^{\prime}$) that is of cohomological type.

When $F=\mathbf{Q}$, the Galois representation associated to such a $\Pi^{\prime}$ can be constructed from the $p$-adic etale cohomology of a suitable Siegel three-fold (work of Laumon \cite{L} and Weissauer \cite{W}). However, when $F \neq \mathbf{Q}$, the Galois representation obtained by using the $p$-adic etale cohomology of a similar Siegel-Hilbert modular variety would be of dimension $4^{[F:\mathbf{Q}]}$ (same phenomenon as in the construction of Galois representation associated to Hilbert modular forms), and is no longer appropriate. Instead, if one has a lifting of $\Pi^{\prime}$ to an automorphic representation $\widetilde{\Pi}^{\prime}$ on $\GL_4(\mathbf{A}_F)$, then the latter would be of cohomological and self-dual type, and the arguments of Sorensen \cite{So} and Chenevier-Harris \cite{ChH} allow one to construct the Galois representation associated to $\widetilde{\Pi}^{\prime}$, using Galois representations that ultimately come from etale cohomology of Shimura varieties of unitary groups \cite{Sh}. 

If $\Pi^{\prime}$ is globally generic, then the lifting to $\GL_4(\mathbf{A}_F)$ follows from the work of Jacquet-Piateski-Shapiro-Shalika (unpublished, see for example \cite{GRS}) and of Asgari-Shahidi \cite{AS}. However, in our case, the members corresponding to the classical points in the $p$-adic family are not generic at the archimedean primes, in particular not globally generic, and their result does not apply. Instead, we use the result of Arthur \cite{A1,A2} on endoscopic classification of the discrete automorphic spectrum of $\GSp_4(\mathbf{A}_F)$, which in our case gives the required lifting to $\GL_4(\mathbf{A}_F)$ (alternatively we can use the result of Arthur \cite{A2} on the generic packet conjecture, but in any case this is proved using the main results of \cite{A2} on endoscopic classification).

The above strategy works in the case when $\Pi$ has Iwahori-fixed vector at primes of $F$ above $p$. To treat the general case, we use the patching lemma of \cite{So}, together with the result that we can do base change for automorphic representation on $\GSp_4$ over solvable extensions. This latter result again relies on the works of Arthur \cite{A1,A2}, combined with the solvable base change theorems of Arthur-Clozel \cite{AC} for general linear groups, and the result of Shahidi \cite{S2} on the non-vanishing of twisted exterior square $L$-function at $s=1$. 

The method of using $p$-adic deformation is the main difference with the method of \cite{T1}, where the method of considering $p$-power congruences between $\Pi$ and forms of cohomological type is used instead. Even though the method of using $p$-adic deformation is less elementary, it has an advantage, namely one can prove a local-global compatibility statement for primes of $F$ not dividing $p$, up to semi-simplification, for the Galois representation associated to $\Pi$ (which is not of cohomological type), using the corresponding local-global compatibility statement for the Galois representations associated to cuspidal automorphic representations on $\GSp_4(\mathbf{A}_F)$ that are of cohomological type. The argument is a generalization of Chenevier \cite{Ch}, and we rely on the work of Gan-Takeda \cite{GT1,GT2}, which gives a fairly explicit description of the local Langlands correspondence for $\GSp_4$, and allows us to control the semi-simple part of the local Langlands parameter at primes not dividing $p$ in a $p$-adic family via Bernstein components. Here we are also using the results of Chan-Gan \cite{CG} to ensure that the local Langlands correspondence for $\GSp_4$ constructed by Gan-Takeda is compatible with the trace formula lifting from $\GSp_4$ to $\GL_4$ in the framework of \cite{A1,A2}. We remark that it seems difficult to obtain the full local-global compatibility statement (i.e. including the monodromy part) using $p$-adic deformation technique.

Once we have a local-global compatibility (up to semi-simplification) statement for the Galois representation associated to $\Pi$, we can use a generalization of the argument of \cite{BH}, again involving quadratic twists, to obtain the local-global compatibility statement, up to semi-simplification, for the Galois representation associated to the cuspidal automorphic representation on $\GL_2(\mathbf{A}_E)$ that we started with. Since base change arguments is used in the proof of local-global compatibility (up to semi-simplification), it is crucial that we are able to do the construction for a general CM field. Finally, the part of theorem 1.2 on the crystalline property can be proved as in the work of A. Jorza \cite{J}, using the results of K. Nakamura \cite{N} and F.C. Tan \cite{Tan} generalizing results of Kisin to arbitraty $p$-adic local fields on analytic continuation of crystalline periods. \footnote{The first draft of this paper was written in 2011; in the intervening years there has been significant progress in the construction of Galois representations attached to cohomological automorphic forms on general linear groups over totally real and $\CM$ fields. For the most up to date results the reader should refer to the works of \cite{HLTT} and \cite{Scho}.}

\bigskip

\noindent {\it On the lifting from $\GL_{2/E}$ to $\GSp_{4/F}$}
\bigskip

To be precise, the relevant group is the algebraic group $H$ over $F$, whose group of $F$-points is given by
\[
H(F) = \GL_2(E) \times F^{\times} \big/ \{    (z \Id_2, \Norm_{E/F} z^{-1} ), z \in E^{\times}    \}.
\]
We have $H \cong \GSO(3,1)_{/F}$ (the identity component of $\GO(3,1)_{/F}$). As before, $\pi$ is a cuspidal automorphic representation on $\GL_2(\mathbf{A}_E)$, whose central character $\omega_{\pi}$ satisfies condition (\textbf{Char}). Thus $\omega_{\pi}$ factors through the norm map $\Norm_{E/F}$, and there are exactly two choices of Hecke characters $\widetilde{\omega}$ of $\mathbf{A}_F^{\times}$ through which $\omega_{\pi}$ factors, corresponding to extensions $\widetilde{\pi}$ of $\pi$ to $H(\mathbf{A}_F) \cong\GSO(3,1)(\mathbf{A}_F)$. As in \cite{HST}, one can extend $\widetilde{\pi}$ to $\GO(3,1)(\mathbf{A}_F)$ by considering auxiliary data (choice of a suitable sign for each place of $F$). One can then consider the orthogonal-symplectic dual reductive pair $(\GO(3,1)_{/F}, \GSp_{4/F})$, and hence construct the theta lifting of $\widetilde{\pi}$ (more precisely an extension of $\widetilde{\pi}$ to $\GO(3,1)(\mathbf{A}_F)$) to $\GSp_4(\mathbf{A}_F)$. It is proved by S. Takeda \cite{Ta} (building on the work of Roberts \cite{R}) that the theta lifting is non-zero (when \cite{HST} was written, it was only known that given $\pi$, there is a sufficiently ample infinite set of quadratic Hecke characters $\eta$ of $\mathbf{A}_E^{\times}$ such that the theta lifting construction applied to $\pi \otimes \eta$ is non-zero).

On the other hand $H$ is one of the elliptic $\epsilon$-twisted endoscopic group of $\GSp_{4/F}$ (here $\epsilon$ is the quadratic idele class character of $\mathbf{A}_F^{\times}$ corresponding to the quadratic extension $E/F$, hence a character of $\GSp_4(\mathbf{A}_F)$ via the similitude factor). In the work \cite{C}, the lifting of $\widetilde{\pi}$ to $\GSp_4(\mathbf{A}_F)$ is constructed as a trace formula lifting in the framework of $\epsilon$-twisted endoscopy (at least under some local conditions on $\pi$). In this paper, we will also deduce this lifting as a corollary of Arthur's work on endoscopic classification of automorphic representation on $\GSp_4(\mathbf{A}_F)$.

\bigskip
\noindent {\it On the use of Arthur's results}

\bigskip
In this paper we crucially rely on Arthur's works on endoscopic classification of automorphic representations for the group $\GSp_4$ \cite{A1,A2}. At the time of writing, the results in {\it loc. cit.} are still conditional on the stabilization of the twisted trace formula; however significant progress in this direction has been made by Waldspurger and others. We also remark that in \cite{A2} the results are worked out for symplectic and orthogonal groups, though it is sketched in \cite{A1} how the formalism and results in the setting of \cite{A2} can be extended to cover the case of $\GSp_4$. Thus the results of this paper will be conditional on the results of Arthur in \cite{A1,A2}. 

\subsection*{Organization of the paper}

In section 2, we summarize the results from Arthur's endoscopic classification of the discrete automorphic spectrum \cite{A1,A2} that we need in the sequel. We also record a consequence of Arthur's results on base change of automorphic representation on $\GSp_4$ (which is already known to several authors earlier \cite{S2,Sou2}).

In section 3 we combine Arthur's work together with the result of Sorensen \cite{So}, and Chenevier-Harris \cite{ChH}, to obtain results on existence of Galois representation associated to cuspidal automorphic representation on $\GSp_4(\mathbf{A}_F)$ of cohomological type, together with a local-global compatibility statement. The fact that Arthur's results imply such statements on Galois representation is well-known, but we record the statements here for reference.

Section 4 is the main technical core of the paper. The goal is to construct the Galois representation associated to cuspidal automorphic representation on $\GSp_4(\mathbf{A}_F)$, whose archimdean components are in the holomorphic limit of discrete series. This uses the technique of $p$-adic deformation from \cite{MT}. We also prove a local-global compatibility statement up to semi-simplification for the Galois representation constructed, using the idea of Chenevier \cite{C}, together with the results of Gan-Takeda \cite{GT1,GT2} and Chan-Gan \cite{CG}.

In section 5, we apply the results in section 4 to the construction of the Galois representation associated to a cuspidal automorphic representation on $\GL_2(\mathbf{A}_E)$ of cohomological type, whose central character satisfies condition (\textbf{Char}), by the method of \cite{T2,BH}. We show that the lifting from $\GL_2(\mathbf{A}_E)$ to $\GSp_4(\mathbf{A}_F)$ follows from Arthur's results. The local-global compatibility statement, up to semi-simplification, follows from the results of section 4, together with a generalization of the argument of \cite{BH}. With this comes the proof of theorem 1.1 and 1.2.

\section*{acknowledgement}
The author would like to thank Wee Teck Gan for his help with $\GSp_4$, without which this paper cannot come into existence. He would also like to thank Ping Shun Chan, whose lecture given at the Chinese University of Hong Kong made the author realized that the work of Harris-Soudry-Taylor can be generalized to the $\CM$ case. He is also very grateful to Rapha\"el Beuzart-Plessis and Olivier Ta\"ibi, who pointed out that the correct condition on central character should be given by (\textbf{Char}), which was overlooked in the first version of the paper. Finally he would like to thank the referee for some useful comments.

\begin{notation}
\end{notation}

In general if $F$ is a number field, and $v$ a prime of $F$ (both finite and archimedean), we denote by $F_v$ the completion of $F$ at $v$. We denote by $\art_v: F_v^{\times} \rightarrow W_{F_v}^{ab}$ the local reciprocity isomorphism at $v$, with $W_{F_v}$ being the Weil group of $F_v$. We generally denote by $\Frob_v$ a geometric Frobenius element at $v$, and $\art_v$ is normalized so that $\Frob_v$ corresponds to a uniformizer of $F_v$. For $w \in W_{F_v}$, we denote by $|w|_v$ the absolute value of $w$ induced by the normalized absolute value on $F_v$ under $\art_v$. If $F_v = \mathbf{R}$, then the isomorphism $W_{\mathbf{R}}^{ab} \cong \mathbf{R}^{\times}$ is induced by sending $z \in \mathbf{C}^{\times} \subset W_{\mathbf{R}} = \mathbf{C}^{\times} \cup \mathbf{C}^{\times} j$ to $|z|_{\mathbf{C}}:=|z|^2$ (thus $z$ is the usual absolute value of a complex number), and $j \mapsto -1$.

If $v$ is finite and $w \in W_{F_v}$, then valuation of $w$ is defined to be the integer $r$ such that $|w|_v = Nv^{-r}$, here $Nv$ is the norm of $v$. 

We denote by $\mathcal{L}_v$ the local Langlands correspondence for $\GL_n(F_v)$ associating $n$-dimensional Frobenius semi-simple Weil-Deligne representation of $W_{F_v}$ to irreducible admissible representation of $\GL_n(F_v)$. We use the normalization as in the works of Harris-Taylor and Henniart. In this paper we need to use the local Langlands correspondence for $\GL_n(F_v)$ when $n=1,2,4$. We denote each of these cases by $\mathcal{L}_v$ as no confusion can arise.  

We denote by $\mathbf{A}_F$ the ring of adeles of $F$, and by $|\cdot|_{\mathbf{A}_F}$ the adelic norm.

We denote by $G_L$ the absolute Galois group of any field $L$. If $E/F$ is a Galois extension, with $\tau \in \Gal(E/F)$, then for any representation $\rho$ of $G_E$, we denote by $\rho^{\tau}$ the conjugation of $\rho$ by $\tau$, given by $\rho^{\tau}(g)=\rho(\tau g \tau^{-1})$ for $g \in G_E$.

If on the other hand $E$ and $F$ are local fields, and $\pi$ is an irreducible admissible representation of $\GL_n(E)$, then we define the irreducible admissible representation $\pi^{\tau}: =\pi \circ \tau$. One has similar definition in the adelic context when $E$ and $F$ are number fields.

If $\rho$ is any $p$-adic representation of $G_F$ (with $F$ number field or local field), then for $n \in \mathbf{Z}$ we denote by $\rho(n)$ the Tate twist of $\rho$ by the $n$-th power of the $p$-adic cyclotomic character. We employ the convention that the Hodge-Tate weight of the $p$-adic cyclotomic character is $-1$.

In this paper, the symplectic similitude group $\GSp_4$ is defined with respect to the following skew-symmetric matrix:
\begin{eqnarray*}
\left( \begin{array}{rrrr} & & & 1 \\ & & 1 & \\ & -1 & & \\ -1 & & & \\ \end{array} \right) 
\end{eqnarray*}

\section{Resume on Arthur's results on endoscopic classification of automorphic representation on $\GSp_4$ }

For this section $F$ will denote a general number field. For any prime $v$ of $F$ (both finite and archimedean), we denote by $F_v$ the completion of $F$ at $v$.

\subsection{Local $L$ and $A$-parameters}

We begin by discussing the local $L$ and $A$-parameters which are needed to state Arthur's results. For more details the reader is referred to \cite{A1} or chapter 1 of \cite{A2}.

For each of the local completions $F_v$, we let $L_{F_{v}}$ be the $\SL_2$ form of the the local Langlands group of $F_v$. Thus if $W_{F_v}$ is the Weil group of $F_v$, then
\begin{eqnarray*}
L_{F_v} =  \left \{ \begin{array}{c} W_{F_v}\mbox{ if } v \mbox{ is archimedean } \\   W_{F_v} \times \SL_2(\mathbf{C}) \mbox{
otherwise. }
\end{array} \right.
\end{eqnarray*}
For $\sigma \in L_{F_v}$, we denote by $|\sigma|_v$ the absolute value of the image of $\sigma$ in $W_{F_v}$.

Let $N \geq 1$ be an integer. Since we will be working with similitude groups, we consider the group $\widetilde{G}_N := \GL_N  \times \GL_1$, as an algebraic group over $F$, whose Langlands dual group can be taken as $\widehat{\widetilde{G}}_N = \GL_N(\mathbf{C}) \times \GL_1(\mathbf{C})$. The dual group $\widehat{\widetilde{G}}_N$ is equipped with the automorphism $\widehat{\alpha}$, where 
\[
\widehat{\alpha}((x,y)) = ( x^{*} \cdot y, y) \mbox{ for } (x,y) \in \widehat{\widetilde{G}}_N.
\]
Here $x^{*} = \leftexp{t}{x}^{-1}$ is (up to conjugacy) the unique outer automorphism of $\GL_n(\mathbf{C})$. 

We denote by $\Phi_v(N)$ the set of admissible homomorphisms from $L_{F_v}$ to $\widehat{\widetilde{G}}_N$ up to $\widehat{\widetilde{G}}_N$-conjugacy:
\[
\widetilde{\phi}: L_{F_v} \rightarrow \widehat{\widetilde{G}}_N.
\]
Elements $\widetilde{\phi} \in \Phi_v(N)$ are called $L$-parameters of $\widetilde{G}_N(F_v)$, and can be written in the form:
\[
\widetilde{\phi} = \phi \oplus \chi
\]
where $\phi$ and $\chi$ are respectively $N$ and one dimensional representation of $L_{F_v}$. Thus $\chi$ is a character of $W_{F_v}$, which will consequently be identified as a character of $F_v^{\times}$ via local class field theory, and $\phi$ is an $L$-parameter for $\GL_N(F_v)$ (in the usual sense). We also denote by $\Phi_{v,\bdd}(N) \subset \Phi_v(N)$ the subset of parameters $\widetilde{\phi}$ whose image in $\widehat{\widetilde{G}}_N$ is bounded.

When $v$ is finite, so that $L_{F_v} = W_{F_v} \times \SL_2(\mathbf{C})$, and $\phi:L_{F_v} \rightarrow \GL_N(\mathbf{C})$ as above, we denote by $\phi^{ss}:W_{F_v} \rightarrow \GL_N(\mathbf{C})$, the semi-simple part of $\phi$, as the homomorphism:

\begin{eqnarray}
\phi^{ss}(w) = \phi \Big( w,    \begin{pmatrix}   |w|_v^{1/2}  &  0 \\    0 &     |w|_v^{-1/2}     \\ \end{pmatrix}  \Big) \mbox{ for } w \in W_{F_v}.
\end{eqnarray}

Now to define the $A$-parameters we enlarge the local Langlands group and form the group $L_{F_v} \times \SL_2(\mathbf{C})$ (when $v$ is finite this extra $\SL_2(\mathbf{C})$ factor of Arthur is not to be confused with the $\SL_2(\mathbf{C})$ factor for monodromy action occuring in the definition of $L_{F_v}$). Denote by $\Psi_v(N)$, the set of $A$-parameters of $\widetilde{G}_N(F_v)$, to be the set of admissible homomorphisms (again up to $\widehat{\widetilde{G}}_N$-conjugacy):
\[
\widetilde{\psi} : L_{F_v} \times \SL_2(\mathbf{C}) \rightarrow \widehat{\widetilde{G}}_N
\]
such that $\widetilde{\psi}|_{\SL_2(\mathbf{C})}$ is algebraic, and that $\psi|_{L_{F_v}}$ has bounded image. Again such a $\widetilde{\psi}$ can be written in the form $\widetilde{\psi} = \psi \oplus \chi$, where $\psi$ is an $N$-dimensional representation of $L_{F_v} \times \SL_2(\mathbf{C})$ such that $\psi|_{\SL_2(\mathbf{C})}$ is algebraic, and $\chi$ is a character of $W_{F_v}$. As in \cite{A1,A2} since one do not yet know the generalized Ramanujan conjecture (for cuspidal automorphic representations on general linear groups), one has to work with the larger set $\Psi_v^+(N)$ consisting of $\widetilde{\psi}$ as baove but without the boundedness condition on $\psi|_{L_{F_v}}$. If $\widetilde{\psi} \in \Psi_v(N)$ is trivial on $\SL_2(\mathbf{C})$ (note that this is {\it not} the $\SL_2(\mathbf{C})$ that occurs in the definition of $L_{F_v}$ when $v$ is finite), then by definition $\widetilde{\psi} \in \Phi_{v,\bdd}(N)$.

Given an $A$-parameter $\widetilde{\psi} = \psi \oplus \chi$, we denote by $\phi_{\psi}: L_{F_v} \rightarrow \GL_n(\mathbf{C})$ the homomorphism defined by:
\begin{eqnarray}
\phi_{\psi}(\sigma) = \psi\Big(\sigma,    \begin{pmatrix}   |\sigma|_v^{1/2}  &  0 \\    0 &     |\sigma|_v^{-1/2}     \\ \end{pmatrix}  \Big) \mbox{ for } \sigma \in L_{F_v}.
\end{eqnarray}
For $\widetilde{\psi} = \psi \oplus \chi$ as above, we will refer to $\widetilde{\phi}_{\widetilde{\psi}} : = \phi_{\psi} \oplus \chi$ as the $L$-parameter associated to the $A$-parameter $\widetilde{\psi}$.

We will be interested in the parameters (both $L$ and $A$) that are stable under the automorphism $\widehat{\alpha}$. Thus if $\widetilde{\psi} = \psi \oplus \chi \in \Psi_v(N)$, then $\widetilde{\psi}$ is $\widehat{\alpha}$-stable if $\widehat{\alpha} \circ \widetilde{\psi}$ is conjugate to $\widetilde{\psi}$ under $\widehat{\widetilde{G}}_N$, or equivalently the representation:
\[
\psi^* \otimes \chi : \sigma  \mapsto \psi(\sigma)^{*} \cdot \chi (\sigma) \mbox{ for } \sigma \in L_{F_v} \times \SL_2(\mathbf{C})
\]
is equivalent to $\psi$ itself, and we refer to this as saying that $\psi$ is $\chi$-self dual. Similarly for (the simpler case of) $\Phi_v(N)$. 

Let $\widetilde{\psi} = \psi \oplus \chi$ be an $\widehat{\alpha}$-stable parameter. Thus $\psi$ is $\chi$-self dual and there exists $A \in \GL_n(\mathbf{C})$ such that
\[
\psi^* \otimes \chi = A^{-1} \psi A.
\]
We say that $\widetilde{\psi}$ is symplectic (respectively orthogonal) type, if $A^t=-A$ (respectively $A^t=A$).

\bigskip

We now specialize to the case $N=4$. Put $G : = \GSp_{4}$ the algebraic group over $F$ given by the symplectic similitude group of a four dimensional symplectic vector space over $F$, and we denote by $c$ the similitude character of $\GSp_4$. The Langlands dual group $\widehat{G}$ of $G$ is given by $\widehat{G} = \GSpin_5(\mathbf{C})$, the similitude spin group in five variables. However, in this case we have the exceptional isomorphism:
\[
\GSpin_5(\mathbf{C}) \cong \GSp_4(\mathbf{C})
\] 
such that the standard four dimensional spin representation of $\GSpin_5(\mathbf{C})$ corresponds to the standard representation of $\GSp_4(\mathbf{C})$. We will used this identification in the sequel. We thus regard $\widehat{G}$ as a subgroup of $\GL_4(\mathbf{C})$, and hence we have the embedding 
\begin{eqnarray*}
\widehat{G} & \hookrightarrow & \widehat{\widetilde{G}}_4 = \GL_4(\mathbf{C}) \times \GL_1(\mathbf{C}) \\
 g & \mapsto & (g,c(g)).
\end{eqnarray*}
In the context of \cite{A1,A2} the group $G$ is an elliptic twisted endoscopic group of $\widetilde{G}_4$ with respect to $\widehat{\alpha}$. We note that in the context of \cite{A1,A2} it is ``more correct" to think of $G$ as the similitude spin group $\GSpin_{5/F}$. Indeed for general $N$ the group $\GSpin_{2N+1/F}$ (and not $\GSp_{2N/F}$ when $N >2$) occurs as twisted endoscopic group of $\widetilde{G}_N$ with respect to $\widehat{\alpha}$. However since we have the exceptional isomorphsim $\GSp_{4/F} \cong \GSpin_{5/F}$ it does not matter in this case.

We define $\Phi_v(G)$ to be the $\widehat{G}$-conjugacy classes of admissible homomorphisms $\phi: L_{F_v} \rightarrow \widehat{G}$, and with $\Phi_{v,\bdd}(G)$ to be the subset of parameters with bounded image. Similarly $\Psi_v(G)$ is the set of $\widehat{G}$-conjugacy classes of admissible homomorphisms $\psi: L_{F_v} \times \SL_2(\mathbf{C}) \rightarrow \widehat{G}$ whose restriction to $\SL_2(\mathbf{C})$ is algebraic, and whose restriction $\psi|_{L_{F_v}}$ has bounded image (one also define the set of parameters $\Psi_v^+(G)$ without the boundedness condition on the restriction to $L_{F_v}$). These are the set of local $L$ and $A$-parameters of $G(F_v)$ respectively. Given $\phi \in \Phi_v(G)$, the composite of $\phi$ with the embedding $\widehat{G} \hookrightarrow \widehat{\widetilde{G}}_4$ gives a $\widehat{\alpha}$-stable symplectic type parameter. In fact this identifies the set $\Phi_v(G)$ with the subset of (equivalence classes of) parameters in $\Phi_v(4)$ that are $\widehat{\alpha}$-stable and of symplectic type, {\it c.f.} \cite{GT2} lemma 6.1. Same remarks apply to $\Psi_v(G)$.

Given $\phi \in \Phi_v(G)$, one can define the semi-simple part of $\phi^{ss}$ of $\phi$ as in equation (2.1), and given $\psi \in \Psi_v(G)$ one defines the $L$-parameter $\phi_{\psi} \in \Phi_v(G)$ associated to $\psi$ as in equation (2.2).

Finally for future reference, given $\phi \in \Phi_v(G)$, we denote by $k_{*} \phi$ the $L$-parameter of $\GL_4(F_v)$ obtained by composing $\phi$ with the inclusion $\widehat{G} \hookrightarrow \GL_4(\mathbf{C})$. 

\subsection{The local classification}

Denote by $\Pi_v(G)$ the set of (isomorphism classes) irreducible admissible representations of $G(F_v)$, and $\Pi_{v,\temp}(G)$ the subset of tempered representations. In this subsection we state Arthur's result on the endoscopic classification of $\Pi_v(G)$.

Let $\psi \in \Psi_v(G)$. Put
\begin{eqnarray*}
S_{\psi} = \mbox{ centralizer of the image of } \psi \mbox { in } \widehat{G}.
\end{eqnarray*} 
\begin{eqnarray*}
\mathcal{S}_{\psi} = S_{\psi} / S_{\psi}^0 Z(\widehat{G}).
\end{eqnarray*}
We denote by $\widehat{\mathcal{S}}_{\psi}$ the character group of $\mathcal{S}_{\psi}$.

\begin{theorem} \cite{A1,A2}

Given any $\psi \in \Psi_v(G)$, there is a finite multi-set $\Pi_{\psi}$ (i.e. a set with multiplicities) consisting of irreducible unitary representations of $G(F_v)$, together with a map
\begin{eqnarray*}
\Pi_{\psi} & \rightarrow & \widehat{\mathcal{S}}_{\psi} \\
\pi & \mapsto & \langle \cdot , \pi \rangle
\end{eqnarray*}
which is characterized by (twisted) endoscopic transfer. 

In the particular case where $\psi = \phi \in \Phi_{v,\bdd}(G)$, then $\Pi_{\phi}$ is actually a set (i.e. with no multiplicities), and consists of tempered representations of $G(F_v)$. The map
\[
\Pi_{\phi} \rightarrow \widehat{\mathcal{S}}_{\phi}
\]
is injective, and and bijective when $v$ is finite. We have
\[
\Pi_{v,\temp}(G) = \coprod_{\phi \in \Phi_{v,\bdd}(G)} \Pi_{\phi}.
\]
\end{theorem}

For $\phi \in \Phi_{v,\bdd}(G)$, the set $\Pi_{\phi}$ is called the $L$-packet associated to $\phi$. Thus theorem 2.1 gives in particular a classification of $\Pi_{v,\temp}(G)$ by partitioning $\Pi_{v,\temp}(G)$ into $L$-packets corresponding to the set of $L$-parameters $\Phi_{v,\bdd}(G)$, and describes the elements in the $L$-packet corresponding to $\phi \in \Phi_{v,\bdd}(G)$ in terms of the characters of the component group $\mathcal{S}_{\phi}$. In our case with $G = \GSp_{4/F}$, the groups $\mathcal{S}_{\phi}$ are either trivial or isomorphic to $\mathbf{Z}/2 \mathbf{Z}$, so the size fo the $L$-packets $\Pi_{\phi}$ is either one or two. We refer the reader to chapter one of \cite{A2} for the precise meaning that the packets $\Pi_{\phi}$ and the map $\Pi_{\phi} \rightarrow \widehat{\mathcal{S}}_{\phi}$ being characterized by twisted endoscopic transfer. The usual Langlands quotient construction gives the construction of packets $\Pi_{\phi} \subset \Pi_v(G)$ for all $\phi \in \Phi_v(G)$. In the following we refer to this classification of $\Pi_v(G)$ as the local Langlands correspondence for $G(F_v)$, and we denote as:
\[
\rec_v : \Pi_v(G) \rightarrow \Phi_v(G)
\]
the surjective map sending $\pi \in \Pi_{\phi}$ to $\phi$. Under this correspondence, the central character of $\pi$ is equal to the similitude character of the parameter $\rec_v(\pi)$ (as usual identified via local class field theory).

In our particular case with $G= \GSp_4$, another construction of the local Langlands correspondence for $\GSp_4(F_v)$ is given by the works of Gan-Takeda \cite{GT1,GT2}, using the method of theta correspondence. We denote:
\[
\rec_v^{\GT} : \Pi_v(G) \rightarrow \Phi_v(G)
\]
the local Langlands correspondence as constructed in \cite{GT1,GT2}. In Chan-Gan \cite{CG} it is shown that the local Langlands correpondence for $\GSp_4$ constructed by Gan-Takeda coincides with that of Arthur, i.e. $\rec_v=\rec_v^{\GT}$, by showing that the $L$-packets constructed by Gan-Takeda satisfy the local character relations as in the case of the packets constructed by Arthur.   

Finally, for $\psi \in \Psi_v(G)$, the multi-set $\Pi_{\psi}$ (commonly called an $A$-packet in the literature) pertains to the global classification instead of the local classification, which we describe in the next subsection. We remark that the $A$-packet $\Pi_{\psi}$ in general contains representations that belong to different $L$-packets (in the sense above). In the case where $v$ is finite and $\psi$ is an unramified parameter, in the sense $\psi|_{L_{F_v}}: L_{F_v} \rightarrow \widehat{G}$ is unramified (i.e. factors through the projection $L_{F_v} \rightarrow W_{F_v} \rightarrow W_{F_v}/I_{F_v}$), then the fibre of the map $\Pi_{\psi} \rightarrow \widehat{\mathcal{S}}_{\psi}$ over the trivial character is the unique spherical representation in this packet with $L$-parameter $\phi_{\psi}$ (having no multiplicities). We refer to p. 44 of \cite{A2} for the extension of the construction of packets for parameters in the larger set $\Psi_v^+(G)$.

\subsection{Formal global parameters and the global classification}

In the absence of the global automorphic Langlands group, we define a global $A$-parameter formally (\cite{A1,A2}). We return momentarily to the setting of the group $\widetilde{G}_N = \GL_N \times \GL_1$ (over $F$). A global $A$-parameter for $\widetilde{G}_N$ is formal expression $\widetilde{\psi} = \psi \oplus \chi$, where $\chi$ is an idele class character of $\mathbf{A}_F^{\times}$, and $\psi$ is a formal unordered sum:
\[
\psi = \psi_1 \boxplus \cdots \boxplus \psi_s
\] 
where each $\psi_r$ is a formal expression
\[
\psi_r = \mu_r \boxtimes \nu_r
\]
where $\mu_r$ is a cuspidal automorphic representation of $\GL_{m_r}(\mathbf{A}_F)$, and $\nu_r$ is an algebraic representation of $\SL_2(\mathbf{C})$ of dimension $n_r$, such that $N = m_1 n_1 + \cdots m_s n_s$. If all the $\nu_r$ are the trivial representation of $\SL_2(\mathbf{C})$, then we say that $\widetilde{\psi}$ is a generic parameter. If on the other hand $s=1$, i.e. $\psi = \mu \boxtimes \nu $,  then we say that $\widetilde{\psi}$ is a simple parameter.

We say that $\widetilde{\psi} = \psi \oplus \chi$ is $\widehat{\alpha}$-discrete if the $\psi_r$'s are all distinct (that is, the pairs $(\mu_r,\nu_r)$ are distinct), and the cuspidal automorphic representations $\mu_r$ are $\chi$-self dual, i.e.
\[
\mu_r^{*} \otimes (\chi \circ \det) \cong \mu_r
\]
here $\mu_r^*$ denotes the contra-gredient of $\mu_r$ (this is to mimic the definition of $\chi$-self duality for local parameters in the previous subsection; note that the finite dimensional representations $\nu_r$ of $\SL_2(\mathbf{C})$ are all self-dual). We denote the set of such $\widehat{\alpha}$-discrete global $A$-parameters of $\widetilde{G}_N$ as $\Psi_{2,F}(N)$.

In general, given a cuspidal automorphic representation $\mu$ of $\GL_n(\mathbf{A}_F)$ that is $\chi$-self dual, i.e. satisfies $\mu^{*} \otimes (\chi \circ \det) \cong \mu$, we say that $\mu$ is of symplectic type (respectively orthogonal type) with respect to $\chi$, if the twisted exterior square $L$-function $L(s,\mu,\Lambda^2 \otimes \chi^{-1})$ (respectively the twisted symmetric square $L$-function $L(s,\mu,\Sym^2 \otimes \chi^{-1})$) has a pole at $s=1$. Under the condition that $\mu$ is $\chi$-self dual, these two cases are mutually exclusive (here and in the following, by $L$-function we only need to work with the partial $L$-function defined by Euler factors outside the set of primes where the data ramifies). Slightly more generally, for a formal expression $\mu \boxtimes v$, with $\mu$ a $\chi$ self-dual cuspidal automorphic representation of $\GL_m(\mathbf{A}_F)$, and $v$ an algebraic representation of $\SL_2(\mathbf{C})$ of dimension $n$, then we say that $\mu \boxtimes v$ is of symplectic type with respect to $\chi$, if either $n$ is odd and $\mu$ is of symplectic type with respect to $\chi$, or $n$ is even and $\mu$ is of orthogonal type with respect to $\chi$ (noting that an algebraic representation of $\SL_2(\mathbf{C})$ of dimension $n$ is symplectic if $n$ is even and orthogonal if $n$ is odd). One similarly defines the condition for $\mu \boxtimes v$ to be of orthogonal type, namely that either $n$ is odd and $\mu$ is of orthogonal type with respect to $\chi$, or $n$ is even and $\mu$ is of symplectic type with respect to $\chi$.  

Then given $\widetilde{\psi} = \psi \oplus \chi \in \Psi_{2,F}(N)$, with $\psi = \psi_1 \boxplus \cdots \boxplus \psi_s $ as above, we define $\widetilde{\psi}$ to be of symplectic type, if all the $\psi_r$ are of symplectic type with respect to $\chi$. One can similarly define the condition for $\widetilde{\psi}$ to be of orthogonal type. 

We now specialize to the case $N=4$. With $G=\GSp_{4/F}$ as before, we define $\Psi_{2,F}(G)$ to be the subset of $\Psi_{2,F}(4)$ consisting of global $A$-parameters of $\widetilde{G}_4$ of symplectic type, and by $\Psi_{2,F}(G,\chi) \subset \Psi_{2,F}(G)$ the subset of parameters with similitude character $\chi$ (if the context is clear we will then denote a general element of $\Psi_{2,F}(G,\chi)$ as $\psi$ rather than $\widetilde{\psi}$). Given $\widetilde{\psi} = \psi \oplus \chi \in \Psi_{2,F}(G,\chi)$, one can define the component group $\mathcal{S}_{\psi}$ formally, {\it c.f.} section 1.4 of \cite{A2}. If $\psi = \psi_1 \boxplus \cdots \boxplus \psi_s$, then $\mathcal{S}_{\psi} \cong \mathbf{Z}/2^{s-1} \mathbf{Z}$ (again in our case where $G = \GSp_4$ the value of $s$ is actually either $1$ or $2$).

A parameter $\widetilde{\psi} = \psi \oplus \chi \in \Psi_{2,F}(G,\chi)$ has a localization $\widetilde{\psi}_v \in \Psi^+_v(4)$ for each place $v$ of $F$. Indeed if $\psi = \psi_1 \boxplus \cdots \boxplus \psi_s$, with $\psi_r = \mu_r \boxtimes \nu_r$ as above, the local $v$-component $(\mu_r)_v$ of $\mu_r$ is an irreducible admissible representation of $\GL_{m_r}(F_v)$ that is self dual with respect to $\chi_v$. Denote by $\mathcal{L}_v((\mu_r)_v)$ the $L$-parameter corresponding to $(\mu_r)_v$ under the local Langlands correspondence for $\GL_{m_r}(F_v)$. We then define the local parameter $\psi_v$ (as an honest representation of $L_{F_v} \times \SL_2(\mathbf{C})$):
\[
\psi_v = \mathcal{L}_v((\mu_1)_v) \boxtimes \nu_1 \oplus \cdots   \oplus    \mathcal{L}_v((\mu_s)_v)  \boxtimes  \nu_s.
\]
\[
\widetilde{\psi}_v = \psi_v \oplus \chi_v.
\]
One then has $\widetilde{\psi}_v \in \Psi^+_v(4)$. One has $\psi_v \in \Phi_v(4)$ if $\psi$ is generic. 

In \cite{A1,A2} it is shown that if $\widetilde{\psi} \in \Psi_2(G,\chi)$, then $\widetilde{\psi}_v \in \Psi^+_v(G)$. Similarly if $\psi$ is generic then $\psi_v \in \Phi_v(G)$. One can also define localization map on the component groups 
\begin{eqnarray*}
\mathcal{S}_{\psi} & \rightarrow & \mathcal{S}_{\psi_v} \\
s &\mapsto & s_v. 
\end{eqnarray*}

Given $\psi \in \Psi_{2,F}(G,\chi)$, we then define the global packet $\Pi_{\psi}$ associated to $\psi$ as:
\begin{eqnarray*}
 \Pi_{\psi} &=& \otimes_v^{\prime} \Pi_{\psi_v} \\
&=& \{ \pi = \otimes_{v}^{\prime} \pi_v | \,\ \pi_v \in \Pi_{\psi_v} \mbox{ for all } v, \langle \cdot , \pi_v \rangle =1 \mbox{ for almost all } v             \}
\end{eqnarray*}
(as in the local case if $\psi$ is not generic, then $\Pi_{\psi}$ is to be regarded as a multi-set). If $\pi \in \Pi_{\psi}$, then $\pi$ defines the character $\langle \cdot , \pi \rangle$ on $\mathcal{S}_{\psi}$ via $s \mapsto \prod_v \langle s_v,\pi_v \rangle$ for $s \in \mathcal{S}_{\psi}$.

We can now state Arthur's global classification of the discrete automorphic spectrum in the particular case of $G=\GSp_4$. Denote by $L^2_{\disc}(G(F) \backslash G(\mathbf{A}_F),\chi)$ the discrete spectrum of the $L^2$ space of (essentially) square-integrable functions on $G(F) \backslash G(\mathbf{A}_F)$ with central character $\chi$.

\begin{theorem} \cite{A1,A2}
We have a decomposition:
\begin{eqnarray*}
 L^2_{\disc}(G(F) \backslash G(\mathbf{A}_F),\chi) = \bigoplus_{\psi \in \Psi_{2,F}(G,\chi)} \bigoplus_{\begin{subarray}{c} \pi \in \Pi_{\psi} \\ \langle \cdot, \pi \rangle = \epsilon_{\psi} \end{subarray}} \pi
\end{eqnarray*}
here $\epsilon_{\psi}$ is a cerain sign character of $\mathcal{S}_{\psi}$ defined in terms of symplectic root numbers (which is trivial if $\psi$ is a generic parameter).
\end{theorem}
In particular, if $\psi \in \Psi_{2,F}(G,\chi)$ is a simple generic parameter, then both $\mathcal{S}_{\psi}$ and $\epsilon_{\psi}$ are trivial, and so all the elements in the global packet $\Pi_{\psi}$ occur in $L^2_{\disc}(G(F) \backslash G(\mathbf{A}_F),\chi)$. In the case of $G=\GSp_4$, we refer to \cite{A1} for a list of the possible types of parameters $\psi$. The only case where $\epsilon_{\psi}$ is non-trivial is the Saito-Kurokawa type.

For our purpose we will be mainly interested in the cuspidal subspace $L^2_{\cusp}(G(F) \backslash G(\mathbf{A}_F),\chi) $. If $\psi \in \Psi_{2,F}(G,\chi)$ is a simple generic parameter, and if $\pi \in \Pi_{\psi}$, so $\pi$ occurs in $L^2_{\disc}(G(F) \backslash G(\mathbf{A}_F),\chi)$, then it is expected that $\pi$ actually occurs in the cuspidal spectrum \linebreak $L^2_{\cusp}(G(F) \backslash G(\mathbf{A}_F),\chi)$. This statement is not recorded in \cite{A1,A2}. However for our applications, we have the following result of Wallach:

\begin{theorem} (theorem 4.3 of \cite{Wal})
Suppose that $\pi$ occurs in \linebreak $L^2_{\disc}(G(F) \backslash G(\mathbf{A}_F),\chi)$, and up to twist $\pi$ tempered at all the archimedean places of $F$. Then $\pi$ is cuspidal.
\end{theorem}

We now draw a corollary from Arthur's classification theorem 2.2 which will be needed in section 4, on base change of automorphic representations on $\GSp_4(\mathbf{A}_F)$. Such a result is already noted by Shahidi \cite{S2} and Soudry \cite{Sou2} before.

\begin{proposition}
Let $\psi \in \Psi_{2,F}(G,\chi)$ be a simple generic parameter corresponding to a $\chi$-self dual cuspidal automorphic representation $\mu$ on $\GL_4(\mathbf{A}_F)$. Suppose that $F^{\prime}/F$ is a cyclic extension. Assume that the Arthur-Clozel base change \cite{AC} $\mu^{\prime}$ of $\mu$ to $\GL_4(\mathbf{A}_{F^{\prime}})$ is cuspidal. Then $\mu^{\prime}$ defines a simple generic parameter $\psi^{\prime} \in \Psi_{2,F^{\prime}}(G,\chi^{\prime})$ (here $\chi^{\prime} =\chi \circ N_{F^{\prime}/F}$).
\end{proposition}
\begin{proof}
The base change lift $\mu^{\prime}$ of $\mu$ exists by the main result of Arthur-Clozel \cite{AC}, which under assumption is a cuspidal automorphic representation on $\GL_4(\mathbf{A}_{F^{\prime}})$. That $\mu^{\prime}$ is self dual with respect to $\chi^{\prime}$ follows easily from the $\chi$-self-duality of $\mu$ (for example using strong multiplicity one and the relation between the Satake parameters of $\mu$ and $\mu^{\prime}$ at the unramified places). All we need to check is that $\mu^{\prime}$ is of symplectic type with respect to $\chi^{\prime}$, i.e. that the twisted exterior square $L$-function $L(s,\mu^{\prime},\Lambda^2 \otimes (\chi^{\prime})^{-1})$ has a pole at $s=1$. However we have the factorization: denoting $C_{F^{\prime}/F} :=\mathbf{A}_F^{\times}/F^{\times} N_{F^{\prime}/F} \mathbf{A}^{\times}_{F^{\prime}}$, we have
\[
L(s,\mu^{\prime},\Lambda^2 \otimes (\chi^{\prime})^{-1}) = \prod_{\eta \in \widehat{C}_{F^{\prime}/F}} L(s,\mu, \Lambda^2 \otimes (\chi^{-1} \cdot \eta ).
\]
The term with $\eta$ being trivial contributes a pole at $s=1$, while for the other terms, the main result of \cite{S2} asserts that, since $\mu$ is cuspidal, the twisted exterior square $L$-function $L(s,\mu, \Lambda^2 \otimes (\chi^{-1} \cdot \eta )$ is non-zero at $s=1$. This concludes the result.
\end{proof}

\begin{rem}
\end{rem}
By \cite{AC} the condition for $\mu^{\prime}$ to be cuspidal is that
\begin{eqnarray}
\mu \otimes (\eta \circ \det) \ncong \mu
\end{eqnarray}
for any non-trivial character $\eta$ of $C_{F^{\prime}/F} $.

\section{Galois Representations attached to cusp forms on $\GSp_4(\mathbf{A}_F)$ of cohomological type}

For the rest of the paper, $F$ will denote a totally real number field. In this section we state the result on Galois representations attached to cuspidal automorphic representations on $\GSp_4(\mathbf{A}_F)$ of cohomological type that is needed for the constructions in section 4. 

\subsection{Archimedean $L$-parameters}

We begin by defining some archimedean $L$-parameters. Recall that $L_{\mathbf{C}} = W_{\mathbf{C}} = \mathbf{C}^{\times}$, and $L_{\mathbf{R}} = W_{\mathbf{R}} = \mathbf{C}^{\times} \cup \mathbf{C}^{\times} j$, with $j^2=-1$, and $j z j^{-1} = \overline{z}$ for $z \in \mathbf{C}^{\times}$.

We first begin with the case of $\GL_2(\mathbf{R})$ and $\GL_2(\mathbf{C})$. For integers $w,n$ with $n \geq 0$ and $n \equiv w+1 \mod{2}$, define
\begin{eqnarray}
 \phi_{w,n} : W_{\mathbf{R}} & \rightarrow & \GL_2(\mathbf{C})  \nonumber \\
z &\mapsto & |z|^{-w} \cdot \begin{pmatrix} (z/\overline{z})^{n/2} & \\ & (z/\overline{z})^{-n/2} \end{pmatrix}  \mbox{ for } z \in \mathbf{C}^{\times} \\
j  &\mapsto & \begin{pmatrix} & +1 \\  (-1)^n &  \end{pmatrix}. \nonumber
\end{eqnarray}

When $n \geq 1$ The $L$-parameter $\phi_{w,n}$ corresponds to the (essentially) discrete series representation of $\GL_2(\mathbf{R})$ with central character $a \mapsto a^{-w}$, and are cohomological representations. The case $n=0$ corresponds to the limit of discrete series representation (up to twist). 

For $\GL_2(\mathbf{C})$ one defines the parameters $\phi_{w,n} : W_{\mathbf{C}} \rightarrow \GL_2(\mathbf{C})$ (again with $n \equiv w+1 \mod{2}$) by the same formula

\begin{eqnarray*}
\phi_{w,n}: z  \rightarrow |z|^{-w} \cdot  \begin{pmatrix} (z/\overline{z})^{n/2} & \\ & (z/\overline{z})^{-n/2} \end{pmatrix}.
\end{eqnarray*}

When $n \geq 1$, the parameter $\phi_{w,n}$ corresponds to an irreducible admissible represention of $\GL_2(\mathbf{C})$ which is cohomological (there are no discrete series representations for $\GL_2(\mathbf{C})$).

Now we come back to $\GSp_4(\mathbf{R})$. Let $w,m_1,m_2$ be integers such that $m_1 > m_2 \geq 0$ and $m_1 + m_2 \equiv w + 1 \mod{2}$. We denote by $\phi_{(w;m_1,m_2)}$ the $L$-parameter given by 
\begin{eqnarray}
& & \\
& & \phi_{(w;m_1,m_2)} : z \mapsto  \nonumber \\ 
& & |z|^{-w} \cdot  \begin{pmatrix}  (z/\overline{z})^{(m_1+m_2)/2}   &  &  &   \\     &   (z/\overline{z})^{(m_1 - m_2)/2}   &  &     \\   &  & (z/\overline{z})^{-(m_1 -m_2)/2}  &     \\  &  &  & (z/\overline{z})^{-(m_1+m_2)/2}        \end{pmatrix}  \nonumber \\
& & \mbox{ for } z \in \mathbf{C}^{\times} \nonumber
\end{eqnarray}
and
\begin{eqnarray}
& & \\
& & \phi_{(w;m_1,m_2)}(j) =  \begin{pmatrix}  & & &+1 \\ &   & +1   & \\ &  (-1)^{w+1}  & & \\ (-1)^{w+1}  & & &   \end{pmatrix}. \nonumber
\end{eqnarray}

The image of $\phi_{(w;m_1,m_2)}$ lies in $\widehat{G} = \GSp_4(\mathbf{C})$, with similitude character given by the character $z \mapsto |z|^{-w}_{\mathbf{C}}=|z|^{-2 w}$, and $j \mapsto (-1)^{w}$. It thus defines an element of $\Phi_{\mathbf{R}}(G)$. The archimedean $L$-packet of $\GSp_4(\mathbf{R})$ corresponding to $\phi_{(w;m_1,m_2)}$ has two elements
\begin{eqnarray}
\{ \pi^{W}_{(w;m_1,m_2)}, \pi^{H}_{(w;m_1,m_2)}        \}
\end{eqnarray}
whose central character is given by $a \mapsto a^{-w}$ for $a \in \mathbf{R}^{\times}$.

The representation $\pi^{W}_{(w;m_1,m_2)}$ is the generic representation in the packet (with $W$ stands for Whittaker), while $\pi^{H}_{(w;m_1,m_2)}$ is the non-generic one. When $m_2 \geq 1$, both $\pi^{W}_{(w;m_1,m_2)}, \pi^{H}_{(w;m_1,m_2)}$ are up to twist belong to the discrete series of $\GSp_4(\mathbf{R})$ and are cohomological representations, with $\pi^{W}_{(w;m_1,m_2)}$ the generic discrete series and $\pi^{H}_{(w;m_1,m_2)}$ in the holomorphic discrete series (hence the superscript $H$). When $m_2=0$ the representations $\pi^{W}_{(w;m_1,0)} ,\pi^{H}_{(w;m_1,0)}   $ belong to the generic limit of discrete series, and holomorphic limit of discrete series respectively (again up to twist). These constitute up to twist the non-degenerate limit of discrete series of $\GSp_4(\mathbf{R})$.

The representations $\pi^W_{(w;m_1,m_2)}, \pi^H_{(w;m_1,m_2)}$ (including both discrete series and non-degenerate limit of discrete series) are essentially tempered, and tempered when $w=0$.

In the holomorphic case it is common to introduce the Blattner parameter:
\begin{eqnarray}
k_1 = m_1 + 1, \,\ k_2 = m_2 + 2
\end{eqnarray}
thus $k_1 \geq k_2 \geq 2$ and the cohomological condition corresponds to $k_2 \geq 3$. The parameter $(k_1,k_2)$ corresponds to the minimal $K$-type of the representation $\pi^{H}_{(w;m_1,m_2)}$ (here $K=U(2)$ is the maximal compact subgroup of $\GSp_4(\mathbf{R})$), given by the representation
\[
\Sym^{k_1-k_2} \mathbf{C}^2 \otimes (\det)^{\otimes k_2}.
\]  

\subsection{Results on Galois representations}

For the rest of this section $\Pi$ will denote a cuspidal automorphic representation of $\GSp_4(\mathbf{A}_F)$ with central character $\chi$. We make the following hypotheses on $\Pi$: there is an integer $w \in \mathbf{Z}$, such that for all archimedean place $v$ of $F$, the component of $\chi$ at $v$ is given by $\chi_v: a \rightarrow a^{-w}$. Assume that for any such $v$, the component $\Pi_v$ is a (essentially) discrete series belonging to the $L$-packet defined by the parameter $\phi_{(w;m_{1,v},m_{2,v})}$ with $m_{1,v} > m_{2,v} \geq 1$ and $w + 1 \equiv m_{1,v} +m_{2,v} \mod{2}$.

In this subsection we state the result on Galois representation associated to $\Pi$. The key case is when $\Pi$ belongs to a global simple generic parameter, so we begin with this case.

\begin{theorem}
Suppose that $\Pi$ belongs to a simple generic parameter. Then for each prime $p$, there exists a continuous semi-simple Galois representation
\[
R_p: G_F \rightarrow \GL_4(\barQp)
\]
which is unramified outside the primes of $F$ dividing $p$ and where $\Pi$ is not spherical, and satisfy the local-global compatibility condition: for any prime $v$ of $F$, we have
\[
 \iota_p \WD(R_p|_{G_{F_v}})^{F-ss} \cong  k_{*} \rec_v(\Pi_v \otimes |c|_v^{-3/2}).
\]
Furthermore at each place $v$ of $F$ dividing $p$, the representation $R_p$ is deRham. If we identify the embeddings $F \hookrightarrow \barQp$ with the archimedean places of $F$ via $\iota_p : \barQp \cong \mathbf{C}$, then the Hodge-Tate weights at the embedding corresponding to $v | \infty$ are given by 
\begin{eqnarray}
\{ \delta_v , \delta_v + m_{2,v},\delta_v + m_{1,v},\delta_v + m_{1,v} + m_{2,v}  \}
\end{eqnarray}
where $\delta_v = \frac{1}{2}(w +3 -m_{1,v} - m_{2,v})$.
If $\Pi_v$ is spherical at $v|p$ then $R_p$ is crystalline at $v$. Furthermore, denoting by $P^{\cris}_{\Pi,v}(X)$ the inverse characteristic polynomial for the crystalline Frobenius on the crystalline representation $R_p|_{G_{F_v}}$, and by $Q_{\Pi,v}$ the inverse characteristic polynomial of the geometric Frobenius on the unramified Weil-Deligne representation $k_{*}\rec_v(\Pi_v \otimes |c|_v^{-3/2})$, we have the equality
\[
P^{\cris}_{\Pi,v}(X)=Q_{\Pi,v}(X).
\]
\end{theorem}

In the above we recall that if $\phi$ is a $L$-parameter of $G(F_v)$, then we denote by $k_{*} \phi$ the $L$-parameter of $\GL_4(F_v)$ obtained by composing $\phi$ with the inclusion $\widehat{G} \hookrightarrow \GL_4(\mathbf{C})$.

\begin{rem}
\end{rem}
Here we are identifying an $L$-parameter $\phi:L_{F_v} = W_{F_v} \times \SL_2(\mathbf{C}) \rightarrow \GL_n(\mathbf{C})$ with the corresponding Frobenius semi-simple Weil-Deligne representation in the usual way: corresponding to $\phi$ is the Weil-Deligne representation $(r,N)$, where $r$, the semi-simple part of the Weil-Deligne representation, is given by the semisimple part $\phi^{ss}$ of $\phi$ as in equation (2.1), i.e.
\[
r(g) = \phi^{ss}(g) = \phi\Big(g,   \begin{pmatrix} |g|_v^{1/2} &  0 \\  0 & |g|_v^{-1/2} \end{pmatrix}     \Big)
\]
for $g \in W_{F_v}$, and the monodromy operator $N$ is given by
\[
N = \log \Big( \phi\Big(1,  \left( \begin{array}{rr} 1 &  1 \\  0 & 1 \end{array} \right)   \Big) \Big).
\]

\begin{proof}
This theorem is essentially proved in \cite{So}, except that when \cite{So} was written, the results of Arthur \cite{A2} and that of \cite{ChH}, \cite{BLGGT}, \cite{Car1,Car2} were not yet available, so extraneous hypotheses on $\pi$ was made in {\it loc. cit.} So we just indicate a few main points:

1. In \cite{So} $\Pi$ is assumed to be globally generic in order to apply the results of Jacquet-Piateski-Shapiro-Shalika and Asgari-Shahidi \cite{AS} in order to lift $\Pi$ to an automorphic representation $\Pi^{\prime}$ on $\GL_4(\mathbf{A}_F)$ (assumed to be cuspidal). In our case, in the context of Arthur's global classification theorem, this assumption is not needed, and by definition the lifting $\Pi^{\prime}$ is given by the global parameter classifying $\Pi$.

2. In \cite{So} the local Langlands correspondence for $\GSp_4$ as constructed by Gan-Takeda \cite{GT1,GT2} was used for the local-global compatibility statement, and where the following fact is used \cite{GT2}: when $\Pi$ is globally generic the lifiting $\Pi \rightarrow \Pi^{\prime}$ (with $\Pi^{\prime}$ being cuspidal) is compatible with the local Langlands correspondence for $\GSp_4$ constructed by Gan-Takeda, and the local Langlands correspondence for $\GL_4,$ at all places. Again in our case in the formalism of Arthur's classification this is tautological, provided we use the local Langlands correspondence for $\GSp_4$ as constructed by Arthur in the context of the local classification theorem (of course in any case by the main result of of Chan-Gan \cite{CG} these two gives the same local Langlands correspondence for $\GSp_4$).    

3. The Galois representation $R_p$ associated to $\Pi$ is by definition the Galois representation associated to $\Pi^{\prime}$. The arguments of \cite{So} constructs the latter if we use the result of \cite{ChH} in the arguement of \cite{So}. The part of the local-global compatibility statement concerning the monodromy operator in \cite{ChH} is completed in the works \cite{BLGGT}, \cite{Car1,Car2}, while the assertion on the equality of the crystalline Frobenius polynomial and the Hecke polynomial when $\Pi$ is spherical at $v |p$ is proved in \cite{Ch}
 \end{proof}

\begin{rem}
\end{rem}

\noindent 1. As in \cite{So} a Baire category argument shows that the image of $G_F$ lies in $\GL_4(L)$ for some finite extension $L$ of $\mathbf{Q}_p$.

\noindent 2. For the applications in section 4 we will only use the local-global compatibility statement up to semi-simplification at primes of $F$ not dividing $p$.

\noindent 3. The main theorem of Bellaiche-Chenevier \cite{BC2} gives the result that the image of $\rho_p$ lies in $\GSp_4(\barQp)$. For our purpose we do not need this more precise result.

\noindent 4. The irreducibility of $R_p$ is now known by Calegari-Gee \cite{CaG}.

\bigskip

For completeness we state the result for Galois representations when the parameter classifying $\pi$ is not of simple generic type. This essentially reduces to the Galois representation associated to cuspidal automorphic representations on $\GL_2(\mathbf{A}_F)$ of cohomological type, i.e. to Galois representations associated to Hilbert modular forms.

We first consider the case where the parameter classifying $\Pi$ is still generic but not of simple type (i.e. endoscopic type). In this case the parameter $\widetilde{\psi} = \psi \oplus \chi$ classifying $\Pi$ has the form $\psi = \mu_1 \boxplus \mu_2$, where $\mu_1$, $\mu_2$ are distinct cuspidal automorphic representations of $\GL_2(\mathbf{A}_F)$ with central character $\chi$. In \cite{A1} it is called a parameter of Yoshida type. We can arrange $\mu_1$, $\mu_2$ so that at any archimedean place $v$ of $F$, the $L$-parameter of $\mu_1$ at $v$ is given by $\phi_{(w;m_{1,v}+m_{2,v})}$, while that of $\mu_2$ is given by $\phi_{(w;m_{1,v}-m_{2,v})}$ (both are discrete series parameters since $m_{1,v} > m_{2,v} \geq 1$).  

It follows that both $\mu_1$ and $\mu_2$ are of cohomological type, and hence we can associate Galois representations
\[
\rho_{\mu_i,p} :G_F \rightarrow \GL_2(\barQp)
\]
for $i=1,2$, such that for any prime $v$ of $F$, we have the local-global compatibility statement:
\begin{eqnarray}
\iota_p \WD(\rho_{\mu_i,p}|_{G_{F_v}})^{F-ss} \cong \mathcal{L}_v((\mu_i)_v \otimes |\det|_v^{-1/2}).
\end{eqnarray}
The Galois representations $\rho_{\mu_i,p}$ are deRham at all places of $v$ of $F$ above $p$, and crystalline at $v$ if $\mu_i$ is spherical at $v$ (in this case of Yoshida type packet we have $\Pi_v$ being spherical if and only if both $\mu_1$ and $\mu_2$ are spherical at $v$). The Hodge-Tate weights of $\mu_1$ at an embedding of $F$ into $\barQp$ that corresponds to an archimedean place $v$ of $F$ (under $\iota_p: \barQp \cong \mathbf{C}$) is given by:
\begin{eqnarray}
\{  \gamma_v  ,\gamma_v + m_{1,v}+m_{2,v} \}
\end{eqnarray}
where $\gamma_v = \frac{1}{2}(w +1-( m_{1,v}+m_{2,v}) )$. For $\mu_2$ the Hodge-Tate weights are given by 
\begin{eqnarray}
\{\gamma^{\prime}_v, \gamma^{\prime}_v  + m_{1,v}-m_{2,v}\}
\end{eqnarray}
with $\gamma^{\prime}_v = \frac{1}{2}(w+1 -(m_{1,v}-m_{2,v}) )$. It follows that if we define the Galois representation associated to $\Pi$ to be
\begin{eqnarray}
R_p:=\rho_{\mu_1,p} (-1) \oplus \rho_{\mu_2,p} (-1)
\end{eqnarray}
then all the assertions in theorem 3.1 hold for $R_p$ in this case also (that the Hodge-Tate weights of $R_p$ is also given by (3.6) follows easily from (3.8) and (3.9)).

Finally we treat the case where the parameter $\widetilde{\psi} = \psi \oplus \chi$ is not generic. In our case $\Pi$ is cuspidal. So if $\Pi$ belongs to the packet whose parameter is not generic, then $\Pi$ is a CAP representation, which were classified by Soudry \cite{Sou1}. In our case where we assume that $\Pi$ is of cohomological type, the only possibility that can occur is the case where the parameter classifying $\Pi$ is of Saito-Kurokawa type, where $\psi$ is of the form
\[
\psi = \mu \boxplus \lambda \boxtimes v(2)
\]
here $\mu$ is a cuspidal automorphic representation of $\GL_2(\mathbf{A}_F)$ with central character $\chi$, and $\lambda$ an idele class character of $\mathbf{A}_F^{\times}$ satisfying $\chi=\lambda^2$, and $v(2)$ is the algebraic representation of $\SL_2(\mathbf{C})$ of dimension $2$, {\it c.f.} \cite{A1} p.79. Note that for any finite place $v$ of $F$, the local $L$-parameter $\phi_{\psi_v}$ of $G(F_v)$ associated to $\psi_v$ is given by (as in equation (2.2)):
\begin{eqnarray}
\phi_{\psi_v}(\sigma) = \mathcal{L}_v(\mu_v)(\sigma)  \oplus  \lambda_v(\sigma) |\sigma |_v^{1/2} \oplus \lambda_v(\sigma) |\sigma |_v^{-1/2}
\end{eqnarray}
for $\sigma \in L_{F_v}$.

In our case where $\Pi$ is cohomological the archimedean $L$-parameter has to be in specific form: more precisely one must have
\begin{eqnarray}
m_{1,v}-m_{2,v}=1.
\end{eqnarray}
In particular since we have the condition $w + 1 \equiv m_{1,v}+m_{2,v} \mod{2}$ we see that $w$ has to be even, and hence the idele class character $\lambda$ is an algebraic idele class character, and $\mu$ is of cohomological type, namely that for each archimedean place $v$ of $F$, the $L$-parameer of $\mu_v$ is given by $\phi_{(w; m_{1,v}+m_{2,v})}= \phi_{(w; 2m_{1,v}-1)}$. We define the Galois representation $R_p$ associated to $\Pi$ to be
\begin{eqnarray}
R_p : = \rho_{\mu,p} (-1) \oplus  \lambda_p^{\Gal} (-1)  \oplus \lambda_p^{\Gal} (-2)
\end{eqnarray}
here $\lambda_p^{\Gal}$ is the $p$-adic Galois character of $G_F$ associated to the algebraic idele class character $\lambda$.

In the case where the parameter is not generic, then in general one does not have the full local-global compatibility statement, but only local-global compatibility up to semi-simplification: for any finite prime $v$ of $F$ not dividing $p$, we have
\begin{eqnarray}
\iota_p \WD(R_p|_{G_{F_v}})^{ss} \cong k_{*} \rec_v(\Pi_v \otimes |c|_v^{-3/2})^{ss}.
\end{eqnarray}

In order to establish (3.14) we first recall from section 2.2 that when we have a local $A$-parameter $\psi_v$, then in general the local $A$-packet corresponding to $\psi_v$ will contain representations that belong to different $L$-packets. However, we still have:

\begin{proposition}
Suppose $\psi$ is a global formal parameter of Saito-Kurokawa type. For each finite prime $v$, let $\psi_v$ be the local A-parameter of $\GSp_4(F_v)$ given by the localization of $\psi$ at $v$. Denote as before by $\Pi_{\psi_v}$ the local A-packet classified by $\psi_v$. Then the semi-simple part of the $L$-parameters of the representations in the $A$-packet $\Pi_{\psi_v}$ have the same semi-simple part as that of $\phi_{\psi_v}$. 
\end{proposition}
\begin{proof}
See Appendix.
\end{proof}

Hence by (3.11) and proposition 3.4, we have

\begin{eqnarray}
& & k_{*} \rec_v(\Pi_v \otimes |c|_v^{-3/2})^{ss} \\ & \cong & \mathcal{L}_v(\mu_v \otimes |\det|_v^{-3/2})^{ss} \oplus \mathcal{L}_v(\lambda_v \cdot |\det|_v^{-1}) \oplus \mathcal{L}_v(\lambda_v \cdot |\det|_v^{-2}). \nonumber
\end{eqnarray}

From this we see that (3.14) follows from the local-global compatibility statement for $\rho_{\mu,p}$ (as in (3.7)).

If $\Pi$ is spherical at $v$, then both $\mu$ and $\lambda$ are spherical at $v$, so $R_p$ is clearly crystalline at $v$ in this case. Finally one sees again (using (3.12)) that the Hodge-Tate weights are given by (3.6) (noting that the Hodge-Tate weight of $\lambda$ at the embedding of $F$ into $\barQp$ corresponding to any archimedean place $v$ of $F$ is given by $w/2$, while the Hodge-Tate weights of $\rho_{\mu,p}$ is given by the same formula as in (3.8)).

\bigskip

We summarize the discussion of this section as:

\begin{theorem}
Suppose that $\Pi$ is a cuspidal automorphic representation on $\GSp_4(\mathbf{A}_F)$ satisfying the hypotheses in the beginning of section 3.2. Then for each prime $p$, there exists a continuous semi-simple Galois representation
\[
R_p: G_F \rightarrow \GL_4(\barQp)
\]
which is unramified outside the primes of $F$ dividing $p$ and where $\Pi$ is not spherical, and satisfy the local-global compatibility condition, up to semi-simplification: for any prime $v$ of $F$, we have
\[
 \iota_p \WD(R_p|_{G_{F_v}})^{ss} \cong  k_{*} \rec_v(\Pi_v \otimes |c|_v^{-3/2})^{ss},
\]
If in addition $\Pi$ is classified by a global generic parameter (i.e. $\Pi$ is not a CAP representation), then we have the full local-global compatibility:
\[
 \iota_p \WD(R_p|_{G_{F_v}})^{F-ss} \cong  k_{*} \rec_v(\Pi_v \otimes |c|_v^{-3/2}),
\]
Furthermore at each place $v$ of $F$ dividing $p$, the representation $R_p$ is deRham. If we identify the embeddings $F \hookrightarrow \barQp$ with the archimedean places of $F$ via $\iota_p : \barQp \cong \mathbf{C}$, then the Hodge-Tate weights at the embedding corresponding to $v | \infty$ are given by 
\begin{eqnarray}
\{ \delta_v , \delta_v + m_{2,v},\delta_v + m_{1,v},\delta_v + m_{1,v} + m_{2,v}  \}
\end{eqnarray}
where $\delta_v = \frac{1}{2}(w +3 -m_{1,v} - m_{2,v})$.
If $\Pi_v$ is spherical at $v|p$ then $R_p$ is crystalline at $v$. Furthermore, denoting by $P^{\cris}_{\Pi,v}(X)$ the inverse characteristic polynomial for the crystalline Frobenius on the crystalline representation $R_p|_{G_{F_v}}$, and by $Q_{\Pi,v}$ the inverse characteristic polynomial of the geometric Frobenius on the unramified Weil-Deligne representation $k_{*}\rec_v(\Pi_v \otimes |c|_v^{-3/2})$, we have the equality
\[
P^{\cris}_{\Pi,v}(X)=Q_{\Pi,v}(X).
\]
\end{theorem}

\section{The case of holomorphic limit of discrete series}

As in the previous section $\Pi$ is a cuspidal automorphic representation on $\GSp_4(\mathbf{A}_F)$, with central character $\chi$. As before we assume that there is an integer $w$ such that for any archimedean place $v$ of $F$, the local component of $\chi$ is given by $\chi_v: a \rightarrow a^{-w}$. In this section we assume that the local components $\Pi_v$ of $\Pi$ at all $v|\infty$ are given by either holomorphic discrete series or the holomorphic limit of discrete series, belonging to the $L$-parameter $\phi_{(w,m_{1,v},m_{2,v})}$, where $m_{1,v} > m_{2,v} \geq 0 $ and $w +1 \equiv m_{1,v} + m_{2,v} \mod{2}$ for all $v | \infty$. In particular it can be represented (up to twist) as a holomorphic Siegel-Hilbert modular form with Blattner parameter $(k_{1,v},k_{2,v})$ where $k_{1,v}=m_{1,v}+1,k_{2,v}=m_{2,v}+2$, for any $v | \infty$ (see equation (3.5)). For the applications in section 5 we are mainly interested in the case where $\Pi_v$ belongs to the holomorphic limit of discrete series for all $v|\infty$. The goal of this section is to construct the four dimensional $p$-adic Galois representation associated to $\Pi$, using idea similar to \cite{ChH}, and to prove a local-global compatibility statement up to semi-simplification (at primes away from $p$), using the technique of $p$-adic deformation and the results on Galois representation associated to forms of cohomologial type from section 3. 

\begin{rem}
\end{rem}
A usual Cebotarev density and Brauer-Nesbitt argument shows that the semi-simple $p$-adic Galois representation associated to $\Pi$ depends only on the the global packet to which $\Pi$ belongs. In particular the results of this section apply to $\Pi$ whose archimedean components belong to a non-degenerate limit of discrete series (i.e. generic or holomorphic), as long as the packet containing $\Pi$ contains a representation that is a holomorphic limit of discrete series or holomorphic discrete series at all $v | \infty$. This will be the case for example when $\Pi$ is classified by a simple generic parameter (which is the case in section 5).

First we make

\begin{hypothesis}
The representation $\Pi$ has Iwahori fixed vectors at all primes of $F$ above $p$ with respect to the standard Iwahori subgroup $J_v \subset G(\mathcal{O}_{F_v}) = \GSp_4(\mathcal{O}_{F_v})$.
\end{hypothesis}

Here the Iwahori subgroup $J_v$ is defined with respect to the standard Borel subgroup of $\GSp_4$. If $\Pi$ satisfies hypothesis 4.2 then in particular $\Pi$ satisfies the finite slope condition \cite{MT} at all primes of $F$ above $p$.

We need the following result from \cite{MT}, which is a generalization of the result of Kisin-Lai \cite{KL} and A. Jorza [J] to the Siegel-Hilbert case. To state the result we need some notations. For each finite $v$ denote by $\mathcal{H}_v^{\sph} $ the spherical Hecke algebra of $G(F_v)=\GSp_4(F_v)$ with coefficients in $\mathbf{Z}$, defined with respect to the hyperspecial maximal compact subgroups $K_v =\GSp_4(\mathcal{O}_{F_v})$. Choose a local uniformizer $\varpi_v$ of $\mathcal{O}_{F_v}$ at each finite $v$. We also denote by $S_p$ the set of primes of $F$ above $p$, and 
\[
\mathcal{H}^p : =\otimes^{\prime}_{v \notin S_p} \mathcal{H}_v
\]
with $\mathcal{H}_v$ the full Hecke algebra of $\GSp_4(F_v)$ with coefficients in $\mathbf{Q}$.

Denote by $S$ a finite set of primes of $F$ outside of which $\Pi$ is spherical. Put $\mathcal{H}^{\sph,S,p} = \otimes^{\prime}_{v \notin S \cup S_p} \mathcal{H}_v^{\sph}$. We have the natural inclusion $\mathcal{H}^{\sph,S,p} \subset \mathcal{H}^p$ (elements of the former embed into the latter whose components at $v \in S \backslash S_p$ are given by the identity element of $\mathcal{H}_v^{\sph}$). Also as in \cite{T1} ({\it c.f.} p. 316 of {\it loc. cit.}) we have the following elements $T_v,S_v,R_v \in \mathcal{H}_v^{\sph}$ for each finite $v$:

\begin{eqnarray}
T_v & = & K_v \left( \begin{array}{rrrr} 1  & & & \\ & 1 &   & \\ & & \varpi_v & \\ & & & \varpi_v \end{array} \right) K_v  \nonumber \\
S_v & = & K_v \left( \begin{array}{rrrr} \varpi_v  & & & \\ & \varpi_v &   & \\ & & \varpi_v & \\ & & & \varpi_v \end{array} \right) K_v  \\
R_v &=& K_v \left( \begin{array}{rrrr} 1  & & & \\ & \varpi_v &   & \\ & & \varpi_v & \\ & & & \varpi_v^2 \end{array} \right) K_v  + (1-Nv^2) S_v. \nonumber
\end{eqnarray}

We now state the result on $p$-adic deformation. We remark that the method of proof as well as the validity of the conclusion relies crucially on the assumption that $\Pi$ is of holomorphic type instead of being of generic type at the archimedean places.

\begin{theorem} (\cite{MT})
Under hypothesis 4.2 on $\Pi$, there is a one dimensional rigid analytic affinoid disk $D \subset \mathbf{A}^1_{/L}$ defined over some finite extension $L$ of $\mathbf{Q}_p$ centered at the origin $0 \in D$, a ring homomorphism 
\begin{eqnarray}
\beta: \mathcal{H}^{\sph,S,p} \rightarrow \mathcal{O}(D)
\end{eqnarray}
and a Zariski dense subset $Z \subset D(L)$ accumulating at $0 \in D$, with the points $t \in D(L)$ corresponding to integers; identifying $Z$ as a subset of the integers one has $0 \in Z$ and $Z \subset \mathbf{N}$. Furthermore, there exists a positive integer $h_0$, satisfying the following conditions:

\bigskip

1. $\beta(\mathcal{H}^{\sph,S,p}) \subset \mathcal{O}(D)^{\leq 1}.$ Here $\mathcal{O}(D)^{\leq 1} \subset \mathcal{O}(D)$ is the subring of elements with norm bounded by one.

\bigskip

2. For each $t \in Z$, there is a cuspidal automorphic representation $\Pi(t)$ on $\GSp_4(\mathbf{A}_F)$, with $\Pi(0)=\Pi$, such that $\Pi(t)$ is spherical outside $S \cup S_p$ and whose archimedean components $\Pi(t)_v$ at an archimedean place $v$ of $F$ is given by $\pi^{H}_{(w(t) ; m_{1,v}(t),m_{2,v}(t) )}$. Here 
\[
w(t) = w +2(p-1) \cdot h_0 \cdot t
\]
\[
m_{1,v}(t) = m_{1,v} + (p-1) \cdot h_0 \cdot t, \,\ m_{2,v}(t) = m_{2,v} + (p-1) \cdot h_0 \cdot t
\]
the central character of $\Pi(t)$ is given by a character $\chi(t)$, such that for any archimedean place $v$ of $F$, we have $\chi(t)_v : a \rightarrow a^{- w(t)}$. Thus it is a ``parallel weight" deformation.

\bigskip

3. Define the polynomial $\mathcal{Q}_v(X) \in \mathcal{H}_v^{\sph}[X]$:
\begin{eqnarray}
& & \mathcal{Q}_v(X)\\ & = & 1 - T_v X + (Nv R_v + 2Nv^3S_v)X^2 - Nv^3 T_v S_v X^3 + Nv^6S_v^2 X^4. \nonumber
\end{eqnarray}

Then for each $t \in Z$ (in particular including $t=0$ where $\Pi(0)=\Pi$), and each $v \notin S \cup S_p$ we have, denoting by $ev_t : \mathcal{O}(D) \rightarrow L$ the evaluation map at $t$, the following:

\begin{eqnarray}
ev_{t,*} \beta_* \mathcal{Q}_v [X] = \iota_p^{-1} Q_{v,t}[X]
\end{eqnarray}
where $Q_{v,t}[X]$ is the inverse characteristic polynomial of the unramified  Weil-Deligne representation $k_* \rec_v(\Pi(t)_v \otimes |c|_v^{-3/2})$ evaluated at the geometric Frobenius element $\Frob_v$ at $v$ (in other words $Q_{v,t}(Nv^{-(s+3/2)})^{-1}$ is the standard local $L$-factor for $\Pi(t)$ with respect to the standard representation of $\widehat{G}=\GSp_4(\mathbf{C})$ at the spherical place $v$).
\end{theorem}

We remark that in the context of theorem 4.3, the points $t \in Z$ corresponds in addition to the system of Hecke eigenvalues a choice of {\it refinement data} for $\Pi(t)_v$ at primes $v$ of $F$ dividing $p$. Since this extra data plays no role in the construction of the Galois representation associated to $\Pi$ below, it is suppressed in the statement of the theorem above.

\bigskip

We now apply theorem 4.3 to the construction of the Galois representation associated to $\Pi$ when hypothesis 4.2 is satisfied. We refer to section 1 of \cite{T1} or chapter 1 of \cite{BC1} for the notion and basic property of pseudo-character (a brief discussion follows after theorem 4.4 below).

\begin{theorem}
There is a continuous pseudo-character
\begin{eqnarray}
T : G_{F,S,p} \rightarrow \mathcal{O}(D)^{\leq 1}
\end{eqnarray}
of dimension $4$, with $G_{F,S,p}$ being the Galois group of the maximal extension of $F$ unramified outside $S \cup S_p$, such that for any prime $v \notin S \cup S_p$, we have
\begin{eqnarray}
T(\Frob_v) = \beta(T_v)
\end{eqnarray}
(the dimension of a pseudo-character is simply the value of the pseudo-character at the identity element).
\end{theorem}
\begin{proof}
The proof is a standard application of the theory of pseudo-characters, the results of section 3, together with theorem 4.3.See for example \cite{BC1} Corollary 7.5.4. 
\end{proof}
    
We can also do slightly better. For this we need some notation ({\it c.f.} section 1.2.1-1.2.2 of \cite{BC1}). In general if $T:G \rightarrow A$ is a function on a group $G$ (topological or otherwise) with values in some $\mathbf{Q}$-algebra $A$, which is central (i.e. $T(xy)=T(yx)$ for all $x,y \in G$), define, for each integer $k \geq 0$, the function:
\[
S_k(T) : G^k \rightarrow A
\]
\[
S_k(T)(x) = \sum_{\sigma \in \mathcal{S}_k} \epsilon(\sigma) T^{\sigma}(x)
\]
(with the convention $S_0(T) \equiv 1$). Here $\mathcal{S}_k$ is the symmetric group in $k$ letters, and $\epsilon(\sigma)$ for $\sigma \in \mathcal{S}_k$ is the sign of this permutation. The function $T^{\sigma}:G^k \rightarrow A$ is defined as follows. Let $x=(x_1,\cdots,x_k) \in G^k$. If $\sigma$ is a cycle, say $\sigma = (j_1,\cdots,j_m)$, then set $T^{\sigma}(x) = T(x_{j_1} \cdots x_{j_m})$ (well-defined by the central property). In general if $\sigma=\sigma_1 \cdots \sigma_r$ is the cycle decomposition of $\sigma$ (including the cycles with $1$ element) then we set $T(x) = \prod_{i=1}^r T^{\sigma_i}(x)$ (in particular $S_1(T)=T$). 

Then the identities defining $T$ to be a pseudo-character of dimension $m$ are:
\bigskip

1. $T(xy)=T(yx)$ for all $x,y \in G$.

\bigskip

2. $T(e)=m$, where $e$ is the identity element of $G$.

\bigskip

3. $S_{m+1}(T)(x) \equiv 0$ for all $x \in G^{m+1}$.

\bigskip

We have the following ({\it c.f. loc. cit.}): if $\rho:G \rightarrow \GL_m(A)$ is a representation, then 
\begin{eqnarray}
\frac{1}{k!}S_k(T)(g,\cdots ,g) = \tr(\Lambda^k \rho)(g) \mbox{ for all } g \in G.
\end{eqnarray}

Recall the elements $T_v,S_v,R_v \in \mathcal{H}_v^{\sph}$ as in equation (4.1).

\begin{proposition}
With $T$ the pseudo-character as in theorem 4.4. For any finite prime $v \notin S \cup S_p$ we have the identities:
\begin{eqnarray*}
 \beta(T_v)  & =& S_1(T)(\Frob_v) \\
           \beta          (Nv R_v + 2 Nv^3 S_v   )   &=&  \frac{1}{2}S_2(T)(\Frob_v,\Frob_v)       \\
   \beta     (   Nv^3 T_v S_v          )        &=&  \frac{1}{3!}S_3(T)(\Frob_v,\Frob_v,\Frob_v)   \\
   \beta(   Nv^6 S_v^2        )               &=&   \frac{1}{4!}S_4(T)(\Frob_v,\Frob_v,\Frob_v,\Frob_v).    \\
\end{eqnarray*}
\end{proposition}
\begin{proof}
Same kind of argument used to prove theorem 4.4, by using equation (4.4) and (4.7).
\end{proof}

\begin{theorem}
As before hypothesis 4.2 is in force. There is a continuous semi-simple $p$-adic Galois representation
\begin{eqnarray}
R_p:G_{F,S,p} \rightarrow \GL_4(\barQp)
\end{eqnarray}
such that for any prime $v \notin S \cup S_p$, we have 
\begin{eqnarray}
P_v(X) = \iota_p^{-1} Q_v(X)
\end{eqnarray}
where $Q_v(X)$ is as in part 3 of theorem 4.3 above with $t=0$, i.e. the inverse characteristic polynomial of the unramified  Weil-Deligne representation $k_* \rec_v(\Pi_v \otimes |c|_v^{-3/2})$ evaluated at $\Frob_v$, and $P_v(X)$ is the inverse characteristic polynomial of $R_p(\Frob_v)$.
\end{theorem}
\begin{proof}
Denote by $T_0 :=ev_{0,*}T$ the $L$-valued pseudo-character of $G_{F,S,p}$ of dimension 4 obtained by composing $T$ with $ev_0$. By Taylor's theorem (theorem 1 of \cite{T1}) there exists a unique continuous semi-simple representation
\[
R_p:G_{F,S,p} \rightarrow \GL_4(\barQp)
\]
such that $T_0 = \tr R_p$. 

Hence by equation (4.9) for any $g \in G_{F,S,p}$:
\begin{eqnarray}
& & \sum_{k=0}^4 \frac{(-1)^k}{k!} S_k(T_0)(g, \cdots, g) X^k \\ &=& \sum_{k=0}^4 (-1)^k  \tr(\Lambda^k R_p)(g) X^k=\det(I_4 - R_p(g) X  ). \nonumber
\end{eqnarray}
In particular this holds for $g = \Frob_v$. We thus obtain the result by using proposition 4.5 and equation (4.4) (with $t=0$). 
\end{proof}

As usual Cebotarev density and the theorem of Brauer-Nesbitt implies that equation (4.9) characterizes the continuous semi-simple $p$-adic Galois representation $R_p$ associated to $\Pi$. 

\bigskip

We now prove a local-global compatibility statement, up to semi-simplification and at primes away from $p$, for the Galois representation $R_p$ asociated to $\Pi$, which can be regarded as generalization of (4.9) to primes where $\Pi$ is not spherical. The argument is a generalization of Chenevier \cite{Ch} to the setting of $\GSp_4$, which entails replacing the spherical Hecke algebra by the centre of a Bernstein component. Again hypothesis 4.2 is in force.

Thus let $v$ be any finite prime of $F$ not dividing $p$. Let $\mathcal{B}_v$ be the Bernstein component of $G(F_v)= \GSp_4(F_v)$ to which the irreducible admissible representation $\Pi_v$ belongs. The Bernstein component is indexed by $(M,\sigma)$, where $M$ is a Levi subgroup of $G$, and $\sigma$ is a supercuspidal representation of $M(F_v)$ (which we take as a base point of $\mathcal{B}_v$). We can choose a finite extension $E$ of $\mathbf{Q}$ such that the Bernstein component $\mathcal{B}_v$ is defined over $E$. 

Denote by $\mathfrak{z}_v=E[\mathcal{B}_v]$ the affine coordinate ring of $\mathcal{B}_v$ over $E$, the Bernstein centre of $\mathcal{B}_v$. We can identify $\mathfrak{z}_v$ as a commutative subalgebra of $\mathcal{H}_v \otimes_{\mathbf{Q}} E$. For any irreducible $\pi$ belonging to $\mathcal{B}_v$, we denote by $z(\pi)$ the character of $\mathfrak{z}_v$ given by the action of $\mathfrak{z}_v$ on $\pi$ (more precisely on $e^{\mathcal{B}_v} \pi$ where $e^{\mathcal{B}_v} \in \mathcal{H}_v \otimes_{\mathbf{Q}} E$ is the idempotent corresponding to $\mathcal{B}_v$). 

We need the following refinement of theorem 4.3:
\begin{proposition} (\cite{MT})
In the setting of theorem 4.3, suppose that $\Pi_v$ belongs to the Bernstein component $\mathcal{B}_v$. Then we can take $D$, and $Z \subset D(L)$ and $L/\mathbf{Q}_p$ such that $E \hookrightarrow L$, and such that for any $t \in Z$, we can take $\Pi(t)$ such that $\Pi(t)_v$ belongs to the Bernstein component $\mathcal{B}_v$. Furthermore, the ring homomorphism (4.2) (when restricted to $\mathcal{H}^{\sph,S \cup \{ v \},p}$) extends to a ring homomorphism:
\[
\beta: \mathcal{H}^{\sph,S \cup \{v\},p} \otimes_{\mathbf{Z}} \mathfrak{z}_v \rightarrow \mathcal{O}(D)
\]
such that for any $t \in Z$, the map:
\[
ev_{t,*} (\beta|_{\mathfrak{z}_v}) : \mathfrak{z}_v \rightarrow L
\]
is the character $z(\Pi(t))$ giving the action of $\mathfrak{z}_v$ on $\Pi(t)$.
\end{proposition}

We first prove:
\begin{theorem}
There exists a pseudo-character (of dimension $4$) 
\[
T^{\mathcal{B}_v} : W_{F_v} \rightarrow \mathfrak{z}_v = E[\mathcal{B}_v]
\]
such that for any irreducible $\pi$ belonging to $\mathcal{B}_v$, we have
\begin{eqnarray}
ev_{z(\pi),*}T^{\mathcal{B}_v}  &=&  \tr (  k_* \rec_v(  \pi  )^{ss} ) 
\end{eqnarray}
here $ev_{z(\pi)}$ is the evaluation map on $\mathfrak{z}_v=E[\mathcal{B}_v]$ at $z(\pi)$.
\end{theorem}
\begin{proof}
We need to use the description by Gan-Takeda \cite{GT1,GT2} of the local Langlands correspondence $\rec_v^{\GT}$ for $\GSp_4(F_v)$, which by the result of Chan-Gan \cite{CG} is the same as the local Langlands correspondence $\rec_v$ for $\GSp_4(F_v)$ constructed by Arthur as in section 2.2. For clarity we will use $\rec_v^{\GT}$ to emphasize the dependence on \cite{GT1,GT2} in this proof. 

First consider the simple case where $M=G$, thus $\sigma$ is a supercuspidal representation of $G(F_v)$. Then we have $\mathcal{B}_v = \mathbf{G}_m$ (defined over $E$), identified as the group of unramified characters of $F_v^{\times}$, hence of $G(F_v)$ via the similitude character $c$. Any $\pi$ belonging to $\mathcal{B}_v$ is of the form $\sigma \otimes \xi \circ c$ (for an unramified character $\xi$ of $F_v^{\times}$). Hence we can define for any $g \in W_{F_v}$
\begin{eqnarray}
T^{\mathcal{B}_v}(g)(\xi) & :=& \tr (k_{*}\rec^{\GT}_v(\sigma  )^{ss}(g)) \cdot \xi(\art^{-1}_v(g)) \\
&=& \tr(     k_{*} \rec^{\GT}_v(  \sigma \otimes \xi \circ c )(g)    ) \nonumber
\end{eqnarray}
(up to replacing $\sigma$ by an unramified twist and enlarging $E$ we can assume without loss of generality that $k_{*}\rec_v (\sigma)$ is defined over $E$). Then $T^{\mathcal{B}_v}(g) \in E[\mathcal{B}_v]$ for any $g \in W_{F_v}$, and $T^{\mathcal{B}_v}$ satisfies equation (4.13).

We next consider the case where $M=M_Q$ is the Levi component of the Klingen parabolic subgroup $Q$ of $G$. In this case we have $M_Q \cong \GL_{1/F} \times \GL_{2/F}$, where
\begin{eqnarray*}
(a,A) \in \GL_{1/F} \times \GL_{2/F} \mapsto \left( \begin{array}{rrr} a  & &  \\ & A    & \\ & &   a^{-1}\det A \\ \end{array} \right) \in M_Q.
\end{eqnarray*}
Take $Y=\mathbf{G}_m \times \mathbf{G}_m$, identified as the group of unramified characters of $M(F_v)$ via the identity on the factor $\GL_1(F_v)$ and the determinant on the factor $\GL_2(F_v)$. We can write the supercuspidal representation $\sigma$ on $M_Q(F_v)$ as $\chi \boxtimes \tau $, where $\chi$ is a character of $F_v^{\times}$, and $\tau$ is a supercuspidal representation on $\GL_2(F_v)$. The Weyl group $\Delta \subset Y \rtimes W_{M_Q}$ of this Bernstein component (here $W_{M_Q}$ is the Weyl group of $M_Q$ relative to $G$, which is of order two) can be defined as follows: we have the following action of $W_{M_Q}$ on $Y$, where the non-trivial element $w_Q$ of $W_{M_Q}$ acts via
\[
w_Q \cdot (\xi_1,\xi_2) = (\xi_1^{-1},\xi_1 \xi_2) \mbox{ for } (\xi_1,\xi_2) \in Y
\]
and hence we have the following action of $Y \rtimes W_{M_Q}$ on $Y$ by extending the multiplication action of $Y$ on itself, i.e.:
\[
((\nu_1,\nu_2),w_Q) \cdot (\xi_1,\xi_2) := (\nu_1 \xi_1^{-1},\nu_2 \xi_1 \xi_2).
\]
Then $\Delta$ is the subgroup of elements $((\nu_1,\nu_2),t)$ of $Y \rtimes W_{M_Q}$ such that 
\begin{eqnarray*}
\tau \otimes \nu_1 \cong \tau, \,\ \nu_2 = 1 
\end{eqnarray*}
when $t = e$, and
\begin{eqnarray*}
\tau \otimes \chi \cong \tau \otimes v_2, \,\ \chi^2 \nu_1 \equiv 1
\end{eqnarray*}
when $t=w_Q$. Put $\widetilde{\mathcal{B}}_v = \Delta \backslash Y.$ Then from the theory of Bernstein components we have a surjective map $\mathcal{B}_v \rightarrow \widetilde{\mathcal{B}}_v$ with finite fibres. The Bernstein centre $E[\mathcal{B}_v]$ is then given by the affine coordinate ring $E[\widetilde{\mathcal{B}}_v]=E[Y]^{\Delta}$ of $\widetilde{\mathcal{B}}_v$.

Now from Table 1 of section 14 of \cite{GT1} (the classification in this case follows from the work of Sally-Tadic \cite{ST}), we see that if $\pi$ belongs to this Bernstein component then there are three possibilities (in the notation of {\it loc. cit.}), here below $(\xi_1,\xi_2) \in Y$:

\bigskip

1. $\pi =J_Q( \chi \xi_1, \tau \otimes \xi_2) $ with $\chi \xi_1 \neq 1$.

\bigskip

2. $\pi$ is an irreducible subrepresentation of $I_Q(1,\tau \otimes \xi_2)$ (i.e. $\chi \xi_1$ is trivial; in this case $I_Q(1,\tau \otimes \xi_2)$ is semi-simple).

\bigskip

3. $\pi = \St(\chi^{\prime}_0,\tau^{\prime}_0)$, where $\tau^{\prime}_0 = \tau \otimes \xi_2 | \det|_v^{1/2}$, while $\chi^{\prime}_0:=\chi \xi_1 \cdot |\cdot|_v^{-1}$ is non-trivial and satisfies $\tau^{\prime}_0 \otimes \chi^{\prime}_0 \cong \tau^{\prime}_0$ and $(\chi^{\prime}_0)^2=1$ (see Lemma 5.1 of \cite{GT1}).

\bigskip

One sees from {\it loc. cit.}, together with the recipe for $L$-parameter in section 13 of \cite{GT1}, that in all the above cases, the semi-simple part of the $L$-parameter $k_{*} \rec^{\GT}_v( \pi)$ has the same semi-simple part as: 
\begin{eqnarray}
\mathcal{L}_v( \tau \otimes \chi \xi_1  \xi_2) \oplus \mathcal{L}_v(\tau \otimes  \xi_2) .
\end{eqnarray}

In fact in case 1 and 2, the $L$-parameter is given by (4.13), while in case 3, the $L$-parameter is given by $\mathcal{L}_v( \tau^{\prime}_0 ) \boxtimes v(2)$, where $v(2)$ is the two-dimensional representation algebraic representation for $\SL_2(\mathbf{C})$ (here this is the $\SL_2(\mathbf{C})$ of $L_{F_v} = W_{F_v} \times \SL_2(\mathbf{C})$). One sees immediately using (2.1) and the condition in case 3 that it has the same semi-simple part as the parameter (4.13).

Hence in this case we can define $T^{\mathcal{B}_v}(g)$ for any $g \in W_{F_v}$ by the rule: for any $(\xi_1,\xi_2) \in Y$:
\[
T^{\mathcal{B}_v}(g)(\xi_1,\xi_2) = \tr \mathcal{L}_v( \tau \otimes \chi )(g) \cdot \xi_1 \xi_2(\art_v^{-1}(g) ) + \tr \mathcal{L}_v(\tau )(g) \cdot \xi_2(\art_v^{-1}(g)) 
\]
(again by replacing $\tau,\chi$ be unramified twist and enlarging $E$ we may assume that $\mathcal{L}_v(\tau)$ and $\mathcal{L}_v(\tau \otimes \chi)$ are defined over $E$) which is easily seen to be (as a function on $Y$) to be in $E[Y]^{\Delta} = E[\mathcal{B}_v]$. It is a pseudo-character because for any $z \in \mathcal{B}_v$, we have $ev_{z,*} T^{\mathcal{B}_v}$ is the trace of a representation. 

Next we consider the case where $M=M_P$ is the Levi component of the Siegel parabolic subgroup.
One has $M_P \cong \GL_{2/F} \times \GL_{1/F}$, where
\begin{eqnarray*}
(A,a) \in \GL_{2/F} \times \GL_{1/F} \mapsto \left( \begin{array}{rr} A  &  \\ & a \cdot (A^t)^{-1}    \\ \end{array} \right) \in M_P.
\end{eqnarray*}

Again take $Y=\mathbf{G}_m \times \mathbf{G}_m$, identified as the group of unramified characters of $M_P(F_v)$ via the determinant on the factor $\GL_2(F_v)$ and the identity on the factor $\GL_1(F_v)$. We can write the supercuspidal representation $\sigma$ on $M_P(F_v)$ as $ \tau \boxtimes \chi $, where $\tau$ is a supercuspidal representation on $\GL_2(F_v)$ and $\chi$ is a character of $F_v^{\times}$. The Weyl group $\Delta \subset Y \rtimes W_{M_P}$ of this Bernstein component (here $W_{M_P}$ is the Weyl group of $M_P$ relative to $G$ which is again of order two) can be defined similar to the Klingen case: define the action of $W_{M_P}$ on $Y$ where the non-trivial element $w_P$ of $W_{M_P}$ acts via
\[
w_P \cdot (\xi_1,\xi_2) = (\xi_1^{-1},\xi_1^2 \xi_2) \mbox{ for } (\xi_1,\xi_2) \in Y
\]
and as before extend this to an action of $Y \rtimes W_{M_P}$ on $Y$.

Then $\Delta$ is the subgroup of elements $((\nu_1,\nu_2),t)$ of $Y \rtimes W_{M_P}$ such that 
\begin{eqnarray*}
\tau \otimes \nu_1 \cong \tau , \,\   \nu_2 =1
\end{eqnarray*}
when $t = e$, and
\begin{eqnarray*}
\nu_2 = \omega_{\tau}, \,\  \tau \otimes \omega_{\tau} \nu_1 \cong \tau 
\end{eqnarray*}
(which implies in particular that $(\omega_{\tau} \nu_1)^2=1$) when $t=w_P$. We have a surjective map from $\mathcal{B}_v$ to $\widetilde{\mathcal{B}}_v := \Delta \backslash Y$ with finite fibres, and as in the previous case $E[\mathcal{B}_v]=E[\widetilde{\mathcal{B}}_v]=E[Y]^{\Delta}$.

From Table 1 of section 14 of \cite{GT1} (again follows from the work of \cite{ST}), we see that if $\pi$ belongs to this Bernstein component then there are two possibilities (in the notation of {\it loc. cit.}), here again $(\xi_1,\xi_2) \in Y$:

\bigskip

1. $\pi =J_P( \tau \otimes \xi_1, \chi  \xi_2) $.

\bigskip

2. $\pi = \St(\tau^{\prime}_0, \chi_0^{\prime})$, where $\tau^{\prime}_0 = \tau \otimes \xi_1 | \det|_v^{-1/2}$, while $\chi^{\prime}_0:=\chi \xi_2 \cdot |\cdot|_v^{1/2}$, such that $\tau_0^{\prime}$ has trivial central character (see Lemma 5.2 of \cite{GT1}).

\bigskip

In case 1 the $L$-parameter $k_* \rec_v^{\GT}(\pi)$ is given by {\it loc. cit.}
\begin{eqnarray}
(  \omega_{\tau} \chi  \xi_1^2 \xi_2 )\circ \art_v^{-1} \oplus \mathcal{L}_v(\tau \otimes \chi \xi_1 \xi_2) \oplus (\chi \xi_2) \circ \art_v^{-1}
\end{eqnarray}
while in case 2 the $L$-parameter is given by
\[
\mathcal{L}_v(\tau_0^{\prime} \otimes \chi_0^{\prime} ) \oplus \chi_0^{\prime} \boxtimes v(2)
\]
which again has the same semi-simple part as (4.14).

Hence in this case we can define $T^{\mathcal{B}_v}(g)$ for any $g \in W_{F_v}$ by the rule: for any $(\xi_1,\xi_2) \in Y$:
\begin{eqnarray}
& & T^{\mathcal{B}_v}(g)(\xi_1,\xi_2) \\ &=& \omega_{\tau} \chi( \art_v^{-1}(g)) \cdot \xi_1^2 \xi_2 (\art_v^{-1}(g))+ \tr \mathcal{L}_v(\tau \otimes \chi)(g) \cdot \xi_1 \xi_2 (\art_v^{-1}(g)) \nonumber \\ & & \,\ \,\ + \chi(\art_v^{-1}(g)) \cdot \xi_2(\art_v^{-1}(g)) \nonumber
\end{eqnarray}
which is in $E[Y]^{\Delta} = E[\mathcal{B}_v]$ as is easily seen from the rule giving the action of $\Delta$. Again this defines a pseudo-character.

Finally we consider the case where $M$ is the maximal torus $T$ of $G$ (i.e. the parabolic in this case is the Borel subgroup $B$ of $G$). We have $T \cong \GL_{1/F} \times \GL_{1/F} \times \GL_{1/F}$, where
\begin{eqnarray*}
 (a,b; t) &\in &    \GL_{1/F} \times \GL_{1/F} \times \GL_{1/F} \\
& \mapsto & \left( \begin{array}{rrrr} a  & & & \\ & b &   & \\ & & t b^{-1} & \\ & & & t a^{-1} \end{array} \right) \in T.
\end{eqnarray*}
The supercuspidal representation $\sigma$ on $M(F_v)=T(F_v)$ is thus given by a character $\chi_1 \boxtimes \chi_2 \boxtimes \chi$, where $\chi_1,\chi_2$ and $\chi$ are characters on $\GL_1(F_v)$.

In this case one has $Y=\mathbf{G}_m \times \mathbf{G}_m \times \mathbf{G}_m$, and we can define similarly define the Weyl group $\Delta \subset Y \rtimes W$ of this Bernstein component (here $W=W_T$ is the usual Weyl group of $G$). 

From Table 1 of section 14 of \cite{GT1} we see that if $\pi$ belongs to this Bernstein component, then there are 5 possibilities (here below $(\xi_1,\xi_2;\xi) \in Y$):

\bigskip

1.$\pi = J_B(\chi_1 \xi_1,\chi_2 \xi_2;\chi \xi)$.

\bigskip

2. $\pi = J_Q(\widetilde{\chi},\tau)$, with $\widetilde{\chi} \neq 1$, and $\tau$ is a twisted Steinberg representation of $\GL_2(F_v)$ of the form $\tau = \St_{\GL_2} \otimes \chi^0$. Here $\widetilde{\chi}$ and $\chi^0$ are characters of $\GL_1(F_v)$, satisfying: $\chi_1 \xi_1 = \widetilde{\chi}$, $\chi_2 \xi_2 = |\cdot|_v^{-1}$, $\chi \xi = \chi^0 |\cdot|_v^{1/2}$.

\bigskip

3. $\pi$ is an irreducible subrepresentation of the semi-simple induced representation $I_Q(1,\tau)$, with $\tau$ twisted Steinberg representation $\tau = \St_{\GL_2} \otimes \chi^0$, such that $\chi_1 \xi_1 = 1$, $\chi_2 \xi_2 = |\cdot|_v^{-1}$, $\chi \xi = \chi^0 |\cdot|_v^{1/2}$.

\bigskip

4. $\pi = J_P(\tau,\chi \xi)$ with $\tau$ twisted Steinberg representation of $\GL_2(F_v)$ given by $\tau= \St_{\GL_2} \otimes \chi^0$. Here $\chi^0$ is a character of $\GL_1(F_v)$ such that $\chi_1 \xi_1 = \chi^0 |\cdot|_v^{-1/2} $, and $\chi_2 \xi_2 = \chi^0 |\cdot|_v^{1/2}$.

\bigskip

5. $\pi = \St_{PGSp_4} \otimes \chi$, twisted Steinberg representation of $\GSp_4(F_v)$ ({\it c.f.} Lemma 5.1 of \cite{GT1}). 

\bigskip

In each of the above cases one sees that $\pi$ belongs to the Bernstein component $(T,\chi_1 \boxtimes \chi_2 \boxtimes \chi)$ using induction by stages. Furthermore  from {\it loc. cit.} one checks directly that in each of these cases, the semi-simple part of the $L$-parameter is given by
\[
((1 \oplus \chi_1 \xi_1 \oplus \chi_2 \xi_2 \oplus \chi_1 \chi_2 \xi_1 \xi_2) \cdot \chi \xi )\circ \art_v^{-1}.
\]
(In case 5 above one has $\chi_1 \xi_1 = |\cdot|_v^2$, $ \chi_2 \xi_2=|\cdot|_v$, $\xi=|\cdot|_v^{-3/2}$.)

Hence we can take, for $g \in W_{F_v}$, and $(\xi_1,\xi_2;\xi) \in Y$:
\begin{eqnarray}
& & T^{\mathcal{B}_v}(g)(\xi_1,\xi_2;\xi) \\
& =& ((1  + \chi_1 \xi_1 +\chi_2 \xi_2 + \chi_1 \chi_2 \xi_1 \xi_2) \cdot \chi \xi )\circ \art_v^{-1}. \nonumber
\end{eqnarray}

We have $T^{\mathcal{B}_v}(g) \in E[Y]$, and one needs to show that it lies in $E[\mathcal{B}_v]$, i.e. that (4.16) is invariant under the action of $\Delta$. Since this involves no new ideas, we leave the verification to the reader.
\end{proof}

\begin{rem}
\end{rem}
The proof of theorem 4.8, together with proposition 3.4 and proposition 4.10, is the only place where we are using the explicit description of $L$-parameters for $\GSp_4$. It would be of interest to have proof of these results without case by case enumeration.

\bigskip

For later use we also draw one corollary from the work of Gan-Takeda:

\begin{proposition}
Suppose that $\phi$ is an $L$-parameter for $\GSp_4(F_v)$ whose semi-simple part is unramified. Then there is a representation in the $L$-packet classified by $\phi$ that has Iwahori fixed vectors.
\end{proposition}
\begin{proof}
This again follows from the description of the $L$-parameters for $\GSp_4$ given in Table 1 of section 14 of \cite{GT1}, using the fact that $\pi$ has Iwahori fixed vectors if and only if it is a subquotient of the parabolic induction from the Borel subgroup $B$ with unramified induction data. Furthermore, we see from ${\it loc. cit.}$ that in fact if the semi-simple part of $\phi$ is unramified, then all the representations in the packet of $\phi$ has Iwahori-fixed vectors, except in the case where $\phi$ has the form
\[
\phi = \lambda_1 \cdot v(2) \oplus \lambda_2 \cdot v(2)
\]
where $\lambda_1,\lambda_2$ are {\it distinct} unramified characters of $\GL_1(F_v)$ such that $\lambda_1^2=\lambda_2^2$, in which case the $L$-packet corresponding to $\phi$ contains a supercuspidal non-generic representation, while the generic representation in the packet is given by (notation as in the proof of theorem 4.8) $J_Q(\lambda_1\lambda_2^{-1},\St_{\GL_2}) \otimes \lambda_2$, which has Iwahori fixed vectors. 
\end{proof}

\bigskip

With proposition 4.7 and theorem 4.8, we can now prove:

\begin{theorem}
Let $R_p$ be the Galois representation associated to $\Pi$ as in theorem 4.6. Assume hypothesis 4.2. Then for any finite prime $v$ of $F$ not dividing $p$, we have local-global compatibility up to semi-simplification:
\begin{eqnarray}
\iota_p \WD (R_p|_{G_{F_v}})^{ss} \cong k_{*} \rec_v(\Pi_v \otimes |c|_v^{-3/2} )^{ss}.
\end{eqnarray}
Furthermore, if we identify the set of embeddings of $F$ into $\barQp$ with the set of archimedean places of $F$ via $\iota_p: \barQp \cong \mathbf{C}$, then the Hodge-Tate-Sen weights of $R_p$ at the embedding corresponding to $v | \infty$ is given by the same formula as (3.6), i.e. 
\[
\{    \delta_v,\delta_v + m_{2,v},\delta_v + m_{1,v},\delta_v + m_{1,v} + m_{2,v}            \}
\]
with 
\[
\delta_v = \frac{1}{2}(w+3 -m_{1,v}-m_{2,v}).
\]
\end{theorem}
\begin{proof}
We follow the arguments of \cite{Ch}. As above denote by $\mathcal{B}_v$ the Bernstein component of $\GSp_4(F_v)$ to which $\Pi_v$ belongs, and assume that $\mathcal{B}_v$ is defined over some finite extension $E$ of $\mathbf{Q}$. And we are in the setting where theorem 4.3 and proposition 4.7 apply. In particular the finite extension $L/\mathbf{Q}_p$ is chosen large enough so that $E \hookrightarrow L$. Denote by
\[
\widetilde{T}^{\mathcal{B}_v}: W_{F_v} \rightarrow \mathcal{O}(D)
\] 
the pseudo-character defined by the rule: for any $g \in W_{F_v}$: 
\[
\widetilde{T}^{\mathcal{B}_v}(g) := \beta(  T^{\mathcal{B}_v}(g)   ) \cdot |g|_v^{-3/2}
\]
where $T^{\mathcal{B}_v}: W_{F_v} \rightarrow \mathfrak{z}_v$ is the pseudo-character as in theorem 4.8. Then by construction, for any $t \in Z \subset D(L)$, we have
\[
ev_{t,*} \widetilde{T}^{\mathcal{B}_v} = \iota_p^{-1} \tr(k_* \rec_v(  \Pi(t)_v \otimes |c|_v^{-3/2}  )^{ss}).
\]

On the other hand recall the psuedocharacter $T:G_{F,S,p} \rightarrow \mathcal{O}(D)$ of theorem 4.4. We continue to denote by $T$ the pseudo-character obtained by pre-composing $T$ with the map $G_F \rightarrow G_{F,S,p}$. By Lemma 7.8.11 of \cite{BC1}, there exists a reduced quasi-compact separated rigid analytic space $\mathcal{Y}$ and a morphism $f :\mathcal{Y} \rightarrow D$ satisfying the following properties:

\bigskip

1. $f$ is proper and surjective.

\bigskip

2. There exists a admissible open $U \subset D$ that is Zariski dense and such that $f^{-1}(U) \rightarrow U$ is finite etale.

\bigskip

3. There exists a locally free $\mathcal{O}_{\mathcal{Y}}$-module $M$ of rank four, with a continuous linear action of $G_{F}$, whose trace is given by the pull-back of $T$ by $f$. For any $y \in \mathcal{Y}$, we denote by $M_y$ the evaluation of $M$ at $y$ (a four dimensional representation of $G_F$ over the residue field $L(y)$ of $\mathcal{O}_{\mathcal{Y}}$ at $y$).

\bigskip

4. For any $y \in f^{-1}(U)$, the representation $M_y$ is semi-simple.

\bigskip

In particular $Z^{\prime} : =f^{-1}(U \cap (Z \backslash \{0\}))$ is Zariski dense in $\mathcal{Y}$. We also see by condtion 3 and 4 for any $z \in Z^{\prime}$ with $t=f(z)$, we have $M_z \cong R_{t,p}$ as representation of $G_F$ (recall that $R_{t,p}$ is the Galois representation associated to $\Pi(t)$), and if $y_0 \in \mathcal{Y}$ with $f(y)= 0 \in Z \subset D$, then $M_{y_0}^{ss} \cong R_p$. 

Now by Lemma 7.8.14 of \cite{BC1}, the action of $G_{F_v} \hookrightarrow G_F$ on $M$ admits a Weil-Deligne representation $\WD(M|_{G_{F_v}})$. By functoriality of the construction of Weil-Deligne representation we have $\WD(M|_{G_{F_v}})_y \cong \WD(M_y|_{G_{F_v}})$ for any $y \in \mathcal{Y}$. 

Now denote $T^{\WD_v}:W_{F_v} \rightarrow \mathcal{O}(\mathcal{Y})$ the $\mathcal{O}(\mathcal{Y})$-valued pseudo-character given by $T^{\WD_v} = \tr \WD(M|_{G_{F_v}})^{ss}$. We claim that 
\[
 T^{\WD_v} = (f^{\#})_{*} \widetilde{T}^{\mathcal{B}_v}
\]
where $f^{\#} : \mathcal{O}(D)\rightarrow \mathcal{O}(\mathcal{Y})$ is the map of induced by $f$. Indeed, by Zariski density it suffices to prove this equality evaluated at any point $z \in Z^{\prime}$. But if $z \in Z^{\prime}$ with $t=f(z) \in Z \backslash \{0\}$, then 
\begin{eqnarray*}
ev_{z,*} T^{\WD_v} = \tr( \WD(M_z |_{G_{F_v}})^{ss})=\tr(\WD(R_{t,p}|_{G_{F_v}})^{ss}).
\end{eqnarray*}
Since $\Pi(t)$ is cohomological, we have by theorem 3.5 that
\[
\tr(\WD(R_{t,p}|_{G_{F_v}})^{ss}) = \iota_p^{-1} \tr( \rec_v(   \Pi(t)_v \otimes |c|_v^{-3/2}    ) ^{ss}         )
\]
which is equal to $ev_{t,*}  \widetilde{T}^{\mathcal{B}_v}$ by above, which is exactly what we want to prove. 

In particular choose $y_0 \in \mathcal{Y}$ such that $f(y)=0 \in Z \subset D$. Then evaluating the equality $T^{\WD_v} = (f^{\#})_{*} \widetilde{T}^{\mathcal{B}_v}$ at $y_0$ we obtain
\begin{eqnarray*}
& & \tr( \WD(R_p|_{G_{F_v} }  )^{ss}) = \tr( \WD(M_{y_0}^{ss}|_{G_{F_v} }  )^{ss}) \\
& = & \tr( \WD(M_{y_0}|_{G_{F_v} }  )^{ss}) = ev_{y_0,*} T^{\WD_v} = ev_{y_0,*} (f^{\#})_{*} \widetilde{T}^{\mathcal{B}_v} = ev_{0,*} \widetilde{T}^{\mathcal{B}_v} \\ & = & \iota_p^{-1} \tr( k_{*} \rec_v(\Pi_v \otimes |c|_v^{-3/2}  )^{ss}  ).
\end{eqnarray*}
We hence conclude (4.17) by Brauer-Nesbitt.

We now prove the assertion on Hodge-Tate-Sen weights. Thus identify the set of embeddings of $F$ into $\barQp$ with the archimedean places of $F$ via $\iota_p : \barQp \cong \mathbf{C}$. Let $v$ be an archimedean place of $F$, and denote by $\mathfrak{p}$ the prime of $F$ above $p$ that is associated to the embedding to $F$ into $\barQp$ corresponding to $v$. We assume without loss of generality that $L$ is chosen large enough so that $L$ contains all the Galois conjugates of $F_{\mathfrak{p}}$ over $\mathbf{Q}_p$. 

For each $y \in \mathcal{Y}$, denote by $P^{\sen}_v(X,y) \in L(y)[X]$ the Sen polynomial at $v$ (in the indeterminate $X$) of the representation $M_y$ restricted to $G_{F_{\mathfrak{p}}}$, whose roots are the Hodge-Tate-Sen weights of $M_y|_{G_{F_{\mathfrak{p}}}}$ at the embedding $F \hookrightarrow \barQp$ corresponding to $v$. By Sen's theory \cite{Sen}, there exists a polynomial
\[
P_v(X) \in \mathcal{O}(\mathcal{Y})[X]
\]
such that for any $y \in \mathcal{Y}$, we have $ev_{y,*} P_v(X) = P^{\sen}_v(X,y)$. 

On the other hand, consider the polynomial $\widetilde{P}_v(X) \in \mathcal{O}(D)[X]$ defined by:
\[
\widetilde{P}_v(X) :=\prod_{i=1}^4(X-  \kappa_{i,v}(t))
\]
where $\kappa_{i,v}(t) \in \mathcal{O}(D)$ are given by:
\begin{eqnarray*}
& & \{\kappa_{1,v},\kappa_{2,v},\kappa_{3,v},\kappa_{4,v}\} \\ &=& \{\delta_v(t),\delta_v(t) +m_{2,v}(t), \delta_v(t) +m_{1,v}(t) , \delta_v(t) + m_{1,v}(t) + m_{2,v}(t)   \} \\
& & \delta_v(t) = \frac{1}{2}( w(t)+3 -m_{1,v}(t) -m_{2,v}(t)  )
\end{eqnarray*}
with $w(t), m_{1,v}(t), m_{2,v}(t)$ as in the statement of theorem 4.3.

A similar argument as above shows that (using the assertion on Hodge-Tate weights in theorem 3.5 about the representations $R_{t,p}$ for $t \in Z \backslash \{0\}$):
\[
P_v(X) = (f^{\#})_* \widetilde{P}_v(X).
\]
Hence we obtain the result again by evaluating at $y=y_0$.
\end{proof}

We now remove the hypothesis 4.2 in theorem 4.6 and theorem 4.11, using the argument of \cite{ChH}. We need the following result from \cite{ChH} which we also use several times implicitly in section 5.

\begin{proposition} (Lemma 3.2.1 of \cite{ChH})
Let $S$ be a finite set of primes of $F$, and $w \notin S$ another finite prime, and $M/F$ be any finite extension. Let $L_v/F_v$ be a finite Galois extension of $F_v$ for every $v \in S$. Then there exists infinitely many totally real real solvable Galois extension $F^{\prime}/F$ in which $w$ splits completely, with $F^{\prime}$ linearly disjoint from $M$ over $F$, and such that for every prime $v^{\prime}$ of $F^{\prime}$ above a prime $v \in S$, the extension $F^{\prime}_{v^{\prime}}/F_v$ is isomorphic to $L_v/F_v$. Furthermore, there is a constant $\mu(\{L_v\}_{v \in S})$ independent of $w$, such that the degree $[F^{\prime}:F]$ can be assumed to be less than $\mu(\{L_v\}_{v \in S})$.
\end{proposition} 

\begin{proposition}
Let $\Pi$ be a cuspidal automorphic representation of $\GSp_4(\mathbf{A}_F)$ that belongs to a global packet corresponding to a simple generic parameter $\mu$. Assume that $\Pi$ is essentially tempered at all the archimedean places of $F$. Let $S$ be a finite set of primes of $F$, and $w \notin S$ be another finite prime. Let $M/F$ be any finite extension of $F$. Then there is a constant $\mu(\Pi)$ indepedent of $w$, and a totally real solvable Galois extension $F^{\prime}/F$ with $[F^{\prime}:F] \leq \mu(\Pi)$ in which $w$ splits completely, such that the Arthur-Clozel base change $\mu^{\prime}$ of $\mu$ to $\GL_4(\mathbf{A}_{F^{\prime}})$ defines a global simple generic parameter of $\GSp_4(\mathbf{A}_{F^{\prime}})$, and such that there exists a cuspidal automorphic representation $\Pi_{F^{\prime}}$ in the packet defined by $\mu^{\prime}$, with the property that $\Pi_{F^{\prime}}$ has Iwahori-fixed vectors at all primes of $F^{\prime}$ above $p$. If $\Pi$ belongs to the holomorphic limit of discrete series (or holomorphic discrete series) at the archimedean places of $F$ (up to twists), then we may also assume the case for $\Pi_{F^{\prime}}$.
\end{proposition}
\begin{proof}
For each prime $v \in S$, choose a finite extension $L_v/F_v$ such that, if $\phi_v$ is the $L$-parameter classifying $\Pi_v$, then the restriction of the semi-simple part $\phi_v^{ss}$ to $W_{L_v}$ is unramified. We choose the constant $\mu(\Pi) = \mu(\{L_v\}_{v \in S})$ as in proposition 4.12 above, which is independent of $w$. Then by this proposition, we can choose one (in fact infinitely many) finite solvable totally real Galois extension $F^{\prime}$ of $F$, with $[F^{\prime}:F] \leq \mu(\Pi)$ in which $w$ splits completely, such that for any prime $v^{\prime}$ of $F^{\prime}$ above $v \in S$, the extension $F^{\prime}_{v^{\prime}}/F_v$ is isomorphic to $L_v/F_v$. By proposition 2.4 (and remark 2.5), we can choose such an $F^{\prime}$ such that the Arthur-Clozel base change $\mu^{\prime}$ of $\mu$ to $\GL_4(\mathbf{A}_{F^{\prime}})$ (defined by successive application of cyclic base change) is cuspidal and defines a global simple generic parameter of $\GSp_4(\mathbf{A}_{F^{\prime}})$. For any representation $\Pi_{F^{\prime}}$ in this packet defined by $\mu^{\prime}$, we have for any prime $v^{\prime}$ of $F^{\prime}$ above $v \in S$:
\begin{eqnarray}
& & k_{*} \rec_{v^{\prime}} (  (\Pi_{F^{\prime}})_{v^{\prime}}  ) = \mathcal{L}_{v^{\prime}}(\mu^{\prime}_{v^{\prime}}   ) \\ &= &\mathcal{L}_v(\mu_v  )|_{W_{F_{v^{\prime}}^{\prime} }} = \phi_v|_{W_{F_{v^{\prime}}^{\prime} }} = \phi_v|_{W_{L_v}}. \nonumber
\end{eqnarray}
In the above we make use of the fact that Arthur-Clozel base change is compatible with local Langlands correspondence (for $\GL_4$). By construction it follows that the semi-simple part of $\rec_{v^{\prime}} (  (\Pi_{F^{\prime}})_{v^{\prime}}  )$ is unramified, and hence by proposition 4.10 the local $L$-packet containing $(\Pi_{F^{\prime}})_{v^{\prime}}$ contains a representation that has Iwahori fixed vectors. Since the parameter $\mu^{\prime}$ is simple generic, it follows that all the representations in the global packet defined by $\mu^{\prime}$ occur in the discrete spectrum of $\GSp_4(\mathbf{A}_{F^{\prime}})$. 

Now by assumption the archimedean components of $\Pi$ are essentially tempered. This implies that the archimedean components of $\mu$, and hence that of $\mu^{\prime}$, are essentially tempered. In turn this implies that all the representations of $\GSp_4(\mathbf{A}_{F^{\prime}})$ in the global packet defined by $\mu^{\prime}$, have essentially tempered archimedean components, thus by theorem 2.3 they are cuspidal. In particular we can choose a representation that occurs in the cuspidal spectrum and has Iwahori fixed vectors at all primes of $F^{\prime}$ dividing $p$. The remark about limit of holomorphic discrete series (or discrete series) follows from similar reasoning.  
\end{proof}

It follows from Lemma 2 of \cite{So} that the collection $\mathcal{I}$ of extensions $F^{\prime}$ of $F$ as in proposition 4.13 is $S$-general and have uniformly bounded heights ({\it c.f.} p. 649 and 654 of \cite{So} for these notions. We now take $S=S_p$ to be set of primes of $F$ above $p$. This is enough to achieve our goal:

\begin{theorem} 
Let $\Pi$ be a cuspidal automorphic representation on $\GSp_4(\mathbf{A}_F)$ satisfying the conditions in the beginning of section 4, and belongs to a global packet defined by a simple generic parameter. For each prime $p$, there exists an unique continuous semi-simple $p$-adic Galois representation
\[
R_p:G_F \rightarrow \GL_4(\barQp)
\]
which is unramified outside the set of primes of $F$ dividing $p$ and where $\Pi$ is not spherical, and such that for any finite prime $v$ of $F$, we have
\begin{eqnarray}
\iota_p \WD(R_p|_{G_{F_v}})^{ss} \cong k_* \rec_v(  \Pi_v \otimes |c|_v^{-3/2} )^{ss}.
\end{eqnarray}
Furthermore the Hodge-Tate-Sen weights of $R_p$ are given as in theorem 4.11.
\end{theorem} 
\begin{proof}
For each of the extension $K \in \mathcal{I}$ of $F$ above we choose a cuspidal automorphic representation $\Pi_{K}$ on $\GSp_4(\mathbf{A}_{K})$ as in the statement of proposition 4.13. By theorem 4.11 applied to $\Pi_{K}$ (which applies because by construction $\Pi_{K}$ has Iwahori fixed vectors at all primes of $K$ above $p$), we have a continuous $p$-adic Galois representation:
\begin{eqnarray}
R_{K,p}:G_{K} \rightarrow \GL_4(\barQp)
\end{eqnarray}
such that for any prime $v^{\prime}$ of $K$ not dividing $p$, we have
\begin{eqnarray}
\iota_p \WD(R_{K,p}  |_{G_{K_{v^{\prime}}}}   )^{ss} \cong k_{*} \rec_{v^{\prime}}(  (\Pi_{K} )_{v^{\prime}} \otimes |c|^{-3/2}_{v^{\prime}}  )^{ss}.
\end{eqnarray}

By the patching lemma of \cite{So} (Theorem 7 of {\it loc. cit.}), since the collection $\mathcal{I}$ of extensions $K$ of $F$ is $S$-general and has uniformly bounded heights, there exists a unique continuous semi-simple $p$-adic Galois representation $R_p:G_F \rightarrow \GL_4(\barQp)$ satisfying:
\begin{eqnarray}
R_p|_{G_K}\cong R_{K,p}
\end{eqnarray}
for all extension $K \in \mathcal{I}$, {\it provided} the following conditions are satisfied:

\bigskip

1. $R_{K,p}^{\theta} \cong R_{K,p}$ for all $K \in \mathcal{I}$ and all $\theta \in \Gal(K/F)$. 

\bigskip

2. $R_{K_1,p}|_{G_{K^{\prime}}} \cong R_{K_2,p}|_{G_{K^{\prime}}} $ for all $K_1,K_2 \in \mathcal{I}$, and $K^{\prime} := K_1 \cdot K_2$.

\bigskip

Both are standard arguments. As in the proof proposition 4.13 denote by $\mu_K$ the global simple generic parameter classifying $\Pi_K$, which is the base change of the global simple generic parameter $\mu$ classifying $\Pi$. Now we have

\begin{eqnarray*}
\iota_p \WD( R_{K,p}|_{G_{K_w}} )^{ss} &\cong &  k_{*} \rec_w( (\Pi_K^{\theta})_w \otimes |c|_w^{-3/2} )^{ss} \\
& = & \mathcal{L}_w( (\mu_K)_w  \otimes |\det|_w^{-3/2}     )^{ss}.
\end{eqnarray*}
In particular take any finite prime $w$ of $K$ not dividing $p$ such that $\Pi_K$ is spherical. Then $R_{K,p}$ is unramified at $w$. By considering Satake parameters, we see that for any $\theta \in \Gal(K/F)$, we have the isomorphism of unramified Weil-Deligne representation:
\[
\iota_p \WD(R_{K,p}^{\theta}|_{G_{K_w}})^{F-ss} \cong \mathcal{L}_w(  (\mu^{\theta}_K)_w  \otimes |\det|_w^{-3/2} ).
\]
But since $\mu_K$ comes from base change from $\GL_4(\mathbf{A}_F)$, we have $\mu_K^{\theta} = \mu_K$ \cite{AC}. Hence for such $w$ we have
\[
\WD(R_{K,p}^{\theta}|_{G_{K_w}} )^{F-ss} \cong \WD(R_{K,p}|_{G_{K_w}} )^{F-ss}
\]
and we obtain condition 1 by Cebotarev density and Brauer-Nesbitt theorem.

For condition 2, let $w$ be any finite prime of $K^{\prime} =K_1 \cdot K_2$ not dividing $p$, and let $w_i$ be the prime of $K_i$ below $w$ for $i=1,2$, with $v$ being the common prime of $F$ below $w_1,w_2$. Then by the compatibility of base change with the local Langlands correspondence for $\GL_4$ we have:
\begin{eqnarray*}
& & \iota_p \WD(R_{K_i,p} |_{G_{K^{\prime}_w}})^{ss} = \iota_p \WD(R_{K_i,p}|_{G_{K_{i,w_i}}})^{ss}|_{W_{K^{\prime}_w}} \\
&\cong & \mathcal{L}_{w_i}( (\mu_{K_i})_{w_i} \otimes |\det|_{w_i}^{-3/2}       )^{ss}|_{W_{K^{\prime}_w}} \\
& \cong & \mathcal{L}_v(  \mu_v \otimes |\det|_v^{-3/2} )^{ss} |_{W_{K_{i,w_i}}  } |_{W_{K^{\prime}_w}} \\
& \cong &  \mathcal{L}_v(  \mu_v \otimes |\det|_v^{-3/2} )^{ss}  |_{W_{K^{\prime}_w}}
\end{eqnarray*}
which does not depend on $i$. Hence again for such $w$ we have
\[
\WD(R_{K_1,p} |_{G_{K^{\prime}_w}})^{ss}  \cong \WD(R_{K_2,p} |_{G_{K^{\prime}_w}})^{ss} 
\]
so we again conclude by Ceboatrev density and Brauer-Nesbitt.

To deduce (4.19), choose $K \in \mathcal{I}$ as above. Then since $v$ does not divide $p$, we have by construction $v$ splits in $K$. Let $v^{\prime}$ be a prime of $K$ above $v$. Then since $R_{K,p} \cong R_{p}|_{G_K}$, we have by (4.21):

\begin{eqnarray*}
& & \iota_p \WD(R_{p}|_{G_{F_v}})^{ss} \cong \iota_p \WD(  R_{K,p}|_{G_{K_{v^{\prime}}}}      )^{ss} \\
& \cong & k_{*} \rec_{v^{\prime}}( (\Pi_K)_{v^{\prime}} \otimes |c|_{v^{\prime}}^{-3/2}  )^{ss} = \mathcal{L}_{v^{\prime}}( ( \mu_K)_{v^{\prime}} \otimes |\det|^{-3/2}_{v^{\prime}}    )^{ss}\\
&\cong & \mathcal{L}_v(   \mu_v \otimes |\det|_v^{-3/2}      )^{ss} =k_* \rec_v( \Pi_v \otimes |c|_v^{-3/2} )^{ss}.
\end{eqnarray*}

Finally for the assertion on Hodge-Tate-Sen weights, again choose any $K \in \mathcal{I}$ as above. Then $R_p|_{G_K} \cong R_{p,K}$, and by theorem 4.11, $R_{p,K}$ has Hodge-Tate-Sen weights given as in the statement of theorem 4.11. By ``invariance of Hodge-Tate-Sen weights under finite extension" (Remarque on p.31 of \cite{F}), we conclude the same for $R_p$.

\end{proof}

\begin{rem}
\end{rem}
For our applications in section 5, we will only need the case where $\Pi$ is classfied by a global simple generic parameter. The argument of section 3 shows that if the parameter classifying $\Pi$ is not simple generic, then it reduces to the properties of Galois representations associated to Hilbert modular forms. 

\bigskip

To end this section we state the following result, that can be proved exactly as in \cite{J}, using the generalization by K.Nakamura \cite{N} and F.C. Tan \cite{Tan} on Kisin's results on analytic continuation of crystalline periods and combining with base change arguments.

\begin{proposition}
Let $\Pi$ be as before satisfying the conditions in the beginning of section 4. Assume that $\Pi$ is classified by a simple generic parameter. Suppose that $\Pi$ is spherical at a prime $v|p$. Denote by $Q_{\Pi,v}(X)$ the inverse characteristic polynomial of the geometric Frobenius on the unramified Weil-Deligne representation $k_{*} \rec_v(\Pi_v \otimes |c|_v^{-3/2})$. Assume that the (inverse) roots of $Q_{\Pi,v}(X)$ are all distinct. Then $R_p$ is crystalline at $v$. 
\end{proposition}

\section{Main Theorems}

We now return to the setting of the introduction, where we denote by $\pi$ a cuspidal automorphic representation on $\GL_2(\mathbf{A}_E)$, with $E$ a $\CM$ extension of the totally real field $F$. Denote by $\omega$ the central character of $\pi$. We make the following hypotheses on $\pi$: there is an integer $\mu_0$ such that for any archimedean place $w$ of $E$, the local component at $w$ of $\omega$ is given by $z \mapsto |z|^{-\mu_0}_{\mathbf{C}} = (z \overline{z})^{-\mu_0}$, and for such a place $w$, the local component $\pi_w$ is an irreducible admissible representation of $\GL_2(\mathbf{C})$ corresponding to the archimedean $L$-parameter $\phi_{\mu_0,n_w}$ of $L_{E_w}$, with $n_w \geq 1$, and such that $n_w \equiv \mu_0 +1 \mod{2} $ for all $w | \infty$. We also denote this value as $n_v$ where $v$ is the place of $F$ below $w$.

In this section we prove theorem 1.1 and 1.2 of the Introduction. Thus for the rest of the paper we assume that the central character $\omega$ satisfies the condition (\textbf{Char}) as in the Introduction section.

First we notice that there are two cases for which theorem 1.1 and 1.2 follows readily:

\bigskip

\noindent 1. $\pi \otimes \delta \cong \pi$ for a non-trivial quadratic idele class character $\delta$ of $\mathbf{A}_E^{\times}$.

\bigskip

\noindent 2. $(\pi \otimes \beta)^{\tau} \cong \pi \otimes \beta$ for an algebraic idele class character $\beta$ of $\mathbf{A}_E^{\times}$.

\bigskip

If case 1 occurs then $\pi = \AI^{E}_{L} (\chi \otimes |\det|_{\mathbf{A}_L}^{1/2})$ where $L$ is the quadratic extenson of $E$ cut out by $\delta$, and $\chi$ an algebraic idele class character of $\mathbf{A}_L^{\times}$ with $\chi \neq \chi^{\tau}$, and $\AI^{E}_{L}$ is the automorphic induction from $\GL_1(\mathbf{A}_L)$ to $\GL_2(\mathbf{A}_E)$. In this case we have $\rho_p := \Ind_L^E \chi^{\Gal}_p$ where $\chi^{\Gal}_p : G_L \rightarrow \overline{\mathbf{Q}}_p^{\times}$ is the $p$-adic Galois character corresponding to $\chi$. The assertions in theorem 1.1 and 1.2 is then the standard consequence of the dictionary between algebraic idele class character and compatible systems of $p$-adic Galois characters, together with the fact that automorphic induction is compatible with local Langlands correspondence of $\GL_2$. 

In case 2, the twist $\pi \otimes \beta$ would then arise from (Arthur-Clozel) base change of a cuspidal automorphic representation on $\GL_2(\mathbf{A}_F)$ of cohomological type, and in this case both theorems reduce to known assertions about Galois representations associated to Hilbert modular forms (in any case is a special case of the result of \cite{ChH}), together with the compatibility of base change with local Langlands correspondence.

Henceforth we assume until the end of section 5.2 that $\pi$ is neither in case 1 or 2.

\subsection{Lifting to $\GSp_4(\mathbf{A}_F)$}

Under the assumption (\textbf{Char}) on the central character $\omega$ there are exactly two algebraic idele class characters $\widetilde{\omega}_1, \widetilde{\omega}_2$ such that $\omega = \widetilde{\omega}_i \circ N_{E/F}$. We have $\widetilde{\omega}_1/\widetilde{\omega}_2$ being equal to the quadratic idele class character $\epsilon$ of $\mathbf{A}_F^{\times}$ corresponding to the quadratic extension $E/F$. Hence there will be exactly one  of the two $\widetilde{\omega}_i$, that we will denote as $\widetilde{\omega}$, for which $\widetilde{\omega}_v(-1) = (-1)^{\mu_0}$ for all archimedean places $v$ of $F$; equivalently for each archimedean $v$ the local component $\widetilde{\omega}_v$ is given by $a \mapsto a^{-\mu_0}$.

We denote $\Pi^{\prime}$ the following representation:
\begin{eqnarray}
\Pi^{\prime} := \AI_E^F (\pi \otimes |\det|_{\mathbf{A}_E}) 
\end{eqnarray}
where $\AI_E^F $ is the automorphic induction from $\GL_2(\mathbf{A}_E)$ to $\GL_4(\mathbf{A}_F)$ in the sense of Arthur-Clozel \cite{AC}. We are assuming that $\pi$ (and hence $\pi \otimes |\det|_{\mathbf{A}_E}$) does not arise as base change from $\GL_2(\mathbf{A}_F)$, so $\Pi^{\prime}$ is cuspidal automorphic by {\it loc. cit.} 

To ease notation put $\pi^{\prime} := \pi \otimes |\det|_{\mathbf{A}_E}, \omega^{\prime} = \omega \cdot |\cdot|_{\mathbf{A}_E}^2, \widetilde{\omega}^{\prime} = \widetilde{\omega} \cdot | \cdot|_{\mathbf{A}_F}^2  $. Then since $\pi^{\prime}$ is on $\GL_2(\mathbf{A}_E)$, we have
\begin{eqnarray}
(\pi^{\prime})^{*} \otimes \omega^{\prime} \cong \pi^{\prime}
\end{eqnarray}
which implies readily that
\begin{eqnarray}
(\Pi^{\prime})^* \otimes \widetilde{\omega}^{\prime} \cong \Pi^{\prime}
\end{eqnarray}
(for example using strong multiplicity one and the relation between the Satake parameters of $\pi^{\prime}$ and $\AI^F_E(\pi^{\prime})$. Thus $\Pi^{\prime}$ defines a simple generic parameter $\Psi_2(4)$ which is self-dual with respect to $\widetilde{\omega}^{\prime}$. It is in fact of symplectic type, i.e. $\Pi^{\prime}$ defines a simple generic parameter of $\Psi_2(G,\widetilde{\omega}^{\prime})$, which is demonstrated as follows (the author is grateful to Wee Teck Gan for pointing the identity (5.4) below):

\begin{proposition}
$\Pi^{\prime}$ is of symplectic type with respect to $\widetilde{\omega}^{\prime}$, i.e. the twisted exterior square $L$-function $L(s,\Pi^{\prime},\Lambda^2 \otimes (\widetilde{\omega}^{\prime})^{-1})$ has a pole at $s=1$.
\end{proposition}
\begin{proof}
We have the following factorization formula of $L$-functions:
\begin{eqnarray}
& & L(s,\Pi^{\prime},\Lambda^2 \otimes (\widetilde{\omega}^{\prime})^{-1}) \\ & =&  \zeta_F(s) \cdot L(s,\epsilon) \cdot L(s,\pi^{\prime},\Asai_{E/F}^- \otimes (\widetilde{\omega}^{\prime})^{-1}). \nonumber
\end{eqnarray}
Here $\zeta_F(s)$ is the usual Dedekind zeta function of $F$ which contributes a pole at $s=1$, while $\epsilon$ is as before the quadratic idele class character of $\mathbf{A}^{\times}_F$ corresponding to $E/F$, and $L(s,\epsilon)$ is the usual Hecke $L$-function of $\epsilon$, which does not vanish at $s=1$. The term $  L(s,\pi^{\prime},\Asai_{E/F}^- \otimes (\widetilde{\omega}^{\prime})^{-1})$ is the twisted Asai $L$-function of $\pi^{\prime}$ (to be precise there are two Asai $L$-functions, the (+) and the (-) Asai $L$-functions and here we are using the (-) one). The non-vanishing of $ L(s,\pi^{\prime},\Asai_{E/F}^- \otimes (\widetilde{\omega}^{\prime})^{-1})$ at $s=1$ follows from Shahidi's theorem \cite{S1}. See p. 302 of \cite{R} for the discussion. Indeed, we can choose an idele class character $\chi$ of $\mathbf{A}_E^{\times}$ such that $\chi|_{\mathbf{A}_F^{\times}} = \widetilde{\omega}^{\prime}$. One then has $L(s,\pi^{\prime},\Asai_{E/F}^- \otimes (\widetilde{\omega}^{\prime})^{-1}) = L(s,\pi^{\prime} \otimes \chi^{-1},\Asai_{E/F}^- )$ which reduces to the standard Asai $L$-function to which the non-vanishing result of \cite{S1}, theorem 5.1, applies. 

Identity (5.4) can immediately be verified directly by computation with Satake parameters. Hence we conclude the result.
\end{proof}

Thus in particular by Arthur's global classification theorem $\Pi^{\prime}$ and the similitude character $\widetilde{\omega}^{\prime}$ define a (global ) simple generic parameter in $\Psi_2(G,\widetilde{\omega}^{\prime})$ hence corresponds to a global packet of cuspidal automorphic representations of $\GSp_4(\mathbf{A}_F)$. For any representation $\Pi$ in this packet, the central character of $\Pi$ is given by $\widetilde{\omega}^{\prime}$, whose local component at any archimedean $v$ of $F$ is given by $\widetilde{\omega}^{\prime}_v: a \rightarrow a^{-\mu_0+2}$. 

\begin{proposition}
For any archimedean place $v$ of $F$, the local component $\Pi_v$ belongs to the archimedean $L$-packet defined by the $L$-parameter $\phi_{(\mu_0 -2;n_v,0)}$ (notation as in section 3.1). Thus it is a limit of discrete series.
\end{proposition}
\begin{proof}
For each archimedean place $w$ of $E$, the local component $\pi^{\prime}_w = (\pi \otimes |\det|_{\mathbf{A}_E})_w$ has $L$-parameter given by $\phi_{(\mu_0 - 2,n_w)}$. The $L$-parameter of $\Pi^{\prime}_v = (\AI^F_E(\pi^{\prime}))_v$, for a place $v$ of $F$ below $w$, is thus given by the induction of $\phi_{(\mu_0- 2,n_w)}$ from $W_{E_w}$ to $W_{F_v}$, i.e. given by
\begin{eqnarray*}
z \mapsto |z|^{-\mu_0 +2} \cdot \left( \begin{array}{rrrr}    (z/\overline{z})^{n_w/2}   &  &  &   \\     &   (z/\overline{z})^{-n_w/2}   &  &     \\   &  & (z/\overline{z})^{-n_w/2}  &     \\  &  &  & (z/\overline{z})^{n_w/2}        \end{array} \right) 
\end{eqnarray*}
for  $z \in W_{E_w } = \mathbf{C}^{\times}$, and
\begin{eqnarray}
j \mapsto  \left( \begin{array}{rrrr}    & &  +1 &   \\ & &  & +1  \\    (-1)^{n_w} & & & \\ &  (-1)^{n_w}    & &   \end{array} \right). \nonumber
\end{eqnarray}
This is thus the $L$-parameter of the irreducible admissible $\Pi_v$ of $\GSp_4(F_v)$. Recall that $n_w \equiv \mu_0+1 \mod{2}$, and we set $n_w = n_v$. Hence it is easy to see that this is equivalent to the parameter $\phi_{(\mu_0-2;n_v,0)}$.
\end{proof}

\begin{rem}
\end{rem}
As mentioned in the introduction the above lifting result can also be obtained using the method of theta correspondence, with the non-vanishing of the theta lifting being completed by S. Takeda \cite{Ta}. Under some local conditions on $\pi$ the result is also obtained by P.S. Chan \cite{C}.

\bigskip

In particular we can pick a representation $\Pi$ in this packet whose archimedean component at all archimedean places of $F$ is given by holomorphic limit of discrete series. This is the representation to which we will apply the results of section 4.

\subsection{Quadratic Twists}
Thus once again fix a prime $p$. The hypotheses for theorem 4.14 are all satisfied for $\Pi$. Thus we see that attached to $\Pi$ is a continuous semi-simple four dimensional Galois representation 
\[
R_p: G_F \rightarrow \GL_4(\barQp)
\] 
unramified outside the places of $F$ dividing $p$ and the places where $\Pi$ is not spherical, and such that for any place $v$ of $F$ not dividing $p$, we have
\begin{eqnarray}
\iota_p \WD(R_p |_{G_{F_v}})^{ss} \cong k_{*} \rec_v(\Pi_v \otimes |c|^{-3/2}_v  )^{ss}.
\end{eqnarray}

Recall that $\Pi$ is classified by the simple generic parameter $\Pi^{\prime}$ of $\GL_4(\mathbf{A}_F)$. So (tautologically) we have 
\[
k_{*} \rec_v(\Pi_v \otimes |c|^{-3/2}_v  ) \cong \mathcal{L}_v(\Pi^{\prime}_v \otimes |\det|_v^{-3/2} ).
\]
Hence
\begin{eqnarray}
\iota_p \WD(R_p |_{G_{F_v}})^{ss} &\cong& \mathcal{L}_v(\Pi^{\prime}_v \otimes |\det|_v^{-3/2} )^{ss} \\ &\cong & \mathcal{L}_v( (\AI_E^F \pi)_v \otimes |\det|_v^{-1/2} )^{ss}. \nonumber
\end{eqnarray} 
(recall that $\Pi^{\prime} = \AI^F_E(\pi^{\prime}) = \AI^F_E(\pi \otimes |\det|_{\mathbf{A}_E})$). If we take $S_{\pi}$ to be the set of finite primes $w$ of $E$ where $\pi_w$, $\pi^{\tau}_w$ or $E_w/F_v$ ramifies (here $v$ is the place of $F$ below $w$), and $S_{\pi,F}$ the set of places of $F$ lying below $S_{\pi}$, then $\Pi$ is spherical outside $S_{\pi,F}$, and hence $R_p$ is unramified outside the set of places dividing $p$ and $S_{\pi,F}$. Note that by (5.6), together with Cebotarev denisty and the Brauer-Nesbitt theorem, the Galois representation $R_p$ depends only on $\pi$ and not on the choice of the cuspidal automorphic representation $\Pi$ belonging to the global packet defined by the parameter $\Pi^{\prime}$.  

Now since $\Pi^{\prime} = \AI^F_E(\pi^{\prime})$ it satisifes:
\begin{eqnarray}
\Pi^{\prime} \cong \Pi^{\prime} \otimes (\epsilon \circ \det)
\end{eqnarray}
where as before $\epsilon$ is the quadratic idele class character of $\mathbf{A}_F^{\times}$ corresponding to $E/F$. Equation (5.6) thus implies 
\begin{eqnarray}
 \WD((R_p \otimes \epsilon) |_{G_{F_v}})^{ss}  \cong  \WD(R_p |_{G_{F_v}})^{ss} 
\end{eqnarray}
(here we abuse notation and we regard $\epsilon$ as a Galois character of $G_F$ by class field theory). By Cebotarev density and the Brauer-Nesbitt theorem, we thus have
\begin{eqnarray}
R_p \otimes \epsilon \cong R_p.
\end{eqnarray}

\begin{proposition}
There exists a continuous semi-simple two-dimensional representation
\begin{eqnarray}
\rho_p : G_E \rightarrow \GL_2(\barQp) 
\end{eqnarray}
such that
\begin{eqnarray}
R_p = \Ind_E^F \rho_p.
\end{eqnarray}
Hence
\begin{eqnarray}
R_p|_{G_E} = \rho_p \oplus \rho_p^{\tau}.
\end{eqnarray}
In particular $\rho_p$ is unramified at $w$ if $w$ does not divide $p$ and $w \notin S_{\pi}$
\end{proposition}
\begin{proof}
The proof is exactly the same as in Lemma 4.1 of \cite{BH}, so we just recall the gist of the argument. Suppose first that $R_p$ is irreducible. Now by (5.9), we have $\Hom_{G_F}(\epsilon,\End R_p) \neq 0$. Since $\epsilon$ is non-trivial but the restriction of $\epsilon$ to $G_E$ is trivial, we see that $\dim_{\barQp} \Hom_{G_E}(1,R_p|_{G_E}) \geq 2$. By Schur's lemma, we see that $R_p|_{G_E}$ is reducible. Let $\rho_p$ be a subrepresentation of $R_p|_{G_E}$ of minimal dimension. Then By Frobenius reciprocity $\Ind_E^F \rho_p$ is a subrepresentation of $R_p$, hence by the irreducibility of $R_p$ we must have $R_p = \Ind_E^F \rho_p$ and hence $\rho_p$ is of dimension $2$.

The cases where $R_p$ is reducible can be treated exactly as in Lemma 4.1 of \cite{BH}. 
\end{proof}

\begin{rem}
\end{rem}
In fact $\rho_p$ is always irreducible (see proposition 5.9 below), and under our assumption that $\pi$ does not arise as base change from $\GL_2(\mathbf{A}_F)$, we will see that we have $\rho_p \ncong \rho_p^{\tau}$ (proposition 5.10 below), so in fact $R_p = \Ind_E^F \rho_p$ is irreducible.

\begin{rem}
\end{rem}
Note that (5.11) does not specify $\rho_p$ uniquely as we can replace $\rho_p$ by $\rho_p^{\tau}$. For the moment we choose either one for $\rho_p$. We will pin down this choice more precisely later.

\bigskip

\begin{proposition}
For any finite place $w$ of $E$ not dividing $p$, we have:
\begin{eqnarray}
& & \iota_p \WD \big( \rho_p \big|_{G_{E_w}} \big)^{ss} \oplus \iota_p \WD \big( \rho_p^{\tau}  \big|_{G_{E_w}}  \big)^{ss}  \\ & & \cong \mathcal{L}_w(\pi_w \otimes |\det|_w^{-1/2})^{ss} \oplus \mathcal{L}_w( (\pi ^{\tau})_w \otimes |\det|_w^{-1/2})^{ss}. \nonumber
\end{eqnarray}
\end{proposition}
\begin{proof}
This is a direct computation. Let $v$ be the prime of $F$ below $w$. Then
\begin{eqnarray}
& & \iota_p \WD(R_p|_{G_{F_v}})^{ss}|_{W_{E_w}} = \iota_p \WD(R_p|_{G_{E_w}})^{ss} \\ & = & \iota_p  \WD \big( \rho_p \big|_{G_{E_w}} \big)^{ss} \oplus  \iota_p \WD \big( \rho_p^{\tau}  \big|_{G_{E_w}}  \big)^{ss}. \nonumber
\end{eqnarray}

On the other hand by (5.6) and the compatibility of the local Langlands correspondence with base change, we see that the left hand side of (5.14) is isomorphic to
\begin{eqnarray}
& & \mathcal{L}_v((\AI^F_E \pi)_v \otimes |\det|_v^{-1/2})^{ss}|_{G_{E_w}} \\ & =&  \mathcal{L}_w( (\BC_F^E (\AI_E^F \pi))_w \otimes |\det|_w^{-1/2}   )^{ss} \nonumber \\
&=&  \mathcal{L}_w( (\pi \boxplus \pi^{\tau} )_w \otimes |\det|_w^{-1/2}   )^{ss} \nonumber \\
&=& \mathcal{L}_w(\pi_w \otimes |\det|_w^{-1/2})^{ss} \oplus \mathcal{L}_w( (\pi ^{\tau})_w \otimes |\det|_w^{-1/2})^{ss}. \nonumber
\end{eqnarray}
(In the above $\pi \boxplus \pi^{\tau}$ is the isobaric sum of $\pi$ and $\pi^{\tau}$ as an automorphic representation of $\GL_4(\mathbf{A}_F)$).
\end{proof}

In order to separate the contributions $\rho_p$ and $\rho_p^{\tau}$ in (5.13), we follow \cite{BH} by considering twisting of $\pi$ by quadratic characters. Denote by $\mathcal{M}$ the set of quadratic idele class characters of $\mathbf{A}_E^{\times}$ (we allow $\eta$ to be trivial). Note that for any $\eta \in \mathcal{M}$, the cuspidal automorphic representation $\pi \otimes (\eta \circ \det)$ satisfies the same hypotheses as those for $\pi$ in the beginning of section 5. Hence we may apply the same constructions above to $\pi \otimes (\eta \circ \det)$: denote by $\rho_p^{\eta}$ and $R_p^{\eta} = \Ind_E^F \rho_p^{\eta}$ the corresponding $2$-dimensional Galois representation of $G_E$ and $4$-dimensional representation of $G_F$ respectively, given by the construction above with $\pi$ being replaced by $\pi \otimes (\eta \circ \det)$ (here $\rho_p^{\eta}$ is not to be confused with $\rho_p^{\tau}$ which is the conjugate of $\rho_p$ by $\tau$). 

Proposition 5.7 applied to $\pi \otimes (\eta \circ \det)$ gives the following result (to ease notation we have written $\pi \otimes \eta$ for $\pi \otimes (\eta \circ \det)$).

\begin{proposition}
For any finite place $w$ of $E$ not dividing $p$, we have:
\begin{eqnarray}
& & \iota_p \WD \big( \rho_p^{\eta} \big|_{G_{E_w}} \big)^{ss} \oplus \iota_p \WD \big( (\rho_p^{\eta})^{\tau}  \big|_{G_{E_w}}  \big)^{ss}  \\ & & \cong \mathcal{L}_w(\pi_w \otimes \eta_w \cdot |\cdot|_w^{-1/2})^{ss} \oplus \mathcal{L}_w( \pi^{\tau}_w \otimes \eta^{\tau}_w \cdot |\cdot|_w^{-1/2})^{ss} . \nonumber
\end{eqnarray}
\end{proposition}

Our goal is to show that we can choose $\rho$ and the $\rho^{\eta}$ compatibly so that in fact $\rho^{\eta} = \rho \otimes \eta$ for all $\eta \in \mathcal{M}$ (here we have abused notation and regard $\eta$ as a Galois character via class field theory). To do this we need some preparations, following the lines of arguments in \cite{BH}. 

First we need the following key result:
\begin{proposition}
The representation $\rho_p^{\eta}|_{G_L}$ is irreducible for any quadratic extension $L$ of $E$ (in particular $\rho_p^{\eta}$ is irreducible).
\end{proposition}

\begin{proof}
This can be proved exactly as in Lemma 5.1 of \cite{BH} using the argument in section 3 of \cite{T2}. This is where we need to use the assumption that we are not in case 1 discussed in the beginning of section 5.
\end{proof}

Now for any $\eta \in \mathcal{M}$ put $\delta_{\eta}:=\eta \cdot \eta^{\tau}$. Denote by $L_{\delta_{\eta}}$ the quadratic extension of $E$ cut out by $\delta$. Twisting equation (5.16) via $ \eta_w^{-1} = \eta_w$, and denoting $\widehat{\rho}_p^{\eta} := \rho_p^{\eta} \otimes \eta^{-1}$, we obtain, for any finite place of $E$ not dividing $p$:
\begin{eqnarray}
& & \iota_p \WD \big( \widehat{\rho}_p^{\eta} \big|_{G_{E_w}} \big)^{ss} \oplus \iota_p \WD \big( (\widehat{\rho}_p^{\eta})^{\tau} \otimes \delta_{\eta} \big) \big|_{G_{E_w}}  \big)^{ss}   \\ & & \cong \mathcal{L}_w(\pi_w \otimes  |\cdot|_w^{-1/2})^{ss} \oplus \mathcal{L}_w( (\pi^{\tau}_w \otimes (\delta_{\eta})_w \cdot |\cdot|_w^{-1/2})^{ss} . \nonumber
\end{eqnarray}

Now by proposition 5.9 the restriction of $\widehat{\rho}_p^{\eta}$ to $G_{L_{\delta_{\eta}}}$ is irreducible, hence semi-simple (in any case being the restriction of a semi-simple representation of $G_E$ to an open normal subgroup is semi-simple). Since by definition the restriction of $\delta_{\eta}$ to $G_{L_{\delta_{\eta}}}$ is trivial, we see that  by using the restriction of equation (5.13) and (5.17) to the various decompositions groups of $L_{\eta}$ at primes not dividing $p$, together with Cebotarev denisty and Brauer-Nesbitt theorem, one has 

\begin{eqnarray}
\widehat{\rho}_p^{\eta}\big|_{G_{L_{\delta_{\eta}}}} \oplus (\widehat{\rho}^{\eta}_p)^{\tau} \big|_{G_{L_{\delta_{\eta}}}} \cong  \rho_p \big|_{G_{L_{\delta_{\eta}}}} \oplus \rho_p^{\tau} \big|_{G_{L_{\delta_{\eta}}}}.
\end{eqnarray}

By proposition 5.9 again each summand of (5.18) is irreducible. Hence we either have $\widehat{\rho}_p^{\eta}\big|_{G_{L_{\delta_{\eta}}}} \cong \rho_p \big|_{G_{L_{\delta_{\eta}}}}$, or  $\widehat{\rho}_p^{\eta}\big|_{G_{L_{\delta_{\eta}}}} \cong  \rho_p^{\tau} \big|_{G_{L_{\delta_{\eta}}}}$. If the latter case occur, we replace $\rho_p^{\eta}$ by $(\rho_p^{\eta})^{\tau}$ (which is legitimate; note that $\eta|_{G_{L_{\delta_{\eta}}}} = \eta^{\tau}|_{G_{L_{\delta_{\eta}}}}$). With this choice, we then have for all $\eta \in \mathcal{M}$:
\begin{eqnarray}
\widehat{\rho}_p^{\eta} \big|_{G_{L_{\delta_{\eta}}}}  \cong \rho_p \big|_{G_{L_{\delta_{\eta}}}}.
\end{eqnarray}

Again by irreducibility of $\rho_p$ and $\rho_p^{\eta}$, equation (5.19) implies that
\begin{eqnarray}
\widehat{\rho}_p^{\eta} \cong \rho_p \otimes \psi_{\eta}
\end{eqnarray}
for a quadratic character $\psi_{\eta}$ of $\Gal(L_{\delta_{\eta}}/E )$, and hence $\psi_{\eta}$ is either trivial or equal to $\delta_{\eta}$. In other words, $\rho_p^{\eta} \cong \rho_p \otimes \eta$ or $\rho_p^{\eta} \cong \rho_p \otimes \eta^{\tau}$.

\begin{proposition}
Let $\eta \in \mathcal{M}$, and $\delta$ be a quadratic character of $G_E$. Then we have

\bigskip

\noindent 1. $\rho_p^{\eta} \otimes \delta \ncong \rho_p^{\eta}$ for $\delta$ non-trivial.

\bigskip

\noindent 2. $\rho_p^{\eta} \otimes \delta \ncong (\rho_p^{\eta})^{\tau}$ in all cases.
\end{proposition}
\begin{proof}
This again can be proved exactly as in Lemma 5.2 of \cite{BH}. However we would like to give some details to clarify some part of the argument in {\it loc. cit.}

For part 1, if $\rho_p^{\eta} \otimes \delta \cong \rho_p^{\eta}$ for some non-trivial quadratic $\delta$, then $\Hom_{G_E}(\delta,\End \rho_p^{\eta}) \neq 0$. Denoting by $L$ the quadratic extension of $E$ cut out by $\delta$, one then has $\dim_{\barQp} \Hom_{G_L}(1,\End (\rho_p^{\eta} |_{G_L})) \geq 2$. Schur's lemma implies that $\rho_p|_{G_L}$ is reducible, contradicting proposition 5.9.

For part 2, first consider the case where $\delta$ is non-trivial. Thus assume that $\rho_p^{\eta} \otimes \delta \cong (\rho_p^{\eta})^{\tau}$. Applying $\tau$ to this isomorphism we obtain $(\rho_p^{\eta})^{\tau} \otimes \delta^{\tau} \cong \rho_p^{\eta}$, and hence we have $\rho_p^{\eta} \otimes \delta^{-1} \delta^{\tau}\cong \rho_p^{\eta}$. By part 1 just proved above, then implies $\delta^{\tau} = \delta$. So twisting equation (5.16) by $\delta_w = \delta^{\tau}_w$ we obtain:

\begin{eqnarray*}
& & \iota_p \WD \big( (\rho_p^{\eta} \otimes \delta) \big|_{G_{E_w}} \big)^{ss} \oplus \iota_p \WD \big( (\rho_p^{\eta} \otimes \delta)^{\tau} \big) \big|_{G_{E_w}}  \big)^{ss}  \\ & & \cong \mathcal{L}_w(\pi_w \otimes (\eta \cdot \delta)_w \cdot |\cdot|_w^{-1/2})^{ss} \oplus \mathcal{L}_w( \pi^{\tau}_w \otimes (\eta^{\tau} \cdot \delta^{\tau})_w \cdot |\cdot|_w^{-1/2})^{ss} . \nonumber
\end{eqnarray*}

But by assumption $\rho_p^{\eta} \otimes \delta \cong (\rho_p^{\eta})^{\tau}$ so we obtain
\begin{eqnarray*}
& & \iota_p \WD \big( \rho_p^{\eta}  \big|_{G_{E_w}} \big)^{ss} \oplus \iota_p \WD \big( (\rho_p^{\eta})^{\tau}  \big) \big|_{G_{E_w}}  \big)^{ss}  \\ & & \cong \mathcal{L}_w(\pi_w \otimes (\eta \cdot \delta)_w \cdot |\cdot|_w^{-1/2})^{ss} \oplus \mathcal{L}_w( \pi^{\tau}_w \otimes (\eta^{\tau} \cdot \delta^{\tau})_w \cdot |\cdot|_w^{-1/2})^{ss} . \nonumber
\end{eqnarray*}

Combining with equation (5.16) we thus obtain
\begin{eqnarray*}
& &   \mathcal{L}_w( (  (\pi \otimes \eta) \boxplus (\pi \otimes \eta)^{\tau})_w \otimes |\cdot|_w^{-1/2})       )^{ss}    \\
& \cong &   \mathcal{L}_w( ((\pi \otimes \eta \cdot \delta) \boxplus (\pi  \otimes  \eta \cdot \delta)^{\tau} )_w \otimes  |\cdot|_w^{-1/2}       )^{ss}. \nonumber
\end{eqnarray*}

Hence by strong multiplicity one for isobaric automorphic representations (theorem of Jacquet-Shalika \cite{JS}) we have
\[
(\pi \otimes \eta) \boxplus (\pi \otimes \eta)^{\tau} = (\pi \otimes \eta \cdot \delta) \boxplus  (\pi \otimes \eta \cdot \delta)^{\tau}.
\]

So we have either $\pi \otimes \eta = \pi \otimes \eta \cdot \delta$, i.e. $\pi \otimes \delta = \pi$, or $\pi \otimes \eta = (\pi \otimes \eta \cdot \delta)^{\tau}$. The former case is already excluded as the case 1 at the beginning of section 5, because $\delta$ is non-trivial. So we have $\pi \otimes \eta = (\pi \otimes \eta \cdot \delta)^{\tau}$.

Now recall that $\delta$ satisfies $\delta^{\tau} = \delta$. Hence $\delta$ (as a quadratic idele class character of $\mathbf{A}_E^{\times}$) can be written as $\delta = \widetilde{\delta} \circ N_{E/F}$ for some quadratic idele class character $\widetilde{\delta}$ of $\mathbf{A}_F^{\times}$. In particular $\delta|_{\mathbf{A}_F^{\times}}$ is trivial. The argument on p. 529 of \cite{H} shows that we can write $\delta = \nu/\nu^{\tau}$ for some finite order idele class character $\nu$ of $\mathbf{A}_E^{\times}$. Thus we have
\[
\pi \otimes \eta \cdot \nu = (\pi \otimes \eta \cdot \nu)^{\tau}
\]
Thus $\pi \otimes \eta \cdot \nu$ arises as base change from $\GL_2(\mathbf{A}_F)$. But this is the case 2 that we excluded in the beginning of section 5.

Finally we treat the case where $\delta$ is trivial, i.e. we assume that $\rho_p^{\eta} \cong (\rho_p^{\eta})^{\tau}$. Recall that for any $\eta^{\prime} \in \mathcal{M}$ (including this particular $\eta$ that we are considering) we have either $\rho_p^{\eta^{\prime}} \cong \rho_p \otimes \eta^{\prime}$ or $\rho_p^{\eta^{\prime}} \cong \rho_p \otimes (\eta^{\prime})^{\tau}$. It follows that for any $\eta^{\prime} \in \mathcal{M}$ there are at most four possibilities: $\rho_p^{\eta \eta^{\prime}} \cong \rho_p^{\eta} \otimes \eta^{\prime}$, $\rho_p^{\eta \eta^{\prime}} \cong \rho_p^{\eta} \otimes \delta_{\eta} \eta^{\prime}$, $\rho_p^{\eta \eta^{\prime}} \cong \rho_p^{\eta} \otimes \delta_{\eta} (\eta^{\prime})^{\tau}$ or $\rho_p^{\eta \eta^{\prime}} \cong \rho_p^{\eta} \otimes (\eta^{\prime})^{\tau}$. A computation shows that assuming $\rho_p^{\eta} \cong (\rho_p^{\eta})^{\tau}$ we would have (in all possible cases)
\[
(\rho_p^{\eta \eta^{\prime}} )^{\tau} \cong \rho_p^{\eta \eta^{\prime}} \otimes \delta_{\eta^{\prime}}
\]
for all $\eta^{\prime} \in \mathcal{M}$. In particular if we choose $\eta^{\prime}$ so that $\delta_{\eta^{\prime}} $ is non-trivial, then we obtain a contradiction to the case just proved above.

\end{proof}

\begin{proposition}
We have either $\rho_p^{\eta} \cong \rho_p \otimes \eta$ for all $\eta \in \mathcal{M}$, or $\rho_p^{\eta} \cong \rho_p \otimes \eta^{\tau}$ for all $\eta \in \mathcal{M}$. 
\end{proposition}
\begin{proof}
We already know that for each $\eta \in \mathcal{M}$, we have either $\rho_p^{\eta} \cong \rho_p \otimes \eta$ or $\rho_p^{\eta} \cong \rho_p \otimes \eta^{\tau}$. Suppose that we have $\rho_p^{\eta_0} \cong \rho_p \otimes \eta_0^{\tau}$ for some $\eta_0 \in \mathcal{M}$ with $\eta_0 \neq \eta_0^{\tau}$, i.e. $\delta_{\eta_0}$ not trivial. Equivalently in (5.20) we have $\psi_{\eta_0} = \delta_{\eta_0}$ non-trivial. We are going to show that $\rho_p^{\eta} \cong \eta_p \otimes \eta^{\tau}$ for all $\eta \in \mathcal{M}$.

Denote by $L_{\eta,\eta_0}$ the quadratic extension of $E$ cut out by $\delta_{\eta} \cdot \delta_{\eta_0}$. Thus $\delta_{\eta} |_{G_{L_{\eta,\eta_0}}} =\delta_{\eta_0} |_{G_{L_{\eta,\eta_0}}}$.  By considering the restriction of (5.17) for both $\eta$ and $\eta_0$ to the various decomposition groups of $G_{L_{\eta,\eta_0}}$ at primes not dividing $p$, we see that (again by Cebotarev density and Brauer-Nesbitt):
\[
\widehat{\rho}_p^{\eta} \big|_{G_{L_{\eta,\eta_0}}} \oplus ((\widehat{\rho}_p^{\eta})^{\tau} \otimes \delta_{\eta}) \big|_{G_{L_{\eta,\eta_0}}} \cong \widehat{\rho}_p^{\eta_0} \big|_{G_{L_{\eta,\eta_0}}} \oplus ((\widehat{\rho}_p^{\eta_0})^{\tau} \otimes \delta_{\eta_0}) \big|_{G_{L_{\eta,\eta_0}}}.
\]
This implies that we have either $\widehat{\rho}_p^{\eta} \big|_{G_{L_{\eta,\eta_0}}} \cong \widehat{\rho}_p^{\eta_0} \big|_{G_{L_{\eta,\eta_0}}}$, or we have $\widehat{\rho}_p^{\eta} \big|_{G_{L_{\eta,\eta_0}}} \cong (\widehat{\rho}_p^{\eta_0})^{\tau} \otimes \delta_{\eta_0} \big|_{G_{L_{\eta,\eta_0}}} $. If the latter case were to occur, we would have $\widehat{\rho}_p^{\eta} \otimes \psi \cong (\widehat{\rho}_p^{\eta_0})^{\tau} \otimes \delta_{\eta_0}$ for some quadratic character $\psi$, which would in turn imply by (5.20) that $\rho_p \otimes \psi^{\prime} \cong \rho_p^{\tau} $ for some quadratic character $\psi^{\prime}$, contradicting part 2 of proposition 5.10. Thus we have $\widehat{\rho}_p^{\eta} \big|_{G_{L_{\eta,\eta_0}}} \cong \widehat{\rho_0}_p^{\eta_0} \big|_{G_{L_{\eta,\eta_0}}}$, hence 
\begin{eqnarray}
\widehat{\rho}_p^{\eta}   \cong \widehat{\rho}_p^{\eta_0} \otimes \psi_{\eta,\eta_0}
\end{eqnarray}
for some quadratic character $\psi_{\eta,\eta_0}$ of $\Gal(L_{\eta,\eta_0}/E)$. Thus $\psi_{\eta,\eta_0}$ is either trivial or equal to $\delta_{\eta} \cdot \delta_{\eta_0}$.

We claim that 
\begin{eqnarray}
\psi_{\eta,\eta_0} = \psi_{\eta} \cdot \psi_{\eta_0}
\end{eqnarray}
(where $\psi_{\eta}, \psi_{\eta_0}$ are as in $(5.20)$). Indeed combining (5.21) and (5.20) (for both $\eta$ and $\eta_0$) we have
\[
\rho_p \cong \rho_p \otimes (\psi_{\eta,\eta_0} \psi_{\eta}^{-1} \psi_{\eta_0})
\]
whence the claim by part 1 of proposition 5.10. 

Recall that we are assuming that $\psi_{\eta_0} = \delta_{\eta_0}$ is non-trivial. By (5.22) we see that $\psi_{\eta,\eta_0}$ and $\psi_{\eta}$ cannot be both trivial, i.e. we have $\psi_{\eta,\eta_0} = \delta_{\eta} \cdot \delta_{\eta_0}$ or $\psi_{\eta} = \delta_{\eta}$. By (5.22) again these two are equivalent, so we have $\psi_{\eta} = \delta_{\eta}$ in all cases, i.e. $\rho_p^{\eta} \cong \rho_p \otimes \eta^{\tau}$ as required.
\end{proof}

Thus, if the latter case of proposition 5.11 occurs, then we just replace $\rho_p$ by $\rho_p^{\tau}$, and $\rho_p^{\eta}$ by $(\rho_p^{\eta})^{\tau}$ for all $\eta \in \mathcal{M}$. With this choices we then have $\rho_p^{\eta} \cong \rho_p \otimes \eta$ for all $\eta \in \mathcal{M}$.

\begin{corollary}
For any $\eta \in \mathcal{M}$, and any finite place $w$ of $E$ not dividing $p$, we have:
\begin{eqnarray}
& & \iota_p \WD \big( \rho_p \otimes \eta \big|_{G_{E_w}} \big)^{ss} \oplus \iota_p \WD \big( (\rho_p \otimes \eta)^{\tau} \big) \big|_{G_{E_w}}  \big)^{ss}  \\ & & \cong \mathcal{L}_w(\pi_w \otimes \eta_w \cdot |\cdot|_w^{-1/2})^{ss} \oplus \mathcal{L}_w( \pi^{\tau}_w \otimes \eta^{\tau}_w \cdot |\cdot|_w^{-1/2})^{ss} . \nonumber
\end{eqnarray}
\end{corollary}
\begin{proof}
Use proposition 5.8.
\end{proof}

\subsection{Proof of main theorems}

With corollary 5.12 in our hand we can now prove the main theorems 1.1 and 1.2 in the introduction. 

First we set up some notation. For any finite prime $w$ of $E$ not dividing $p$, and element $\sigma \in W_{E_w}$, denote by $P_{\pi,w}(X,\sigma)$ the inverse characteristic polynomial (in the variable $X$) of the semisimple part of the Weil-Deligne representation $\iota_p \WD(\rho_p|_{G_{E_w}})^{F-ss}$ evaluated at the element $\sigma$. Similarly denote by $Q_{\pi,w}(X,\sigma)$ the inverse characteristic polynomial of the semi-simple part of the Weil-Deligne representation $\mathcal{L}_w(\pi_w \otimes |\cdot|_w^{-1/2})$ evaluated at $\sigma$. Same notation of course apply with $\pi$ replaced by $\pi^{\tau}$ or $\pi \otimes \eta$.

\begin{proposition}
Let $w$ be a place of $E$ not dividing $p$, and splits over $F$. Then for any element $\sigma \in W_{E_w}$ of odd valuation, we have the identity:
\begin{eqnarray}
P_{\pi,w}(X,\sigma)=Q_{\pi,w}(X,\sigma)
\end{eqnarray}
\end{proposition}
\begin{proof}
The argument is similar to section 5 of \cite{BH}. Choose an $\eta^{\prime} \in \mathcal{M}$ which is unramified at $w$ and $w^{\tau}$, such that $\eta^{\prime}(w)=1$ and $\eta^{\prime}(w^{\tau}) = -1$ (here $\eta^{\prime}(w)$ is the value of $\eta$ at a uniformizer of $E_w$ and similarly for $\eta^{\prime}(w^{\tau})$). Denote by $\{\alpha,\beta\}$ and $\{ \alpha^{\prime},\beta^{\prime} \}$ the inverse roots of $P_{\pi,w}(X,\sigma)$ and $P_{\pi^{\tau},w}(X,\sigma)$ respectively. Similarly denote by $\{ \gamma, \delta \}$ and $\{ \gamma^{\prime},\delta^{\prime} \}$ the inverse roots of $Q_{\pi,w}(X,\sigma)$ and $Q_{\pi^{\tau},w}(X,\sigma)$ respectively. 

Equation (5.23) applied to the case where $\eta$ is trivial gives the following identity of multi-sets (i.e. sets with multiplicities):
\begin{eqnarray}
\{\alpha,\beta,\alpha^{\prime},\beta^{\prime} \} = \{\gamma,\delta,\gamma^{\prime},\delta^{\prime} \}
\end{eqnarray}
and hence
\begin{eqnarray}
\alpha + \beta + \alpha^{\prime} + \beta^{\prime} = \gamma + \delta + \gamma^{\prime} + \delta^{\prime}.
\end{eqnarray}
Similarly (5.23) applied to the case where $\eta=\eta^{\prime}$ gives (this is the place where we need to assume that $\sigma$ has odd valuation):
\begin{eqnarray}
\{\alpha,\beta,-\alpha^{\prime},-\beta^{\prime} \} = \{\gamma,\delta,-\gamma^{\prime},-\delta^{\prime} \}
\end{eqnarray}
and hence
\begin{eqnarray}
\alpha + \beta -\alpha^{\prime} - \beta^{\prime} = \gamma + \delta - \gamma^{\prime} - \delta^{\prime}.
\end{eqnarray}
Once can then deduce (5.24) from (5.25)-(5.28) as in {\it loc. cit.}
\end{proof}

As before let $S_{\pi}$ be the set of finite places $w$ of $E$ such that $\pi_w$, $(\pi^{\tau})_w$ or $E_w/F_v$ ramifies ($v$ being the place of $F$ below $w$), and put $S_{\pi,p}$ to be the union of $S$ and the set of places dividing $p$. Then $\rho_p$ is unramified outside $S_{\pi,p}$. In particular we obtain, for $w \notin S_{\pi,p}$, with $w$ splits over $F$, the identity
\begin{eqnarray}
P_{\pi,w}(X,\Frob_w)=Q_{\pi,w}(X,\Frob_w).
\end{eqnarray} 
Here $\Frob_w$ denotes a lift of the geometric Frobenius element of $W_{E_w}$. 

When $w$ is inert over $F$, then in general one has no relation between $\pi_w$ and $(\pi^{\tau})_w$. However, when $\pi$ is spherical at $w$ then we do have $\pi_w \cong (\pi^{\tau})_w$ as is immediately seen from the classification of spherical representation. Similarly if $\rho_p$ is unramified at $w$ then we have $\rho_p^{\tau}|_{G_{E_{w}}} \cong \rho_p|_{G_{E_w}}$. Hence equation (5.23) (with $\eta$ being trivial) implies that for $w$ inert over $F$, $w \in S_p$, identity (5.29) also holds. Thus to conclude

\begin{corollary}
For $w \notin S_{\pi,p}$ we have
\begin{eqnarray}
P_{\pi,w}(X,\Frob_w)=Q_{\pi,w}(X,\Frob_w).
\end{eqnarray} 
\end{corollary}
As usual (5.30) characterizes the Galois representation $\rho_p$ by Cebotarev density and Brauer-Nesbitt. This completes the construction of the Galois representation $\rho_p$ associated to $\pi$.

To complete the cases of proposition 5.13 where $\sigma$ is of even valuation or when $w$ does not split over $F$, we use base change arguments. For any solvable $\CM$ extension $E^{\prime}$ of $E$, let $\pi^{\prime} = \BC_E^{E^{\prime}}(\pi)$ be the Arthur-Clozel base change of $\pi$ to $E^{\prime}$. Assume that $\pi^{\prime}$ is cuspidal. Then $\pi^{\prime}$ satisifies all the hypotheses at the beginning of section 5 (of course if $\pi^{\prime}$ were in case 1 or case 2 discussed at the beginning of section 5, then theorem 1.1 and 1.2 are known for $\pi^{\prime}$ anyway). Denote by $\rho_{\pi,p}$ and $\rho_{\pi^{\prime},p}$ the Galois representations associated to $\pi$ and $\pi^{\prime}$ respectively. Using (5.30) applied to $\pi$ and $\pi^{\prime}$, a standard argument, using Cebotarev denisty and the relation between the Satake parameters of $\pi$ and $\pi^{\prime}$, shows that 
\begin{eqnarray}
\rho_{\pi^{\prime},p} \cong \rho_{\pi,p}|_{G_{E^{\prime}}}.
\end{eqnarray}

\begin{proposition}
For any finite prime $w$ of $E$ not dividing $p$ and splits over $F$, and $\sigma \in W_{E_w}$, we have
\begin{eqnarray}
P_{\pi,w}(X,\sigma)=Q_{\pi,w}(X,\sigma).
\end{eqnarray}
\end{proposition}
\begin{proof}
First consider the case where $\sigma$ has nonzero valuation, i.e. $\sigma \notin I_{E_w}$. Write the valuation $m$ of $\sigma$ in the form $m=2^a u$, where $u$ is odd. 

Now choose a solvable totally real extension $F^{\prime}$ of $F$ such that for the solvable $\CM$ extension $E^{\prime} = E \cdot E^{\prime}$ of $E$ we have

\bigskip

 1. $\pi^{\prime} = \BC^{E^{\prime}}_E(\pi)$ is cuspidal.

\bigskip

2. There is a place $w^{\prime}$ of $E^{\prime}$ above $w$, such that the extension $E^{\prime}_{w^{\prime}}/E_w$ is the unramified extension of $E_w$ of degree $2^a$.

\bigskip
With this choice we see that $\sigma \in W_{E^{\prime}_{w^{\prime}}}$ and has odd valuation with respect to $W_{E^{\prime}_{w^{\prime}}}$. Note also that $w^{\prime}$ splits over $F^{\prime}$ because $w$ splits over $F$. Hence by (5.24) applied to $\pi^{\prime}$ and $w^{\prime}$, we obtain
\begin{eqnarray}
P_{\pi^{\prime},w^{\prime}}(X,\sigma) =Q_{\pi^{\prime},w^{\prime}}(X,\sigma). 
\end{eqnarray}
Now by (5.31) we have $P_{\pi^{\prime},w^{\prime}}(X,\sigma) = P_{\pi,w}(X,\sigma)$, and similarly by the compatibility of base change with local Langlands correspondence we have $Q_{\pi^{\prime},w^{\prime}}(X,\sigma) = Q_{\pi,w}(X,\sigma)$. Hence we obtain (5.32) from (5.33).

Thus (5.32) holds when $\sigma$ has nonzero valuation. In particular, denoting by $r^{\Gal}_{\pi,w}$ and $r^{\Aut}_{\pi,w}$ the semi-simple part of the Weil-Deligne representations of $\iota_p \WD(\rho_p|_{G_{E_w}})^{F-ss}$ and $\mathcal{L}_w(\pi_w \otimes |\cdot|_w^{-1/2})$ respectively, we have 
\begin{eqnarray}
\tr r_{\pi,w}^{\Gal}(\sigma) = \tr r_{\pi,w}^{\Aut}(\sigma)
\end{eqnarray}
for all $\sigma \in W_{E_w}$ with nonzero valuation. An argument of T.Saito (argument in part 1 of Lemma 1 of \cite{Sa}) then shows that this implies in fact (5.34) holds for all $\sigma \in W_{E_w}$. Thus by Brauer-Nesbitt we have $r_{\pi,w}^{\Gal} \cong r_{\pi,w}^{\Aut}$. In particular (5.32) holds for all $\sigma$.
\end{proof}

\begin{proposition}
For any finite place $w$ of $E$ not dividing $p$, and any $\sigma \in W_{E_w}$ we have
\begin{eqnarray}
P_{\pi,w}(X,\sigma)=Q_{\pi,w}(X,\sigma).
\end{eqnarray}
\end{proposition}
\begin{proof}
Let $v$ be the prime of $F$ below $w$. Choose a solvable totally real extension $F^{\prime}$ of $F$, such that the solvable $\CM$ extension $E^{\prime} = E \cdot F^{\prime}$ satisfies:

\bigskip

1. For any place $u$ of $F^{\prime}$ above the place $v$ of $F$ we have $E_{w} = F^{\prime}_u$.

\bigskip

2. $\pi^{\prime} = \BC_E^{E^{\prime}}(\pi)$ is cuspidal.

\bigskip

By condition 1 we see that we can choose a place $w^{\prime}$ of $E^{\prime}$ above the prime $w$ of $E$ such that $E^{\prime}_{w^{\prime}} = E_w$ and such that $w^{\prime}$ splits over $F^{\prime}$. In particular $\sigma \in W_{E^{\prime}_{w^{\prime}}}$. Hence proposition 5.15 applied to $\pi^{\prime}$ and $w^{\prime}$ gives
\[
P_{\pi^{\prime},w^{\prime}}(X,\sigma) = Q_{\pi^{\prime},w^{\prime}}(X,\sigma).
\] 
As in the proof of proposition 5.15 we have $P_{\pi^{\prime},w^{\prime}}(X,\sigma)= P_{\pi,w}(X,\sigma)$ and $Q_{\pi^{\prime},w^{\prime}}(X,\sigma)= Q_{\pi,w}(X,\sigma)$, hence we conclude the result.
\end{proof}

\bigskip

\noindent {\it Proof of theorem 1.1}

\bigskip

With the proof of proposition 5.16 we see that the statement of local-global compatibility up to semi-simplification at primes of $E$ not dividing $p$ follows from (5.35) together with Brauer-Nesbitt. The assertion on the full local-global compatibility, in the case when $\pi_w$ is not of the form $\St_{E_w} \otimes \chi$, where $\St_{E_w}$ is the Steinberg representation of $\GL_2(E_w)$ and $\chi$ a character of $E_w^{\times}$, follows from the statement already proved. Indeed, from the classification of irreducible admissible representations of $\GL_2(E_w)$ we see that the monodromy operator of $\mathcal{L}_w(\pi_w \otimes |\det|_w^{-1/2})$ has to be trivial and that $\mathcal{L}_w(\pi_w \otimes |\det|_w^{-1/2})^{ss}$ is not a twist of the semi-simple part of the special Weil-Deligne representation of $W_{E_w}$. Since we know that $\iota_p \WD(\rho_p |_{G_{E_w}})^{ss} \cong \mathcal{L}_w(\pi_w \otimes |\det|_w^{-1/2})^{ss} $, it follows that the monodromy operator of $\iota_p \WD(\rho|_{G_{E_w}})^{F-ss}$ has to be trivial also (by the classification of two-dimensional Frobenius semi-simple Weil-Deligne representations).

\bigskip
\bigskip

We now prove theorem 1.2, which we state more precisely as theorem 5.17 and 5.18 below. Recall that the central character $\omega$ of $\pi$ has local component $\omega_w : a \rightarrow |a|_{\mathbf{C}}^{-\mu_0} = (a \overline{a})^{-\mu_0}$ for $a \in E_w^{\times}$ for all archimdean place $w$, and that the archimdean $L$-parameter of $\pi_w$ is given by $\phi_{\mu_0,n_w}$ where $n_w \geq 1$ and $\mu_0+1 \equiv n_w \mod{2}$.

Identify the embeddings of $E$ into $\mathbf{C}$ and the embeddings of $E$ into $\barQp$ via $\iota_p :\barQp \cong \mathbf{C}$. Under this identification a conjugate pair of embeddings of $E$ into $\barQp$ (under $\tau$) can be identified with an archimdean place of $E$.

\begin{theorem}
The Galois representation $\rho_p$ associated to $\pi$ is Hodge-Tate at primes of $E$ above $p$. The Hodge-Tate weights of $\rho_p$ at the conjugate pair of embeddings of $E$ into $\barQp$ corresponding to the archimdean place $w$ of $E$ is given by
\begin{eqnarray}
\{\delta_w, \delta_w + n_w \}
\end{eqnarray}
where $\delta_w := \frac{1}{2}( 1+\mu_0 - n_w)$.
\end{theorem}
\begin{proof}
Recall that as in proposition 5.4 we have
\[
R_p = \Ind^F_E \rho_p
\]
where $R_p$ is the $p$-adic Galois representation associated to the cuspidal automorphic representation $\Pi$ of $\GSp_4(\mathbf{A}_F)$, that belong to the global packet classified by the simple generic parameter $\Pi^{\prime} = \AI^F_E(\pi \otimes |\det|_{\mathbf{A}_E})$ together with the similitude character $\widetilde{\omega}^{\prime}$ as in the discussion in the beginning of section 5.1 ($\widetilde{\omega}^{\prime}$ is thus the central charcter of $\Pi$). Recall that we have chosen $\widetilde{\omega}^{\prime}$ so that for any archimedean prime $v$ of $F$ we have $\widetilde{\omega}^{\prime}_v$ is given by $ a\rightarrow a^{-\mu_0+2}$, and that the local component $\Pi_v$ at $v$ is classified by the archimdean $L$-parameter $\phi_{(\mu_0-2,n_v,0)}$ (proposition 5.2; recall that we have denoted $n_v=n_w$ for the place $v$ of $F$ below $w$). 

By theorem 4.14, the Hodge-Tate-Sen weights of $R_p$ at the embedding of $F$ into $\barQp$ that corresponds to the archimdean place $v$ of $F$ is given by
\begin{eqnarray}
\{ \delta_v,\delta_v,\delta_v + n_v,\delta_v+n_v \}
\end{eqnarray}
where 
\[
\delta_v = \frac{1}{2}(  (\mu_0- 2) + 3 -n_v)=\frac{1}{2}(1+\mu_0-n_w).
\]
Now since $R_p|_{G_E} = \rho_p \oplus \rho_p^{\tau}$, it follows that the Hodge-Tate-Sen weights of $\rho_p \oplus \rho_p^{\tau}$ at the conjugate pair of embedings of $E$ into $\barQp$ corresponding to the archimdean place $w$ of $E$ is also given by (5.37). 

In any case (5.37) is a muti-set of the form $\{a,a,b,b\}$ with $a,b$ being unequal integers. Now we note that if $\{x,y\}$ (respectively $\{x^{\prime},y^{\prime} \}$) is the Hodge-Tate-Sen weights of $\rho_p$ (respectively $\rho_p^{\tau}$) at the pair of embeddings corresponding to $w$, then $x+y$ (respectively $x^{\prime}+y^{\prime}$) is the Hodge-Tate-Sen weight of $\det \rho_p$ (respectively $\det \rho_p^{\tau}$) at the same pair of embeddings. But we have $\det \rho_p = \det \rho_p^{\tau}$ since $\omega = \omega^{\tau}$. Hence we have $x+y=x^{\prime}+y^{\prime}$. It follows that the Hodge-Tate-Sen weights of $\rho_p$ at this pair of embeddings must be $\{a,b\}$, i.e. given by (5.36). Finally since the Hodge-Tate-Sen weights are distinct integers, it follows that the Sen operator is semi-simple and so $\rho_p$ is Hodge-Tate.

\end{proof}

\begin{theorem}
 Suppose that $w$ is a place of $E$ above $p$, inert over $F$, and such that $\pi_w$ is spherical, with distinct Satake parameters $\alpha_w \neq \beta_w$. Then $\rho_p$ is crystalline at $w$.

Suppose that $w$ is a prime above $p$ that splits over $F$, with conjugate prime $w^{\tau}$, such that $\pi_w$ and $\pi_{w^{\tau}}$ are spherical, with Satake parameters $\alpha_w,\beta_w$ for $\pi_w$, and $\alpha_{w^{\tau}},\beta_{w^{\tau}}$ for $\pi_{w^{\tau}}$ respectively. Suppose that the elements $\{\alpha_w,\beta_w,\alpha_{w^{\tau}} ,\beta_{w^{\tau}}\}$ are all distinct. Then $\rho_p$ is crystalline at w (and also at $w^{\tau}$).
\end{theorem}
This follows from proposition 4.16, and is the same as the proof of theorem 5.3.1 in \cite{J}.

\begin{rem}
\end{rem}
Even though we do not know that the compatible system $\{ \rho_p \}_p$ is pure (in the sense of Galois representations), theorem 5.17 suggests that the motivic weight of the compatible system $\{\rho_p \}_p$ should be $1+\mu_0$. In the case considered in \cite{T2,BH}, we have $\mu_0 = k-2$, corresponding to modular forms of weight $k$ over $E$ (with $E$ being imaginary quadratic in \cite{T2,BH}).

\section{Appendix}
In this appendix we prove proposition 3.4. Recall the statement:

\bigskip

\noindent {\bf Proposition 3.4.}
{\it Suppose $\psi$ is a global formal parameter of Saito-Kurokawa type. For each finite prime $v$, let $\psi_v$ be the local A-parameter of $\GSp_4(F_v)$ given by the localization of $\psi$ at $v$. Denote as before by $\Pi_{\psi_v}$ the local A-packet classified by $\psi_v$. Then the semi-simple part of the $L$-parameters of the representations in the $A$-packet $\Pi_{\psi_v}$ have the same semi-simple part as that of $\phi_{\psi_v}$.} 

\bigskip

Thus let $\psi = \mu \boxplus (\lambda \boxtimes v(2))$ be a global formal parameter of Saito-Kurokawa type, where $\mu$ is a cuspidal automorphic representation of $\GL_2(\mathbf{A}_F)$, and $\lambda$ is an idele class character of $\mathbf{A}_F^{\times}$, satisfying $\omega_{\mu}=\lambda^2$ (with $\omega_{\mu}$ being the central character of $\mu$). In this case the local and global $A$-packets corresponding to $\psi$ are constructed by R. Schmidt \cite{Sch}, using the method of theta correspondence (strictly speaking in \cite{Sch} the results are stated only for the group $\PGSp_4$, but the arguments in {\it loc. cit.} applies verbatim to the case of $\GSp_4$). See also Proposition 5.5 and 5.6 of \cite{G}. To state the results we first set up some notation: denote by $\GSO(2,2)$ and $\GSO(4,0)$ the identity component groups of $\GO(2,2)$ and $\GO(4,0)$ respectively (notation as in \cite{GT1,GT2}). For any (finite) prime $v$, one has:

\begin{eqnarray*}
\GSO(2,2)(F_v) & \cong & (\GL_2(F_v) \times \GL_2(F_v))/\{ (z,z^{-1}): z \in F_v^{\times}      \} \\
\GSO(4,0)(F_v) & \cong & (D_v^{\times} \times D_v^{\times})/\{ (z,z^{-1}): z \in F_v^{\times}      \}
\end{eqnarray*}
where $D_v$ is the quaternion algebra over $F_v$. We denote by $\theta$, respectively $\theta^{+}$, the local theta lifting from $\GSO(2,2)(F_v)$, respectively from $\GSO(4,0)(F_v)$, to $\GSp_4(F_v)$. Then according to \cite{Sch}, the local $A$-packets of $\GSp_4(F_v)$ associated to the local $A$-parameter $\psi_v$ are given as follows: if $\mu_v$ is not a discrete series representation of $\GL_2(F_v)$, then the $A$-packet corresponding to $\psi_v$ is a singleton $\{\theta( \mu_v \boxtimes \lambda_v   )  \}$; here to ease notation we have written $\lambda_v$ instead of $\lambda_v \circ \det$ for the corresponding one-dimensional representation of $\GL_2(F_v)$. On the other hand, if $\mu_v$ is a discrete series representation of $\GL_2(F_v)$, denote by $\mu_v^{\JL}$ the irreducible admissible representation of $D_v^{\times}$ that corresponds to $\GL_2(F_v)$ under the Jacquet-Langlands correspondence. Then in this case the $A$-packet corresponding to $\psi_v$ is given by $\{ \theta(\mu_v \boxtimes \lambda_v), \theta^+( \mu_v^{\JL}  \boxtimes \lambda_v     ) \}$ (here again for $\theta^+(\mu_v^{\JL} \boxtimes \lambda_v)$ we have ease notation by writing $\lambda_v$ for $\lambda_v \circ \Nrd_{D_v/F_v}$, where $\Nrd_{D_v/F_v}$ is the reduced norm from $D_v$ to $F_v$). 

Using the results of \cite{GT1} we can compute the $L$-parameters of the representations in these $A$-packets to verify proposition 3.4. There is nothing to prove when $\mu_v$ is not a discrete series representation, since the corresponding $A$-packet is a singleton, and is equal to the singleton $L$-packet corresponding to $\phi_{\psi_v}$. In any case the $L$-parameter can be computed from \cite{GT1} as follows. In the notation of section 14 of \cite{GT1}, since $\mu_v$ is not a discrete series representation, we have $\mu_v = J(\pi(\chi_{v}^{\prime},\chi_{v}))$ where $\chi_v,\chi_v^{\prime}$ are characters of $F_v^{\times}$, while we have $\lambda_v \circ \det = J(\pi(\lambda_v |\cdot|_v^{1/2} , \lambda_v |\cdot|_v^{-1/2} ) )$. Hence by table 2, row {\it f} of {\it loc. cit.}, the representation $\theta(\mu_v \boxtimes \lambda_v)$ is given by $J_B(\lambda_v |\cdot|_v^{1/2} \chi_v^{-1},\lambda_v |\cdot|_v^{-1/2} \chi_v^{-1}; \chi_v )$. By table 1, row {\it e} for N.D.S. of {\it loc. cit.} (N.D.S. for non dsicrete series) the representation $J_B(\lambda_v |\cdot|_v^{1/2} \chi_v^{-1},\lambda_v |\cdot|_v^{-1/2} \chi_v^{-1}; \chi_v )$ has $L$-parameter given by:
\begin{eqnarray*}
& & (\chi_v \cdot (1 \oplus \lambda_v |\cdot|_v^{1/2} \chi_v^{-1} \oplus \lambda_v |\cdot|_v^{-1/2} \chi_v^{-1} \oplus \lambda_V^2 \chi_v^{-2})) \circ \art_v^{-1} \\
&=& (\chi_v \oplus \chi_v^{\prime} \oplus \lambda_v |\cdot|_v^{1/2} \oplus \lambda_v |\cdot|_v^{-1/2}) \circ \art_v^{-1}
\end{eqnarray*}
which is the $L$-parameter $\phi_{\psi_v}$ (to see the last equality use the fact that the condition $\omega_{\mu_v} =\lambda_v^2$ in this case is $\chi_v \chi_v^{\prime} = \lambda_v^2$).

Next suppose that $\mu_v$ is either supercuspidal or a twisted Steinberg representation $\St_{\GL_2} \otimes \chi_v$. Then by Table 2, row {\it e} of {\it loc. cit.} the representation $\theta(\mu_v \boxtimes \lambda_v)$ is equal to $J_{P}(\mu_v \otimes \lambda_v^{-1} |\cdot|_v^{1/2} ,    \lambda_v |\cdot|_v^{-1/2})$. Denoting by $\tau_v$ the representation $\mu_v \otimes \lambda_v^{-1} |\cdot|_v^{1/2}$, table 1 of {\it loc. cit.} shows that the $L$-parameter of the representation $J_P(\tau_v,\lambda_v |\cdot|_v^{-1/2})$ is given by:
\begin{eqnarray}
& & (\omega_{\tau_v} \lambda_v |\cdot|_v^{-1/2}) \circ \art_v^{-1}   \oplus \mathcal{L}_v( \tau_v \otimes \lambda_v  |\cdot|_v^{-1/2} )   \oplus (\lambda_v |\cdot|_v^{-1/2}) \circ \art_v^{-1} \\
&= & (\lambda_v |\cdot|_v^{1/2}) \circ \art_v^{-1} \oplus \mathcal{L}_v(\mu_v ) \oplus (\lambda_v |\cdot|_v^{-1/2}) \circ \art_v^{-1} \nonumber
\end{eqnarray} 
(here we are using $\omega_{\mu_v} = \lambda_v^2$). Next we consider the representation $\theta^+(\mu_v^{\JL} \boxtimes \lambda_v)$. We first make the assumption that $\mu_v$ is not of the form $\St_{\GL_2} \otimes \lambda_v$. Then by table 3, row {\it b} of {\it loc. cit.} the representation $\theta^+(\mu_v^{\JL} \boxtimes \lambda_v)$ is a non-generic supercuspidal representation, and by table 1, row {\it c} for S.C. of {\it loc. cit.} the $L$-parameter of $\theta^+(\mu_v^{\JL} \boxtimes \lambda_v)$ is given by
\begin{eqnarray}
\mathcal{L}_v(\mu_v) \oplus \mathcal{L}_v( \St_{\GL_2} \otimes \lambda_v).
\end{eqnarray}
It is clear that (6.1) and (6.2) have the same semi-simple part. Furthermore, if $\mu_v$ is supercuspidal, one sees from table 1, row {\it b} of D.S. the $L$-packet corresponding to the $L$-parameter (6.2) contains the other representation $\St(\mu_v \otimes \lambda_v^{-1},\lambda_v)$. If $\mu_v = \St_{\GL_2} \otimes \chi_v$ with $\chi_v \neq \lambda_v$, then from table 1, row {\it a} for N.D.S. of {\it loc. cit.} the $L$-packet corresponding to the $L$-parameter (6.2) contains the other representation $J_Q(\chi_v \lambda_v^{-1},\St_{\GL_2}) \otimes \lambda_v$.

It remains to treat the case where $\mu_v = \St_{\GL_2} \otimes \lambda_v$. Then by table 3, row {\it a} of {\it loc. cit.} the representation $\theta^+(\mu_v^{\JL} \boxtimes \lambda_v) = \theta^+(\lambda_v \boxtimes \lambda_v)$ is (in the notation of {\it loc. cit.}) given by $\pi_{ng}(\mu_v) = \pi_{ng}( \St_{\GL_2} \otimes \lambda_v )$ (non-generic representation), hence by table 1, row {\it c} for N.D.S. of {\it loc. cit.} has $L$-parameter given by 
\begin{eqnarray}
\mathcal{L}_v(\St_{\GL_2} \otimes \lambda_v) \oplus \mathcal{L}_v(\St_{\GL_2} \otimes \lambda_v)
\end{eqnarray}
i.e. the same as (6.2) with $\mu_v = \St_{\GL_2} \otimes \lambda_v$, and hence has the same semi-simple part as (6.1). One also see from table 1, row {\it b} for N.D.S. of {\it loc. cit.} that the $L$-packet corresponding to the $L$-parameter (6.3) contains the other representation $\pi_{gen}(\St_{\GL_2} \otimes \lambda_v)$ which is generic.


\begin{thebibliography}{99}


\bibitem[A1]{A1} J.Arthur, \textit{Automorphic representations of $\GSp (4)$.} Contributions to automorphic forms, geometry, and number theory, 65--81, Johns Hopkins Univ. Press, Baltimore, MD, 2004.

\bibitem[A2]{A2} J.Arthur, \textit{The endoscopic classification of representations: orthogonal and symplectic groups.} Colloquium Publication Series, AMS. To appear. Available at \url{http://www.claymath.org/cw/arthur/}

\bibitem[AC]{AC} J.Arthur, L.Clozel \textit{Simple algebras, base change, and the advanced theory of the trace formula.} Annals of Mathematics Studies, 120. Princeton University Press, Princeton, NJ, 1989.

\bibitem[AS]{AS} M.Asgari, F.Shahidi, \textit{Generic transfer from $\GSp (4)$ to $\GL (4)$.} 
Compos. Math. 142 (2006), no. 3, 541–-550.

\bibitem[BC1]{BC1} J.Bella\"iche, G.Chenevier, \textit{Families of Galois representations and Selmer groups.} Astérisque No. 324 (2009).

\bibitem[BC2]{BC2} J.Bella\"iche, G.Chenevier, \textit{The sign of Galois representations attached to automorphic forms for unitary groups.} Compos. Math. 147 (2011), no. 5, 1337–-1352.

\bibitem[BH]{BH} T.Berger, G.Harcos, \textit{$l$-adic representations associated to modular forms over imaginary quadratic fields.}
Int. Math. Res. Not. IMRN 2007, no. 23.

\bibitem[BLGGT]{BLGGT} T.Barnet-Lamb, T.Gee, D.Geraghty, R.Taylor, \textit{Local-global compatibility for $l=p$. I} Ann. de Math. de Toulouse 21 (2012), 57--92; \textit{II.} To appear in Ann. Sci. de l'ENS.

\bibitem[CaG]{CaG} F.Calegari, T.Gee, \textit{Irreducibility of automorphic Galois representations of $\GL(n)$, $n$ at most $5$.} To appear in Annales de l'Institut Fourier.

\bibitem[Car1]{Car1} A.Caraiani, \textit{Local-global compatibility and the action of monodromy on nearby cycles.} Duke Math. J. 161 (2012), no. 12, 2311–-2413.

\bibitem[Car2]{Car2} A.Caraiani, \textit{Monodromy and local-global compatibility for $l=p$}. Available at arXiv:1202.4683

\bibitem[C]{C} P.S.Chan, \textit{Invariant representations of $\GSp(2)$ under tensor product with a quadratic character.} Mem. Amer. Math. Soc. 204 (2010), no. 957.

\bibitem[Ch]{Ch} G.Chenevier, \textit{Une application des vari\'et\'es de Hecke des groupes unitaires.} Preprint. Available at 
\url{http://www.math.polytechnique.fr/~chenevier/pub.html}

\bibitem[ChH]{ChH} G.Chenevier, M.Harris, \textit{Construction of Automorphic Galois Representations II.} Cambridge Math. Journal 1, 53--73 (2013).


\bibitem[CG]{CG} P.S.Chan, W.T.Gan, \textit{The local Langlands conjecture for $\GSp (4)$ III: stability and twisted endoscopy.} To appear in Journal of Number Theory, Rallis Memorial Volume. Available at \url{http://www.math.nus.edu.sg/~matgwt/}

\bibitem[F]{F} J.-M.Fontaine, \textit{Arithm\'etique des repr\'esentations galoisiennes $p$-adiques.} Ast\'erisque No. 295 (2004), 1–-115. 

\bibitem[G]{G} W.T.Gan, \textit{The Saito-Kurokawa space of $\PGSp_4$ and its transfer to inner forms.} Eisenstein series and applications, 87–123, 
Progr. Math., 258, Birkhäuser Boston, Boston, MA, 2008. 

\bibitem[GRS]{GRS} D.Ginzburg, S.Rallis, D.Soudry, \textit{Periods, poles of $L$-functions and symplectic-orthogonal theta lifts.} J. Reine Angew. Math. 487 (1997), 85-–114.

\bibitem[GT1]{GT1} W.T.Gan, S.Takeda, \textit{Theta Correspondence for $\GSp(4)$.} Represent. Theory 15 (2011), 670--718. 

\bibitem[GT2]{GT2} W.T.Gan, S.Takeda, \textit{The local Langlands conjecture for $\GSp (4)$.} Annals of Math. Vol. 173 (2011), no. 3, 1841--1882.

\bibitem[HLTT]{HLTT} M.Harris, K.W.Lan, R.Taylor, J.Thorne, \textit{On the rigid cohomology of certain
Shimura varieties}. Preprint (2013). 

\bibitem[HST]{HST} M.Harris, D.Soudry, R.Taylor, \textit{$l$-adic representations associated to modular forms over imaginary quadratic fields. I. Lifting to $\GSp_4(\mathbf{Q})$.} 
Invent. Math. 112 (1993), no. 2, 377–-411. 

\bibitem[H]{H} H.Hida, \textit{Anticyclotomic Main Conjectures.} Doc. Math. 2006, Extra Vol., 465-–532.

\bibitem[JS]{JS} H.Jacquet, J.Shalika, \textit{On Euler products and the classification of automorphic forms, II.} Amer. J. Math. 103(4): 777--815 (1981).

\bibitem[J]{J} A.Jorza, \textit{Crystalline Representations for $\GL (2)$ over Quadratic Imaginary Fields.} Princeton thesis (2010). 

\bibitem[KL]{KL} M.Kisin, K.F.Lai, \textit{Overconvergent Hilbert modular forms.} Amer. J. Math. 127 (2005), no. 4, 735–-783. 


\bibitem[L]{L} G.Laumon, \textit{Fonctions zetas des vari\'et\'es de Siegel de dimension trois.} Ast\'erisque 302 (2005): 1–-66.



\bibitem[MT]{MT} C.P.Mok, F.C.Tan, \textit{Overconvergent family of Siegel-Hilbert modular forms.} Preprint (2012). Available at arXiv:1208.1093

\bibitem[N]{N} K.Nakamura, \textit{Zariski density of crystalline representations for any $p$-adic field.} Preprint (2011). Available at arXiv:1104.1760v1

\bibitem[R]{R} B.Roberts, \textit{Global $L$-packets for $\GSp(2)$ and theta lifts.} Documenta Mathematica, 6, (2001), 247--314.

\bibitem[Sa]{Sa} T.Saito, \textit{Modular forms and $p$-adic Hodge theory.} Invent. Math. 129 (1997), no. 3, 607–-620. 

\bibitem[Sch]{Sch} R.Schmidt, \textit{The Saito-Kurokawa lifting and functoriality.} Amer. J. Math. 127 (2005), no. 1, 209–-240. 

\bibitem[Scho]{Scho} P.Scholze, \textit{On torsion in the cohomology of locally symmetric varieties.} Preprint (2013). Available at arXiv:1306.2070

\bibitem[Sen]{Sen} S.Sen, \textit{An infinite dimensional Hodge-Tate theory.} Bull. Soc. Math. France 121 (1993), 13--34.

\bibitem[S1]{S1} F.Shahidi, \textit{On certain $L$-functions.} Amer. J. Math., 103 (1981), 297–-355.

\bibitem[S2]{S2} F.Shahidi, \textit{On non-vanishing of twisted symmetric and exterior square $L$-functions for $\GL (n)$.} 
Olga Taussky-Todd: in memoriam. 
Pacific J. Math. 1997, Special Issue, 311-–322. 



\bibitem[Sh]{Sh} S.W.Shin, \textit{Galois representations arising from some compact Shimura varieties.} Annals of Math. Vol. 173 (2011), no. 3, 1645--1741.

\bibitem[So]{So} C.Sorensen, \textit{Galois representations attached to Hilbert-Siegel modular forms.}
Doc. Math. 15 (2010), 623–-670. 

\bibitem[Sou1]{Sou1} D.Soudry, \textit{Automorphic forms on $\GSp (4)$.} Festschrift in honor of I. I. Piatetski-Shapiro, Part II (Ramat Aviv, 1989), 291–-303, Israel Math. Conf. Proc., 3, Weizmann, Jerusalem, 1990. 

\bibitem[Sou2]{Sou2} D.Soudry, \textit{On Langlands functoriality from classical groups to $\GL_n$.} Automorphic forms. I. 
Ast\'erisque No. 298 (2005), 335-–390. 

\bibitem[ST]{ST} P.Sally, M.Tadic, \textit{Induced representations and classifications for $\GSp(2,F)$ and $\Sp(2,F)$.} Memoirs Soc. Math. France (N.S.) No. 52 (1993), 75--133.

\bibitem[Ta]{Ta} S.Takeda, \textit{Some local-global non-vanishing results of theta lifts for symplectic-orthogonal dual pairs.} J. Reine Angew. Math. 657 (2011), 81–-111. 

\bibitem[Tan]{Tan} F.C.Tan, \textit{Families of $p$-adic Galois Representations.} MIT thesis (2011).

\bibitem[T1]{T1} R.Taylor, \textit{Galois representations associated to Siegel modular forms of low weight.} 
Duke Math. J. 63 (1991), no. 2, 281–-332. 

\bibitem[T2]{T2} R.Taylor, \textit{$l$-adic representations associated to modular forms over imaginary quadratic fields. II.}
Invent. Math. 116 (1994), no. 1-3, 619–-643.

\bibitem[Wal]{Wal} N. Wallach, \textit{On the constant term of a square integrable automorphic form.} Operator algebras and group representations, Vol. II (Neptun, 1980), 227–-237, 
Monogr. Stud. Math., 18, Pitman, Boston, MA, 1984. 

\bibitem[W]{W} R.Weissauer, \textit{Four dimensional Galois representations.} Ast\'erisque 302 (2005): 67–-150.









\end{thebibliography}
\end{document}